\documentclass[reqno,11pt]{article}
\usepackage{mathrsfs}
\usepackage{amssymb}
\usepackage[usenames,dvipsnames]{xcolor}
\usepackage{amsthm}
\usepackage{amsmath}
\usepackage{amsfonts}
\usepackage{enumerate}
\usepackage{txfonts}
\usepackage{bm}

\usepackage{graphicx}

\numberwithin{equation}{section}

\newtheorem{theorem}{Theorem}[section]
\newtheorem{lemma}[theorem]{Lemma}

\newtheorem{proposition}[theorem]{Proposition}

\theoremstyle{definition}

\theoremstyle{remark}
\newtheorem{remark}[theorem]{Remark}

\begin{document}
\title{Existence and stability of cylindrical transonic shock solutions under three dimensional perturbations}
\author{Shangkun WENG\thanks{School of Mathematics and Statistics, Wuhan University, Wuhan, Hubei Province, 430072, People's Republic of China. Email: skweng@whu.edu.cn}\and Zhouping XIN\thanks{The Institute of Mathematical Sciences and Department of Mathematics, The Chinese University of Hong Kong, Shatin, NT, Hong Kong. Email: zpxin@ims.cuhk.edu.hk}}
\date{}

\pagestyle{myheadings} \markboth{Cylindrical transonic shock}{Cylindrical transonic shock}\maketitle

\def\be{\begin{eqnarray}}
\def\ee{\end{eqnarray}}
\def\ba{\begin{aligned}}
\def\ea{\end{aligned}}
\def\bay{\begin{array}}
\def\eay{\end{array}}
\def\bca{\begin{cases}}
\def\eca{\end{cases}}
\def\p{\partial}
\def\no{\nonumber}
\def\e{\epsilon}
\def\de{\delta}
\def\De{\Delta}
\def\om{\omega}
\def\Om{\Omega}
\def\f{\frac}
\def\th{\theta}
\def\la{\lambda}
\def\lab{\label}
\def\b{\bigg}
\def\var{\varphi}
\def\na{\nabla}
\def\ka{\kappa}
\def\al{\alpha}
\def\La{\Lambda}
\def\ga{\gamma}
\def\Ga{\Gamma}
\def\ti{\tilde}
\def\wti{\widetilde}
\def\wh{\widehat}
\def\ol{\overline}
\def\ul{\underline}
\def\Th{\Theta}
\def\si{\sigma}
\def\Si{\Sigma}
\def\oo{\infty}
\def\q{\quad}
\def\z{\zeta}
\def\co{\coloneqq}
\def\eqq{\eqqcolon}
\def\di{\displaystyle}
\def\mb{\mathbb}
\def\ms{\mathscr}
\def\lan{\langle}
\def\ran{\rangle}
\def\lb{\llbracket}
\def\rb{\rrbracket}

\begin{abstract}
  We establish the existence and stability of cylindrical transonic shock solutions under three dimensional perturbations of the incoming flows and the exit pressure without any further restrictions on the background transonic shock solutions. The strength and position of the perturbed transonic shock are completely free and uniquely determined by the incoming flows and the exit pressure. The optimal regularity is obtained for all physical quantities, and the velocity, the Bernoulli's quantity, the entropy and the pressure share the same regularity. The approach is based on the deformation-curl decomposition to the steady Euler system introduced by the authors to decouple the hyperbolic and elliptic modes effectively. However, one of the key elements in application of the deformation-curl decomposition is to find a decomposition of the Rankine-Hugoniot conditions, which shows the mechanism of determining the shock front uniquely by an algebraic equation and also gives an unusual second order differential boundary conditions on the shock front for the first order deformation-curl system. After homogenizing the curl system and introducing a potential function, this unusual condition on the shock front becomes the Poisson equation with homogeneous Neumann boundary condition on the intersection of the shock front and the cylinder walls from which an oblique boundary condition for the potential function can be uniquely derived.
\end{abstract}

\begin{center}
\begin{minipage}{5.5in}
\end{minipage}
\end{center}

\section{Introduction and main results}\label{sec1}

In this paper, we study the transonic shock problem in a de Laval nozzle described by Courant and Friedrichs (\cite[Page 386]{cf48}): given appropriately large receiver pressure $P_e$, if the upstream flow is still supersonic behind the throat of the nozzle, then at a certain place in the diverging part of the nozzle a shock front intervenes and the gas is compressed and slowed down to subsonic speed. The position and the strength of the shock front are automatically adjusted so that the end pressure at the exit becomes $P_e$. This problem can be described by the three-dimensional steady compressible Euler system:
\be\label{com-euler}\begin{cases}
\text{div }(\rho {\bf u})=0,\\
\text{div }(\rho {\bf u}\otimes {\bf u}+ P I_3) =0,\\
\text{div }(\rho (\f12|{\bf u}|^2 +e) {\bf u}+ P{\bf u}) =0,
\end{cases}\ee
where ${\bf u}=(u_1, u_2,u_3)$, $\rho$, $P$, and $e$ stand for the velocity, density, pressure, and internal energy, respectively. For polytropic gases, the equation of state and the internal energy are of the form
\begin{equation}\no
P=K(S)\rho^{\ga}=A  e^{\frac{S}{c_v}}\rho^{\gamma}\quad \text{and}\quad
e=\f{P}{(\ga-1)\rho},\ \ \ \gamma>1,
\end{equation}
respectively, where $\gamma> 1$, $A$, and $c_v$ are positive constants, and $S$ is called the specific entropy. The system \eqref{com-euler} is  hyperbolic for supersonic flows ($M_a>1$),  hyperbolic-elliptic coupled for subsonic flows ($M_a<1$), and degenerate at sonic points (i.e. $M_a=1$), where $M_a=\frac{|{\bf u}|}{c(\rho,K)}$ is the Mach number of the flow with $c(\rho,K)=\sqrt{\partial_\rho P(\rho, K)}$ being the local sound speed. $B=\frac{|{\bf u}|^2}{2}+e+\frac{P}{\rho}$ is called the Bernoulli's quantity.

The studies on transonic shocks using the quasi-one-dimensional model can be found in \cite{cf48,egm84,liu82}. There are several typical transonic shock solutions with symmetry to the steady Euler system that are well-known and had been investigated extensively. One is the transonic shock solution in a duct with both upstream supersonic state and downstream subsonic state being constant and its shock position can be arbitrary.  The structural stability of these transonic shocks for multidimensional steady potential flows in nozzles was studied in \cite{cf03,cf04,xy05, xy08a}. The authors in \cite{xy05,xy08a} showed that the stability of transonic shocks for potential flows is usually ill-posed under the perturbations of the exit pressure. Many researchers also used the steady Euler system to study the transonic shock problem in the flat or almost flat nozzles with the exit pressure satisfying some special constraints, see \cite{chen05,chen08,cy08,ly08,xyy09} and the references therein. The authors in \cite{chen08,cy08} had used the characteristic decomposition of the steady Euler system to prove the structural stability of the transonic shock in a rectangle cylinder or a flat nozzle with general section under the requirement that the shock front must pass through a fixed point and the exit pressure can only be prescribed up to a constant. Recently, the authors in \cite{fx21} have established the stability and existence of transonic shock solutions to the two dimensional  steady compressible Euler system in an almost flat finite nozzle with the exit pressure, where the shock position was uniquely determined by solving an elaborate linear free boundary problem. See also the three dimensional axisymmetric generalization in \cite{fg21}.

The other are the radially symmetric transonic shock in a divergent sector and the spherically symmetric transonic shock in a conic cone in which the shock position is uniquely determined by the exit pressure. The authors in \cite{lxy09a} had proved the well-posedness of the transonic shock problem in two dimensional divergent nozzles under the perturbations for the exit pressure when the opening angle of the nozzle is suitably small. This restriction was removed in \cite{lxy09b} and the transonic shock in a two dimensional straight divergent nozzle is shown in \cite{lxy13} to be structurally stable under generic perturbations for both the nozzle walls and the exit pressure. One of the key ideas in \cite{lxy13} is to introduce a Lagrangian transformation to straighten the streamlines and reduce the transonic shock problem to a second order elliptic equation with a nonlocal term and an unknown parameter together with an ODE for the shock front. In \cite{lxy10a,lxy10b}, the existence and stability of transonic shocks for three dimensional axisymmetric flows without swirl in a conic nozzle were proved to be structurally stable under suitable perturbations of the exit pressure. For the structural stability under the axisymmetric perturbation of
the nozzle wall, a modified Lagrangian coordinate was introduced in \cite{wxx} to
deal with the corner singularity near the intersection points of the shock surface
and nozzle boundary and the artificial singularity near the axis simultaneously. The authors in \cite{lxy11} proved the uniqueness of transonic shock solutions in a conic nozzle without requiring the nonphysical assumption on the shock front past a fixed point and established the monotonicity of the shock position relative to the end pressure. There have been many other interesting results on transonic shock problems in a nozzle for different models with various exit boundary conditions, such as the non-isentropic potential model, the exit boundary condition for the normal velocity, the transonic shock flows in a spherical shell, etc, see \cite{bf11, ccf07, lxy16} and references therein. In particular, the authors in \cite{lxy16} proved the conditional structural stability of the spherical transonic shock in a spherical shell under the perturbations of supersonic incoming flows and the exit pressure, which required that the background transonic shock solutions satisfy some ``Structure Condition".

Most recently, the authors in \cite{wxy21a,wxy21b} studied radially
symmetric transonic spiral flows with/without shock in an annulus. It is interesting to notice that the angular velocity may induce new wave patterns. Indeed, it was found in \cite{wxy21a} that besides the well-known
supersonic-subsonic shock, there exists a supersonic-supersonic shock solution, where the downstream
state may change smoothly from supersonic to subsonic. Furthermore, there exists a
supersonic-sonic shock solution where the shock circle and the sonic circle coincide.

In the cylindrical coordinate
\be\no
x_1=r\cos\theta,\ x_2=r\sin\theta,\ x_3=x_3,
\ee
the velocity field can be represented as ${\bf u}(x)=U_1{\bf e}_r+U_2{\bf e}_\th+U_3{\bf e}_3$, where
\be\no
{\bf e}_r=(\cos\theta, \sin\theta, 0)^t, \ \ {\bf e}_\theta=(-\sin\theta,\cos\theta,0)^t,\ \ {\bf e}_3=(0,0,1)^t.
\ee
Then the steady Compressible Euler equations in cylindrical coordinates take the form
\begin{eqnarray}\label{euler-cyl}
\begin{cases}
\p_r(\rho U_1)+\frac1r \rho U_1+\frac{1}{r}\p_{\theta} (\rho U_2) + \p_{x_3} (\rho U_3)=0,\\
(U_1\p_r +\frac{U_2}{r}\p_{\theta}+U_3\p_{x_3}) U_1+\frac1\rho \p_r P-\frac{U_2^2}r=0,\\
(U_1\p_r +\frac{U_2}{r}\p_{\theta}+U_3\p_{x_3}) U_2+\frac1{r\rho} \p_{\theta} P+\frac{U_1U_2}r=0,\\
(U_1\p_r +\frac{U_2}{r}\p_{\theta}+U_3\p_{x_3}) U_3+\frac1\rho \p_{x_3}P=0,\\
(U_1\p_r +\frac{U_2}{r}\p_{\theta}+U_3\p_{x_3}) K=0.
\end{cases}
\end{eqnarray}
The flow region is assumed to be a part of a concentric cylinder described as
\be\no
\Omega=\{(r,\theta,x_3):r_1<r<r_2,(\theta,x_3)\in E\}, \ \ E:=(-\theta_0,\theta_0)\times (-1,1),
\ee
where $0<r_1<r_2<\infty,\theta_0\in (0,\frac{\pi}{2})$ are fixed positive constants.

Suppose the incoming supersonic flow is prescribed at the inlet $r=r_1$, i.e.,
\begin{equation}\no
({\bf u}^-, {\bar P}^-, \bar{K}^-)(r_1,\theta,x_3)=(\bar{U}^-(r_1){\bf e}_r,\bar{P}^-(r_1),\bar{K}^-), \ \forall (\theta,x_3)\in E,
 \end{equation}
where $\bar{U}^-(r_1)>c(\bar{\rho}^-(r_1),\bar{K}^-)>0$ with $\bar{K}^-$ being a constant. Then there exist two positive constants $P_1$ and $P_2$ depending only on the incoming supersonic flows and the nozzle, such that if the pressure $P_e\in (P_1, P_2)$ is given at the exit $r=r_2$,  then there exists a unique piecewise smooth cylindrically symmetric transonic shock solution
\be\no
\overline{\bm{\Psi}}(r,\theta,x_3)= (\overline{{\bf u}},\bar{\rho},\bar{K})({\bf x})=\left\{
\begin{aligned}
&{\overline{\bm{\Psi}}}^-(r,\theta,x_3):=(\bar{U}^{-}(r),0, 0, \bar{\rho}^{-}(r),\bar{K}^{-}),\,\,\text{in }\Omega_b^-\\
&{\overline{\bm{\Psi}}}^+(r,\theta,x_3):=(\bar{U}^{+}(r), 0, 0, \bar{\rho}^{+}(r),\bar{K}^{+}),\,\,\text{in }\Omega_b^+
\end{aligned}
\right.
\ee
to \eqref{euler-cyl} with a shock front located at $r=r_s\in (r_1,r_2)$, where
\begin{eqnarray}\nonumber
&&\Om_{b}^-=\{(r,\theta,x_3): r\in (r_1,r_s),(\theta,x_3)\in E\},\\\nonumber
&&\Om_{b}^+=\{(r,\theta,x_3): r\in (r_s,r_2),(\theta,x_3)\in E\}.
\end{eqnarray}
Across the shock, the Rankine-Hugoniot conditions and the physical entropy condition are satisfied:
\begin{equation}\no
[\bar{\rho} \bar{U}]|_{r=r_s}=0,\quad [\bar{\rho} \bar{U}^2+\bar{P}]|_{r=r_s}=0,\quad [\bar{B}]|_{r=r_s}=0, \quad \bar{K}^+>\bar{K}^-,
\end{equation}
where $[g]|_{r=r_s}:=g(r_s+)-g(r_s-)$ denotes the jump of $g$ at $r=r_s$. Later on, this special solution, ${\bf \overline{\Psi}}$, will be called the background solution. Clearly, one can extend the supersonic and subsonic parts of ${\bf \overline{\Psi}}$ in a natural way, respectively. With an abuse of notations, we still call the extended subsonic and supersonic solutions ${\bf\overline{\Psi}}^+$ and ${\bf\overline{\Psi}}^-$, respectively.  For detailed properties of this cylindrically symmetric transonic shock solution, we refer to \cite[Section 147]{cf48} or \cite[Theorem 1.1]{xy08b}. The main goal of this paper is to establish the structural stability of this cylindrically symmetric transonic shock solution under generic three dimensional perturbations of the incoming supersonic flows and the exit pressure.

Let the incoming supersonic flow at the inlet $r=r_1$ be prescribed as
by \begin{equation}\lab{super1}
{\bf \Psi}^-(r_1, \theta,x_3)={\bf \overline{\Psi}}^-(r_1)+ \epsilon (U_{1,0}^-,U_{2,0}^-,U_{3,0}^-, P_0^-,K_0^-)(\theta,x_3),
\end{equation}
where
$(U_{1,0}^-,U_{2,0}^-,U_{3,0}^-, P_0^-,K_0^-)\in (C^{2,\alpha}(\overline{E}))^5$.
The flow satisfies the slip condition ${\bf u}\cdot {\bf n}$=0 on the nozzle wall, where ${\bf n}$ is the outer normal of the nozzle wall, which in the cylindrical coordinates, can be written as
\be\lab{slip1}\begin{cases}
U_2(r,\pm \theta_0,x_3) =0\quad \ &\text{on } \Gamma_{2,\pm}:=\{(r,\pm\theta_0, x_3): r_1\leq r\leq r_2, -1\leq x_3\leq 1\},\\
U_3(r,\theta,\pm 1) =0\quad \ &\text{on } \Gamma_{3,\pm}:=\{(r,\theta,\pm 1): r_1\leq r\leq r_2, -\theta_0\leq \theta\leq \theta_0\}.
\end{cases}\ee
At the exit of the nozzle, the end pressure is prescribed by
\be\lab{pressure}
P(r_2,\theta,x_3)= P_e + \epsilon P_{ex}(\theta,x_3)\ \text{at}\ \ \Gamma_o:=\{(r_2,\theta,x_3): (\theta,x_3)\in E\},
\ee
here $P_{ex}\in C^{2,\alpha}(\overline{E})$ satisfies the compatibility conditions
\be\label{pressure-cp}\begin{cases}
\p_{\theta} P_{ex}(\pm \theta_0,x_3)=0,\ \  \ &\forall x_3\in [-1,1],\\
\p_{x_3} P_{ex}(\theta,\pm 1)=0,\ \  \ &\forall \theta\in [-\theta_0,\theta_0].
\end{cases}\ee

The goal is to find a piecewise smooth solution $\bm\Psi$ to \eqref{euler-cyl} supplemented with the boundary conditions \eqref{super1}, \eqref{slip1}, and \eqref{pressure}, which jumps only at a shock front $\mathcal{S}=\{(r, \theta,x_3): r=\xi(\theta,x_3), (\theta,x_3) \in E\}$. More precisely, $\bm\Psi$ has the form
\begin{equation}\nonumber
\bm\Psi=\begin{cases}
(U_1^-, U_2^-, U_3^-, P^-, K^-),\quad\text{in}\,\, \Omega^-=\{(r,\theta,x_3):r_1< r<\xi(\theta,x_3),(\theta,x_3)\in E\},\\
(U_1^+, U_2^+, U_3^+, P^+, K^+),\quad\text{in}\,\, \Omega^+=\{(r,\theta,x_3):\xi(\theta,x_3)<r<r_2,(\theta,x_3)\in E\},
\end{cases}
\end{equation}
and satisfies the Rankine-Hugoniot conditions on the shock front $r=\xi(\theta, x_3)$:
\begin{eqnarray}\label{rh}
\begin{cases}
[\rho U_1]-\frac{1}{\xi}\p_{\theta}\xi [\rho U_2]-\p_{x_3}\xi[ \rho U_3]=0,\\
[\rho U_1^2+P]-\frac{1}{\xi}\p_{\theta}\xi[ \rho U_1U_2]-\p_{x_3}\xi[ \rho U_1U_3]=0,\\
[\rho U_1U_2]-\frac{1}{\xi}\p_{\theta}\xi[ \rho U_2^2+P]-\p_{x_3}\xi[ \rho U_2U_3]=0,\\
[\rho U_1U_3]-\frac{1}{\xi}\p_{\theta}\xi[ \rho U_2U_3]-\p_{x_3}\xi[ \rho U_3^2+P]=0,\\
[B]=0.\\
\end{cases}
\end{eqnarray}

The existence and uniqueness of the supersonic flow to \eqref{euler-cyl} follows from the theory of classical solutions to the boundary value problem for quasi-linear symmetric hyperbolic equations (See \cite{bs07}).
\begin{lemma}\label{supersonic}
{\it For the incoming data given in \eqref{super1} satisfying the compatibility conditions
\be\lab{super3}\begin{cases}
(U_{2,0}^-,\p_{\theta}^2U_{2,0}^-)(\pm\theta_0,x_3)=\p_{\theta}(U_{1,0}^-,U_{3,0}^-, P_0^-,K_0^-)(\pm\theta_0, x_3)=0,\ \forall x_3\in [-1,1],\\
(U_{3,0}^-,\p_{x_3}^2 U_{3,0}^-)(\theta,\pm 1)=\p_{x_3}(U_{1,0}^-,U_{2,0}^-, P_0^-,K_0^-)(\theta, \pm 1)=0,\  \forall x_2\in [-\theta_0,\theta_0],
\end{cases}\ee
then there exists $\epsilon_0>0$ depending only on the background solution and the boundary data, such that for any $0<\epsilon<\epsilon_0$, there exists a unique $C^{2,\alpha}(\overline{\Omega})$ solution $\bm\Psi^-=(U_1^-,U_2^-,U_3^-,P^-,K^-)(r,\theta,x_3)$ to \eqref{euler-cyl} with \eqref{super1} and \eqref{slip1}, which satisfies
\be\no
\|(U_1^-,U_2^-,U_3^-,P^-,K^-)-(\bar{U}^-,0,0,\bar{P}^-, \bar{K}^-)\|_{C^{2,\alpha}(\overline{\Omega})}\leq C_0\epsilon,
\ee
and
\be\lab{super5}
\begin{cases}
(U_{2}^-,\p_{\theta}^2U_{2}^-,p_{\theta}(U_{1}^-,U_{3}^-, P^-,K^-))(r,\pm\theta_0, x_3)=0,\ \forall (r,x_3)\in [r_1,r_2]\times [-1,1],\\
(U_{3}^-,\p_{x_3}^2U_{3}^-,\p_{x_3}(U_{1}^-,U_{2}^-, P^-,K^-))(r,\theta, \pm 1)=0, \forall (r,\theta)\in [r_1,r_2]\times [-\theta_0,\theta_0].
\end{cases}
\ee
}\end{lemma}

The compatibility conditions in \eqref{super5} in Lemma \ref{supersonic} will be verified in the Appendix. Therefore, our problem is reduced to solve a free boundary value problem for the steady Euler system in which the shock front and the downstream subsonic flows are unknown. Then the main result in this paper is stated as follows.
\begin{theorem}\label{existence}
{\it Assume that the compatibility conditions \eqref{pressure-cp} and  \eqref{super3} hold. There exists a suitable constant $\epsilon_0>0$ depending only on the background solution $\overline{\bm{\Psi}}$ and the boundary data $U_{1,0}^-,U_{2,0}^-,U_{3,0}^-, P_0^-$, $K_0^-$, $P_{ex}$ such that if $0< \epsilon<\epsilon_0$, the problem \eqref{euler-cyl} with \eqref{super1}-\eqref{pressure}, and \eqref{rh} has a unique solution $\bm{\Psi}^+=(U_1^+,U_2^+,U_3^+,P^+,K^+)(r,\theta,x_3)$ with the shock front $\mathcal{S}: r=\xi(\theta,x_3)$ satisfying the following properties.
\begin{enumerate}[(1)]
  \item The function $\xi(\theta,x_3)\in C^{3,\alpha}(\overline{E})$ satisfies
  \be\no
  \|\xi(\theta, x_3)-r_s\|_{C^{3,\alpha}(\overline{E})}\leq C_*\epsilon,
  \ee
  and
  \be\no\begin{cases}
  \p_{\theta}\xi(\pm\theta_0, x_3)=\p_{\theta}^3\xi(\pm\theta_0, x_3)=0,\ \ \forall x_3\in [-1,1],\\
  \p_{x_3}\xi(\theta,\pm 1)=\p_{x_3}^3\xi(\theta, \pm 1)=0,\ \ \forall \theta\in [-\theta_0,\theta_0],
  \end{cases}\ee
  where $C_*$ is a positive constant depending only on the background solution and the supersonic incoming flow and the exit pressure.
  \item The solution $\bm{\Psi}^+=(U_1^+,U_2^+,U_3^+,P^+,K^+)(r,\theta,x_3)\in C^{2,\alpha}(\overline{\Omega^+})$ satisfies the entropy condition
  \begin{equation}\no
  K^+(\xi(\theta,x_3)+,\theta,x_3) >K^-(\xi(\theta,x_3)-, \theta,x_3)\quad \forall (\theta,x_3)\in E
  \end{equation}
   and the estimate
  \be\no
  \|\bm{\Psi}^+ -\overline{{\bm{\Psi}}}^+\|_{C^{2,\alpha}(\overline{\Omega^+})}\leq C_*\epsilon
  \ee
  with the compatibility conditions
  \be\no
  \begin{cases}
  (U_{2}^+,\p_{\theta}^2U_{2}^+,\p_{\theta}(U_{1}^+,U_{3}^+, P^+,K^+))(r,\pm\theta_0, x_3)=0,\forall (r,x_3)\in [\xi,r_2]\times [-1,1],\\
  (U_{3}^+,\p_{x_3}^2U_{3}^+,\p_{x_3}(U_{1}^+,U_{2}^+, P^+,K^+))(r,\theta, \pm 1)=0,\forall (r,\theta)\in [\xi,r_2]\times [-\theta_0,\theta_0].
  \end{cases}\ee
\end{enumerate}
}\end{theorem}

\begin{remark}
{\it As far as we know, priori to our paper, there are only three papers (\cite{chen08,cy08,lxy16}) on the stability of the transonic shock for the 3D steady Euler system under general three dimensional perturbations of the incoming flows and exit pressure. However, the works in \cite{chen08,cy08} assumed a priorily the shock passing through a fixed point which is not physical and may lead to ill-posedness, while \cite{lxy16} required that the background transonic shock solutions satisfy the ``Structure Condition" which seems extremely difficult to verify. In Theorem \ref{existence}, we remove the above two stringent assumptions, the strength and position of the shock front is completely free and uniquely determined by the incoming flows and the exit pressure. Furthermore, there are no any restrictions on the background shock solutions and the opening angle $\theta_0$ of the cylinder.
}
\end{remark}

\begin{remark}\label{r2}
{\it The results in \cite{chen08,cy08,lxy16} are based on the characteristic decomposition of the steady Euler equations and the crucial fact that the pressure satisfies a second order elliptic equation in subsonic regions. One of main ingredients in our analysis here is quite different from those in \cite{chen08,cy08,lxy16}. We use the deformation-curl decomposition developed by us earlier in \cite{wx19,w19} to decouple the hyperbolic and elliptic modes and transform the steady Euler system to three transport equations for the Bernoulli's quantity, the entropy and the first component of the vorticity and a deformation-curl system for the velocity field.
}\end{remark}

\begin{remark}\label{r21}
{\it In applications of the deformation-curl decomposition to the transonic shock problem, there are several major obstacles that must be overcome: the mechanism of determining the shock front uniquely; the decomposition of the Rankine-Hugoniot condition that matches the deformation-curl decomposition well so that admissible boundary conditions for the hyperbolic and elliptic modes can be identified, respectively; and the regularity near the boundary caused by the intersection of the shock front and the cylinder wall and its propagation along the streamlines. One of the key elements in this paper is a new and elaborate decomposition of the Rankine-Hugoniot conditions, which determines the shock front uniquely by an algebraic equation, provides the boundary data on the shock front for the Bernoulli's quantitiy, the entropy and the first component of the vorticity which represent the hyperbolic modes, and an unusual second order differential boundary condition on the shock front to the velocity field which represents the elliptic modes. See \eqref{rh00}-\eqref{rh0} for detailed explanations.
}\end{remark}

\begin{remark}\label{r22}
{\it The compatibility conditions \eqref{pressure-cp} and \eqref{super3} are prescribed so that the regularity near the cylinder wall can be improved as $C^{2,\alpha}$. This regularity improvement is nontrivial since \eqref{super3} are prescribed only at the entrance. We must show that they are preserved not only by the transportation along the streamlines but also when the fluid moves across the shock front. We emphasize that the deformation-curl decomposition and the decomposition of the Rankine-Hugoniot conditions developed here are universal even for generic perturbations of supersonic incoming, the exit pressure and the nozzle geometry. The remaining unsolved difficulty for generic perturbations is the loss of uniqueness for the streamlines to a class of velocity fields with low regularity (only $C^{\alpha} (0<\alpha<1)$ regularity as expected).}
\end{remark}

\begin{remark}\label{r3}
{\it It should be noted that for the transonic shock flows in Theorem \ref{existence}, the velocity, the entropy and the pressure in the subsonic region have the same $C^{2,\alpha}(\overline{\Omega^+})$ regularity, which is in contrast to \cite{chen08,cy08,lxy16}, where the pressure has one order higher regularity than the velocity, thus they require higher regularity, $C^{3,\alpha}$, of the boundary datum in \cite{chen08,cy08,lxy16}. This improvement of regularity is one of the crucial advantages of the deformation-curl decomposition and should play more important role in the studies of the stability under generic perturbations of the nozzle geometry.
}\end{remark}

\begin{remark}\label{r5}
{\it Our method can be applied to establish the existence and stability of the spherical transonic shock without the ``Structure Condition" required in \cite{lxy16}, this will be reported in \cite{wx22b}.
}\end{remark}

We make some detailed explanations on the new ingredients in our analysis for the transonic shock problem. Note that the transonic shock problem can be formulated as a free boundary value problem to the steady Euler system which is hyperbolic-elliptic mixed in subsonic regions, whose effective decomposition into elliptic and hyperbolic modes is crucial for the solvability of the nonlinear free boundary problem. There are several different decompositions to the three dimensional steady Euler system in subsonic regions \cite{cx14,chen08,cy08,lxy16,w15,xy08b} developed by many researchers from different point of views. As pointed in above remarks, here we use the deformation-curl decomposition introduced in \cite{wx19,w19} to effectively decouple the hyperbolic and elliptic modes in subsonic region. The basic idea for this decomposition is as follows. It is well-known that the Bernoulli's quantity and the entropy are conserved along the streamlines. One can use the momentum equations to represent the second and third components of the vorticity as two algebraic equations for the Bernoulli's quantity, the entropy and the first component of the vorticity. Together with the divergence free condition for the vorticity, one could derive a transport equation for the first component of the vorticity. Furthermore, the continuity equation can be rewritten as a Frobenius inner product of a symmetric matrix and the deformation matrix by employing the Bernoulli's law and representing the density as a function of the Bernoulli's quantity, the entropy and the velocity field. This together with the vorticity equations constitute a deformation-curl system for the velocity field which is elliptic in subsonic regions.

Besides the deformation-curl decomposition, we also need an elegant decomposition of the Rankine-Hugoniot conditions to determine the shock front uniquely and identify suitable boundary conditions for the hyperbolic and elliptic modes. Indeed, the Rankine-Hugoniot conditions in \eqref{rh} can be reformulated as follows:
\be\label{rh00}\begin{cases}
\xi(\theta,x_3)-r_s=\frac{1}{a_1}(U_1^+(\xi(\theta,x_3),\theta,x_3)-\bar{U}^+(\xi(\theta,x_3)))-\frac{R_1}{a_1},\\
B^+(\xi(\theta,x_3),\theta,x_3)=B^-(\xi(\theta,x_3),\theta,x_3),\\
K^+(\xi(\theta,x_3),\theta,x_3)-\bar{K}^+= a_2 (\xi(\theta,x_3)-r_s) +R_2
\end{cases}\ee
and
\be\label{rh0}\begin{cases}
F_2(\theta,x_3):=\p_{\theta}\xi- a_0 \xi U_2^+(\xi(\theta,x_3),\theta,x_3) -\xi(\theta,x_3)g_2(\theta,x_3)=0,\ \ \text{in }E,\\
F_3(\theta,x_3):=\p_{x_3}\xi- a_0 U_3^+(\xi(\theta,x_3),\theta,x_3) -g_3(\theta,x_3)=0,\ \ \text{in }E.
\end{cases}\ee
where $a_0,a_1,a_2$ are constants depending on the background solutions, and $R_1, R_2,g_2,g_3$ are error terms (see Section \S\ref{22}). The shock front will be determined by the first equation in \eqref{rh00} in which the principal term is the difference between the radial velocity and the background radial velocity. The other two equations in \eqref{rh00} will be used as the boundary conditions for the Bernoulli's quantity and the entropy. It is important to realize that the system \eqref{rh0} is equivalent to the following divergence-curl system with normal boundary conditions:
\be\label{rh01}\begin{cases}
\frac{1}{r_s}\p_{\theta }F_3-\p_{x_3} F_2=0,\ \ &\text{in }E,\\
\frac{1}{r_s}\p_{\theta} F_2 + \p_{x_3} F_3=0, \ \ &\text{in }E,\\
n_2 F_2+ n_3 F_3=0,\ \ &\text{on }\partial E,
\end{cases}\ee
where $(n_2, n_3)$ is the unit outer normal to $\partial E$. Then the first equation in \eqref{rh01} yields the boundary data on the shock front for the first component of the vorticity. The second equation in \eqref{rh01} leads to an unusual second order differential boundary condition on the shock front for the deformation-curl first order elliptic system with nonlocal terms involving the trace of the radial velocity on the shock front (See \eqref{shock19}).

Note that after linearization, the source term in the curl system (i.e. \eqref{vor504}-\eqref{vor506}) is not divergence free, thus the solvability condition for the curl system does not hold in general. Instead we consider an enlarged deformation-curl system which includes a new unknown function $\Pi$ with the homogeneous Dirichlet boundary conditions. The Duhamel's principle is then used to solve this enlarged deformation-curl system. First we determine the function $\Pi$ by solving the Poisson equation with the homogeneous Dirichlet boundary conditions. Then we solve the standard div-curl system with the homogeneous normal boundary conditions. After homogenizing the curl system and introducing a potential function, we find that on the shock front the unusual second order differential boundary condition becomes the Poisson equation with homogeneous Neumann boundary conditions (i.e. the third equation in \eqref{rh01}) on the intersection of the shock front and the cylinder walls from which an oblique boundary condition for the potential function can be derived uniquely. The boundary value problem of the deformation-curl system for the velocity field is reduced to an oblique boundary value problem to a second order elliptic equation for the potential function with a nonlocal term involving only the trace of the potential function on the shock front. Another interesting issue that deserves further attentions is when using the deformation-curl decomposition to deal with the transonic shock problem, the exit pressure boundary condition becomes nonlocal since it involves the information from the shock front (see \eqref{pres4}). However, this nonlocal boundary condition reduces to be local after introducing a potential function (see \eqref{den37}).

This paper will be organized as follows. Section \ref{reformulation} is devoted to the decomposition of the hyperbolic and elliptic modes for the steady Euler equations in subsonic regions in terms of the deformation and curl, and the corresponding reformulation of the Rankine-Hugoniot jump conditions. We also introduce a coordinate transformation such that the free boundary becomes fixed. In Section \ref{proof}, we design an iteration scheme and solve the deformation-curl system with nonlocal terms and the unusual second order differential boundary condition on the shock front. In the Appendix, we verify the compatibility conditions of the supersonic flows and also the compatibility conditions on the intersection of the shock front and the cylinder wall.

\section{The reformulation of the transonic shock problem}\label{reformulation}

\subsection{The deformation-curl decomposition of the steady Euler equations}

The steady Euler system is hyperbolic-elliptic mixed in subsonic regions. Thus the solvability and formulation of suitable boundary conditions of even fixed boundary value problems for such mixed system is extremely subtle and difficult. Some of the key difficulties lie in identifying suitable hyperbolic and elliptic modes with proper boundary conditions. To overcome these difficulties, we utilize the deformation-curl decomposition for the steady Euler system introduced by the authors\cite{wx19,w19} to decouple the hyperbolic and elliptic modes. Let us give the details of the deformation-curl decomposition to the steady Euler system in cylindrical coordinates.

First, one can identify the hyperbolic modes in the system in \eqref{euler-cyl}. The Bernoulli's quantity and the entropy are transported by the following equations
\be\lab{ber10}
(\p_r+\frac{U_2}{U_1}\frac{1}{r}\p_{\theta} +\frac{U_3}{U_1}\p_{x_3}) B=0,\\\label{ent10}
(\p_r+\frac{U_2}{U_1}\frac{1}{r}\p_{\theta} +\frac{U_3}{U_1}\p_{x_3}) K=0.
\ee

Rewrite the vorticity $\bm{\omega}(r,\theta,x_3)=\text{curl }{\bf u}=\omega_1 {\bf e}_r + \omega_2 {\bf e}_{\theta}+ \omega_3 {\bf e}_3$ with
\be\no
\om_1= \frac{1}{r}\p_{\theta} U_3- \p_{x_3} U_2 ,\q \om_2=\p_{x_3} U_1- \p_r U_3,\q \om_3= \p_r U_2- \frac{1}{r}\p_\theta U_1 + \frac{U_2}{r}.
\ee
It then follows from the third and fourth equations in \eqref{euler-cyl} that
\be\no\begin{cases}
U_1 \om_3 - U_3 \om_1 + \frac{1}{r}\p_{\theta} B-\frac{B-\frac{1}{2}|{\bf U}|^2} {\gamma K(S)}\frac{1}{r}\p_{\theta} K(S)=0,\\
-U_1 \om_2 + U_2 \om_1 +\p_{x_3} B-\frac{B-\frac{1}{2}|{\bf U}|^2} {\gamma K(S) U_1}\p_{x_3} K(S)=0.
\end{cases}\ee
Thus one gets
\be\lab{vor13}\begin{cases}
\om_2 =\frac{U_2\om_1 +\p_{x_3} B}{U_1} - \frac{B-\frac{1}{2}|{\bf U}|^2} {\gamma K(S) U_1}\p_{x_3} K(S),\\
\om_3 =\frac{U_3\om_1 -\frac{1}{r}\p_{\theta} B}{U_1} + \frac{B-\frac{1}{2}|{\bf U}|^2} {\gamma K(S) U_1}\frac{1}{r}\p_{\theta} K(S).
\end{cases}\ee

Since
\be\no
\text{div }\text{curl }{\bf u}=\p_r\om_1+\frac{1}{r}\p_{\theta}\om_2+\p_{x_3}\om_3 +\frac{\omega_1}{r}=0,
\ee
substituting \eqref{vor13} into the above equation yields
\be\label{vor14}
&&(\p_r+\frac{U_2}{U_1}\frac{1}{r}\p_{\theta} +\frac{U_3}{U_1}\p_{x_3})\om_1 +(\f{1}{r} +\frac{1}{r}\p_{\theta}(\f{U_2}{U_1})+\p_{x_3}(\frac{U_3}{U_1}))\om_1+ \frac{1}{r}\p_{\theta}(\frac1{U_1}) \p_{x_3}B\\\no
&&-\p_{x_3}(\f{1}{U_1}) \frac{1}{r}\p_{\theta} B-\f{1}{r} \p_{\theta}(\frac{B-\frac{1}{2}|{\bf U}|^2}{\gamma K(S)U_1}) \p_{x_3} K(S) + \p_{x_3} (\frac{B-\frac{1}{2}|{\bf U}|^2}{\gamma K(S)U_1}) \frac{1}{r}\p_{\theta} K(S)=0.
\ee

Next, we study the elliptic modes in the steady Euler system \eqref{euler-cyl}. Using the Bernoulli's quantity $B=\f12 |{\bf U}|^2 + h(\rho,K)$ with $h(\rho,K)=e+\frac{P}{\rho}$ being the enthalpy, one can represent the density as a function of $B, K$, and $|{\bf U}|^2$:
\be\lab{den}
\rho= \rho(B,K, |{\bf U}|^2)=\left(\frac{\ga-1}{\ga K}\right)^{\frac{1}{\ga-1}}\left(B-\frac12 |{\bf U}|^2\right)^{\frac{1}{\ga-1}}.
\ee
Substituting \eqref{den} into the continuity equation and using \eqref{ent10} and \eqref{ber10} lead to
\be\no
&&(c^2(B,|{\bf U}|^2)-U_1^2)\p_r U_1 + (c^2(B,|{\bf U}|^2)-U_2^2)\frac{1}{r}\p_{\theta} U_2+ (c^2(B,|{\bf U}|^2)-U_3^2)\p_{x_3} U_3\\\no
&&\quad
+ \frac{c^2(B,|{\bf U}|^2) U_1}{r}=U_1(U_2\p_r U_2+ U_3 \p_r U_3)+ U_2(U_1\frac{1}{r}\p_{\theta} U_1+U_3\frac{1}{r}\p_{\theta} U_3)\\\lab{den11}
&&\quad+ U_3 (U_1\p_{x_3} U_1+U_2 \p_{x_3} U_2),
\ee
which can be rewritten as a Frobenius inner product of a symmetric matrix and the deformation matrix.

The equation \eqref{den11} together with the vorticity equations constitute a deformation-curl system for the velocity field:
\be\label{dc}\begin{cases}
(c^2-U_1^2)\p_r U_1 + (c^2-U_2^2)\frac{1}{r}\p_{\theta} U_2+ (c^2-U_3^2)\p_{x_3} U_3 + \frac{c^2 U_1}{r}=U_1(U_2\p_r U_2+ U_3 \p_r U_3)\\
\q+ U_2(U_1\frac{1}{r}\p_{\theta} U_1+U_3\frac{1}{r}\p_{\theta} U_3)+ U_3 (U_1\p_{x_3} U_1+U_2 \p_{x_3} U_2),\\
\frac{1}{r}\p_{\theta} U_3- \p_{x_3} U_2=\om_1,\\
\p_{x_3} U_1- \p_r U_3=\om_2,\\
\p_r U_2- \frac{1}{r}\p_\theta U_1 + \frac{U_2}{r}=\om_3.
\end{cases}\ee
The system \eqref{dc} is equivalent to an enlarged system including a new unknown function (see \eqref{den32} in Section \S\ref{proof}) and this enlarged system in subsonic region is elliptic in the sense of Agmon-Dougalis-Nirenberg \cite{adn64}. One may refer to \cite[Section 4]{w19} of a detailed verification for the ellipticity of an enlarged div-curl system in the sense of Agmon-Dougalis-Nirenberg.

\begin{lemma}\label{equiv}({\bf Equivalence.})
{\it Assume that $C^1$ smooth vector functions $(\rho, {\bf U}, K)$ defined on a domain $\Omega$ do not contain the vacuum (i.e. $\rho(r,\theta,x_3)>0$ in $\Omega$) and the radial velocity $U_1(r,\theta,x_3)>0$ in $\Omega$, then the following two statements are equivalent:
\begin{enumerate}[(1).]
  \item $(\rho, {\bf U}, K)$ satisfy the steady Euler system \eqref{euler-cyl} in $\Omega$;
  \item $({\bf U}, K, B)$ satisfy the equation \eqref{ber10}, \eqref{ent10}, \eqref{vor13} and \eqref{dc}.
\end{enumerate}
}\end{lemma}

\begin{proof}
We have proved that Statement (1) implies Statement (2). It remains to prove the converse. Define $\rho$ by \eqref{den}. Then using the last three equations in \eqref{dc} and the equations \eqref{vor13},  one can verify directly that the second and third momentum equations in \eqref{euler-cyl} are satisfied. These together with \eqref{ber10} imply that the first momentum equation in \eqref{euler-cyl} holds. Finally, according to \eqref{den}, the continuity equation in \eqref{euler-cyl} follows directly from the first equation in \eqref{dc}, \eqref{ber10} and \eqref{ent10}.

\end{proof}

\begin{remark}
{\it It is worthy noting the roles played by the equations \eqref{vor13}. On one hand, the equations \eqref{vor13} are used to derive a transport equation for $\omega_1$, \eqref{vor14}, which is hyperbolic. On the other hand, the equations \eqref{vor13} are also used to constitute a deformation-curl system for the velocity, which reveals the ellipticity for subsonic flows. Thus the hyperbolicity and ellipticity for subsonic flows are coupled in \eqref{vor13} and we decouple them in the above way which turns out to be effective and suit our purpose for solving various problems such as smooth transonic spiral flows in \cite{wxy21b} and the transonic shock problem herein. It should be noted that the equation \eqref{vor14} is not independent and is regarded as a byproduct of the divergence free of the vorticity and \eqref{vor13}.
}\end{remark}

\subsection{The reformulation of the Rankine-Hugoniot conditions and boundary conditions}\label{22}

Due to the mixed elliptic-hyperbolic structure of the steady Euler system in the subsonic region, it is important to formulate proper boundary conditions and their compatibility. To this end, we set
\be\no
&&W_1(r,\theta,x_3)= U_1^+(r,\th,x_3)- \bar{U}^+(r),\quad W_j(r,\theta,x_3)=U_j^+(r,\theta,x_3), j=2,3,\\\no
&&W_4(r,\theta,x_3)=K^+(r,\theta,x_3)-\bar{K}^+,\ W_5(r,\theta,x_3)= B^+(r,\theta,x_3)-\bar{B}^+,\ ,\\\no
&&W_6(\theta,x_3)= \xi(\theta,x_3)-r_s, \ \ \ {\bf W}=(W_1,\cdots, W_5).
\ee
Then the density and the pressure can be expressed as
\be\label{denw1}
&&\rho({\bf W})=\b(\frac{\gamma-1}{\gamma (\bar{K}^++W_4)}\b)^{\frac{1}{\gamma-1}}\b(\bar{B}^++W_5-\frac12(\bar{U}^++W_1)^2-\frac12\sum_{i=2}^3 W_i^2\b)^{\frac{1}{\gamma-1}},\\\label{denw2}
&&P({\bf W})=\b(\frac{(\gamma-1)^{\gamma}}{\gamma^{\gamma}(\bar{K}^++W_4)}\b)^{\frac{1}{\gamma-1}}\b(\bar{B}^++W_5-\frac12(\bar{U}^++W_1)^2-\frac12\sum_{i=2}^3 W_i^2\b)^{\frac{\gamma}{\gamma-1}}.
\ee

It follows from the third and fourth equations in \eqref{rh} that
\be\lab{shock11}
\frac{1}{\xi(\theta,x_3)}\p_{\theta} \xi =\frac{J_2(\xi,\theta,x_3)}{J(\xi,\theta,x_3)},\ \,\,\, \p_{x_3}  \xi =\frac{J_3(\xi,\theta,x_3)}{J(\xi,\theta,x_3)},
\ee
where
\be\no
&&J(\xi,\theta,x_3)=   [\rho U_2^2 + P] [\rho U_3^2 +P]- ([\rho U_2 U_3])^2,\\\no
&&J_2(\xi,\theta,x_3)= [\rho U_3^2+P] [\rho U_1 U_2]- [\rho U_1 U_3] [\rho U_2 U_3],\\\no
&&J_3(\xi,\theta,x_3)= [\rho U_2^2+P][\rho U_1 U_3]- [\rho U_1 U_2] [\rho U_2 U_3].
\ee
Rewrite \eqref{shock11} as
\be\lab{shock12}\begin{cases}
\p_{\theta} \xi= a_0 r_s W_2(\xi,\th,x_3) +r_s g_2(\bm{\Psi}^-(r_s+W_6,\theta,x_3) - \overline{\bm{\Psi}}^-(r_s+W_6),{\bf W}(\xi,\theta,x_3), W_6),\\
\p_{x_3} \xi= a_0 W_3(\xi,\th,x_3) +g_3(\bm{\Psi}^-(r_s+W_6,\theta,x_3) - \overline{\bm{\Psi}}^-(r_s+W_6),{\bf W}(\xi,\theta,x_3), W_6),
\end{cases}\ee
where $a_0= \f{(\bar{\rho}^+ \bar{U}^+)(r_s)}{[\bar{P}(r_s)]}>0$ and
\be\no
g_2=\frac{1}{r_s}\left(\frac{\xi J_2}{J}-a_0 r_s W_2(\xi(\th,x_3),\th,x_3)\right),\ \ \ g_3=\frac{J_3}{J}- a_0 W_3(\xi(\th,x_3),\th,x_3).
\ee
The functions $g_i, i=2,3$ are regarded as error terms which can be bounded by
\be\label{g23}
|g_i|\leq C_*(|\bm{\Psi}^-(r_s+W_6,\theta,x_3) - \overline{\bm{\Psi}}^-(r_s+W_6)|+|{\bf W}(\xi,\theta,x_3)|^2+|W_6|^2).
\ee
To see this, we only estimate $g_3$, the estimate of $g_2$ is similar. Indeed,
\be\no
&&g_3
= \frac{[\rho U_2^2+P][\rho U_1 U_3]}{J}- a_0 W_3- \frac{[\rho U_1 U_2] [\rho U_2 U_3]}{J}\\\no
&&=W_3\left\{\frac{\rho^+ U_1^+ }{[ P]}-a_0\right\}-\frac{\rho^+ U_1^+W_3}{[P]}\frac{[\rho U_3^2]}{[\rho U_3^2 +P]}-\frac{\rho^- U_1^- U_3^-}{[\rho U_3^2 +P]}\\\no
&&\quad+\frac{[\rho U_1 U_3]}{[\rho U_3^2 +P]}\frac{([\rho U_2 U_3])^2}{[\rho U_2^2 + P] [\rho U_3^2 +P]- ([\rho U_2 U_3])^2}- \frac{[\rho U_1 U_2] [\rho U_2 U_3]}{J}
\ee
and
\be\no
&&\left(\frac{\rho^+ U_1^+ }{[P]}-a_0\right)W_3=\frac{(\bar{\rho}^+(\xi) +\hat{\rho})(\bar{U}^+(\xi)+W_1)W_3}{\bar{P}^+(\xi)-\bar{P}^-(\xi)+P^+(\xi)-\bar{P}^+(\xi)-(P^-(\xi)-\bar{P}^-(\xi))}-a_0W_3\\\no
&&=W_3\left(\frac{(\bar{\rho}^+\bar{U}^+)(r_s+W_6)}{(\bar{P}^+-\bar{P}^-)(r_s+W_6)}-\frac{(\bar{\rho}^+\bar{U}^+)(r_s)}{(\bar{P}^+-\bar{P}^-)(r_s)}\right)\\\no
&&\quad- W_3\frac{(\bar{\rho}^+\bar{U}^+)(\xi)}{\bar{P}^+(\xi)-\bar{P}^-(\xi)}\frac{P^+(\xi)-\bar{P}^+(\xi)-(P^-(\xi)-\bar{P}^-(\xi))}{\bar{P}^+(\xi)-\bar{P}^-(\xi)+P^+(\xi)-\bar{P}^+(\xi)-(P^-(\xi)-\bar{P}^-(\xi))}
\\\no
&&\quad+ W_3\frac{\bar{\rho}^+(\xi)W_1 + \bar{U}^+(\xi) \hat{\rho} + \hat{\rho} W_1}{\bar{P}^+(\xi)-\bar{P}^-(\xi)+P^+(\xi)-\bar{P}^+(\xi)-(P^-(\xi)-\bar{P}^-(\xi))}.
\ee
Thus the estimate \eqref{g23} for $i=3$ follows easily.

It follows from \eqref{shock11} and \eqref{rh} that
\be\label{rh1}\begin{cases}
[\rho U_1]=\frac{[\rho U_2] J_2+[\rho U_3] J_3}{J},\\
[\rho U_1^2+P]=\frac{[\rho U_1 U_2] J_2+[\rho U_1 U_3] J_3}{J},\\
[B]=0.
\end{cases}\ee
Denote $\dot{\rho}(r,\theta,x_3)= \rho^+(r,\theta,x_3)-\bar{\rho}^+(r)$. Then the first equation in \eqref{rh1} implies that
\be\no
&&-[\bar{\rho} \bar{U}](\xi)+(\rho^- U_1^-)(\xi,\theta,x_3)-(\bar{\rho}^-\bar{U}^-)(\xi)+\frac{[\rho U_2] J_2+[\rho U_3] J_3}{J}\\\no
&&=(\rho^+ U_1^+)(\xi,\theta,x_3)-(\bar{\rho}^+\bar{U}^+)(\xi)\\\no
&&=\bar{\rho}^+(r_s)W_1(\xi,\theta,x_3)+ \bar{U}^+(r_s)\dot{\rho}(\xi,\theta,x_3)+(W_1+\bar{U}^+(\xi)-\bar{U}^+(r_s))\dot{\rho}(\xi,\theta,x_3)\\\no
&&\quad\quad+ (\bar{\rho}^+(\xi)-\bar{\rho}^+(r_s))W_1(\xi,\theta,x_3).
\ee
Thus
\be\no
&&\bar{\rho}^+(r_s) W_1(\xi,\theta,x_3)+{\bar U}^+(r_s) \dot{\rho}(\xi,\theta,x_3)=-[\bar{\rho} \bar{U}](\xi)+ \frac{[\rho U_2] J_2+[\rho U_3] J_3}{J}\\\no
&&+(\rho^- U_1^-)(\xi,\theta,x_3)-(\bar{\rho}^-\bar{U}^-)(\xi)-(W_1+\bar{U}^+(r_s+W_6)-\bar{U}^+(r_s))\dot{\rho}(\xi,\theta,x_3)\\\no
&&\quad-(\bar{\rho}^+(r_s+W_6)-\bar{\rho}^+(r_s))W_1(\xi,\theta,x_3):=R_{01}(\bm{\Psi}^-(\xi,\theta,x_3) - \overline{\bm{\Psi}}^-(\xi),{\bf W}(\xi,\theta,x_3), W_6).
\ee

Similarly, one can conclude from \eqref{rh1} that at $(\xi(\theta,x_3),\theta,x_3)$, it holds that
\be\label{rh2}\begin{cases}
\bar{\rho}^+(r_s) W_1+{\bar U}^+(r_s) \dot{\rho} = R_{01}(\bm{\Psi}^-(\xi,\theta,x_3) - \overline{\bm{\Psi}}^-(\xi),{\bf W}(\xi,\theta,x_3), W_6),\\
2(\bar{\rho}^+ \bar{U}^+)(r_s)W_1+ \{(\bar{U}^+(r_s))^2+c^2(\bar{\rho}^+(r_s),\bar{K}^+)\}\dot{\rho}+(\bar{\rho}^+(r_s))^{\gamma} W_4\\
\q=-\frac{1}{r_s}[\bar{P}(r_s)] W_6+ R_{02}(\bm{\Psi}^-(\xi,\theta,x_3) - \overline{\bm{\Psi}}^-(\xi),{\bf W}(\xi,\theta,x_3), W_6),\\
\bar{U}^+(r_s) W_1+ \frac{c^2(\bar{\rho}^+(r_s),\bar{K}^+)}{\bar{\rho}^+(r_s)} \dot{\rho}+ \frac{\ga (\bar{\rho}^+(r_s))^{\gamma-1}}{(\ga-1)} W_4=R_{03},
\end{cases}\ee
where
\be\no
&&R_{02}=-\left\{[\bar{\rho} \bar{U}^2+ \bar{P}](r_s +W_6)-\frac{1}{r_s}[\bar{P}(r_s)] W_6\right\}+ (\rho^- (U_1^-)^2 +P^-)(\xi,\theta,x_3)\\\no
&&\quad-(\bar{\rho}^-(\bar{U}^-)^2+\bar{P}^-)(\xi)-\bigg\{(\rho^+ (U_1^+)^2 +P^+)(\xi,\theta,x_3)-(\bar{\rho}^+(\bar{U}^+)^2+\bar{P}^+)(\xi)\\\no
&&\quad\quad-2(\bar{\rho}^+ \bar{U}^+)(r_s)W_1- \{(\bar{U}^+(r_s))^2+c^2(\bar{\rho}^+(r_s),\bar{K}^+)\}\dot{\rho}-(\bar{\rho}^+(r_s))^{\gamma} W_4\bigg\}\\\no
&&\quad\quad+\frac{[\rho U_1 U_2] J_2+[\rho U_1 U_3] J_3}{J},\\\no
&&R_{03}= B^-(r_s+ W_6,\theta,x_3)- \bar{B}^-- \bar{U}^+(r_s+W_6) W_1(\xi,\theta,x_3)-\frac{1}{2}\sum_{j=1}^3 W_j^2(\xi,\theta,x_3)\\\no
&&\quad\quad -\frac{\gamma}{\gamma-1}\left((\bar{K}^+ +W_4(\xi,\theta,x_3))(\rho(\xi,\theta,x_3))^{\gamma-1}-\bar{K}^+ (\bar{\rho}^+(\xi))^{\gamma-1}\right)\\\no
&&\quad\quad + \bar{U}^+(r_s) W_1+ \frac{c^2(\bar{\rho}^+(r_s),\bar{K}^+)}{\bar{\rho}^+(r_s)} \dot{\rho}+ \frac{\ga (\bar{\rho}^+(r_s))^{\gamma-1}}{(\ga-1)} W_4.
\ee

Note that
\be\no
\frac{d}{dr}(\bar{\rho}^{\pm} \bar{U}^{\pm})(r)= -\frac{1}{r}(\bar{\rho}^{\pm} \bar{U}^{\pm})(r),\quad\quad
\frac{d}{dr}(\bar{\rho}^{\pm} (\bar{U}^{\pm})^2+ \bar{P}^{\pm})=-\frac{1}{r}\bar{\rho}^{\pm} (\bar{U}^{\pm})^2,
\ee
hence
\be\no
[\bar{\rho} \bar{U}](r_s+W_6)=O(W_6^2),\q [\bar{\rho}\bar{U}^2+\bar{P}](r_s+W_6)-\frac{1}{r_s}[\bar{P}(r_s)] W_6=O(W_6^2).
\ee

Using the equation \eqref{denw1}, $\dot{\rho}$ can be represented as a function of ${\bf W}$ and $W_6$, there exists a constant $C_0>0$ depending only on the background solution, such that
\be\no
|R_{0i}|\leq C_0(|\bm{\Psi}^-(\xi,\theta,x_3)-\overline{\bm{\Psi}}^-(\xi)|+ |{\bf W}(\xi,\theta,x_3)|^2+|W_6|^2), i=1,2,3.
\ee

Then solving the algebraic equations in \eqref{rh2}, one gets
\be\lab{shock13}\begin{cases}
W_1(\xi,\theta,x_3)=a_1 W_6+ R_1(\bm{\Psi}^-(r_s+W_6,\theta,x_3) - \overline{\bm{\Psi}}^-(r_s+W_6),{\bf W}(\xi,\theta,x_3), W_6),\\
W_4(\xi,\theta,x_3)=a_2 W_6+R_2(\bm{\Psi}^-(r_s+W_6,\theta,x_3) - \overline{\bm{\Psi}}^-(r_s+W_6),{\bf W}(\xi,\theta,x_3), W_6),\\
W_5(\xi,\theta,x_3)=B^-(r_s+W_6(\theta,x_3),\theta,x_3)-\bar{B}^-.
\end{cases}\ee
where
\be\no
&&a_1=\f{\gamma \bar{U}^+(r_s)[\bar{P}(r_s)]}{r_s\bar{\rho}^+(r_s) (c^2(\bar{\rho}^+(r_s),\bar{K}^+)-(\bar{U}^+(r_s))^2)}>0,\\\no
&&a_2=\frac{(\gamma-1)[\bar{P}(r_s)]}{r_s(\bar{\rho}^+(r_s))^{\gamma}}>0
\ee
and
\be\no
&&R_1=\frac{(c^2(\bar{\rho}^+(r_s),\bar{K}^+)+\gamma(\bar{U}^+(r_s))^2) R_{01}-\gamma \bar{U}^+(r_s)R_{02} + (\gamma-1)(\bar{\rho}^+ \bar{U}^+)(r_s)R_{03}}{\bar{\rho}^+(r_s) (c^2(\bar{\rho}^+(r_s),\bar{K}^+)-(\bar{U}^+(r_s))^2)}:= \sum_{i=1}^3 b_{1i} R_{0i},\\\no
&&R_2=\frac{\gamma-1}{(\bar{\rho}^+(r_s))^{\gamma-1}}\left(\bar{U}^+(r_s)R_{01}-R_{02}+ \bar{\rho}^+(r_s) R_{03}\right):= \sum_{i=1}^3 b_{2i} R_{0i}.
\ee

In the following, the superscript ``+" in $\bar{U}^+, \bar{P}^+,\bar{K}^+, \bar{B}^+$ will be ignored to simplify the notations. To derive the boundary conditions at the exit, one has by the definition of the Bernoulli's quantity that
\be\no
W_5= \bar{U}W_1 +\f{\p h}{\p P}(\bar{P}, \bar{K})(P-\bar{P})+ \frac{\p h}{\p K}(\bar{P}, \bar{K})W_4+ \frac12 \sum_{j=1}^3 W_j^2+E({\bf W}(r,\theta,x_3)),
\ee
where
\be\no
&&E({\bf W}(r,\theta,x_3))=\frac{\gamma}{\gamma-1}(\bar{K}+W_4)^{\frac1{\gamma}}(P({\bf W}))^{\frac{\gamma-1}{\gamma}}-\frac{\gamma}{\gamma-1}\bar{K}^{\frac1{\gamma}}\bar{P}^{\frac{\gamma-1}{\gamma}}\\\no
&&\quad\quad\quad-\frac{1}{\bar{\rho}(r)}(P({\bf W})- \bar{P})-\frac{\bar{B}-\frac{1}{2}\bar{U}^2(r)}{\gamma \bar{K}} W_4.
\ee

This, together with \eqref{pressure} implies that
\be\no
&&W_1(r_2,\theta,x_3)=\frac{W_5(r_2,\theta,x_3)}{\bar{U}(r_2)}- \frac{\epsilon P_{ex}(\theta,x_3)}{(\bar{\rho} \bar{U})(r_2)} -\frac{\bar{B}-\frac{1}{2}\bar{U}^2(r_2)}{\gamma \bar{K}\bar{U}(r_2)} W_4(r_2,\theta,x_3)\\\label{pres2}
&&\quad\quad-\frac{1}{2\bar{U}(r_2)}\sum_{j=1}^3 W_j^2(r_2,\theta,x_3)-\frac{1}{\bar{U}(r_2)}E({\bf W}(r_2,\theta,x_3)).
\ee
Note that $E$ is an error term that can be bounded by
\be\no
|E({\bf W}(r_2,\theta,x_3))|\leq C_*|{\bf W}(r_2,\theta,x_3)|^2.
\ee

The boundary conditions for $W_2$ and $W_3$ on the nozzle walls are
\be\label{slip15}\begin{cases}
W_2(r,\pm\theta_0, x_3)=0,\ \ &\text{on } r_s+W_6(\pm\theta_0,x_3)<r<r_2,x_3\in [-1,1],\\
W_3(r,\theta,\pm 1)=0,\ \ &\text{on }r_s+W_6(\theta,\pm 1)<r<r_2,\theta\in [-\theta_0,\theta_0].
\end{cases}\ee


One can rewrite the equations \eqref{ent10},\eqref{ber10}, \eqref{vor13} and \eqref{vor14} in terms of ${\bf W}$ as follows. The equations for the hyperbolic quantities $W_4$ and $W_5$ are
\be\label{ent11}
&&\left(\p_r+\frac{W_2}{\bar{U}+W_1}\frac{1}{r}\p_{\theta}+\frac{W_3}{\bar{U}+W_1}\p_{x_3}\right) W_4=0,\\\label{ber11}
&&\left(\p_r+\frac{W_2}{\bar{U}+W_1}\frac{1}{r}\p_{\theta}+\frac{W_3}{\bar{U}+W_1}\p_{x_3}\right) W_5=0.
\ee
The equations for the vorticity $\omega$ are
\be\no
&&\left(\p_r+\frac{W_2}{\bar{U}+W_1}\frac{1}{r}\p_{\theta} +\frac{W_3}{\bar{U}+W_1}\p_{x_3}\right)\om_1 +\left(\f{1}{r} +\frac{1}{r}\p_{\theta}\b(\frac{W_2}{\bar{U}+W_1}\b)+\p_{x_3}\b(\frac{W_3}{\bar{U}+W_1}\b)\right)\om_1\\\label{vor21}
&&\quad+ \frac{1}{r}\p_{\theta}\b(\frac1{\bar{U}+W_1}\b) \p_{x_3}W_5-\p_{x_3}\b(\frac{1}{\bar{U}+W_1}\b) \frac{1}{r}\p_{\theta} W_5\\\no
&&\quad-\f{1}{r} \p_{\theta} \b(\frac{\bar{B}-\frac{1}{2}\bar{U}^2+W_5-\bar{U}W_1-\frac{1}{2}\sum_{j=1}^3W_j^2}{\gamma (\bar{K}+W_4)(\bar{U}+W_1)}\b) \p_{x_3} W_4\\\no
&&\quad+ \p_{x_3} \b(\frac{\bar{B}-\frac{1}{2}\bar{U}^2+W_5-\bar{U}W_1-\frac{1}{2}\sum_{j=1}^3W_j^2}{\gamma (\bar{K}+W_4)(\bar{U}+W_1)}\b) \frac{1}{r}\p_{\theta} W_4=0,
\ee
and
\be\label{vor22}\begin{cases}
\om_2 =\frac{W_2\om_1 +\p_{x_3} W_5}{\bar{U}+W_1} - \frac{\bar{B}-\frac{1}{2}\bar{U}^2+W_5-\bar{U}W_1-\frac{1}{2}\sum_{j=1}^3W_j^2} {\gamma (\bar{K}+W_4)(\bar{U}+W_1)}\p_{x_3} W_4,\\
\om_3 =\frac{W_3\om_1 -\frac{1}{r}\p_{\theta} W_5}{\bar{U}+W_1} + \frac{\bar{B}-\frac{1}{2}\bar{U}^2+W_5-\bar{U}W_1-\frac{1}{2}\sum_{j=1}^3W_j^2} {\gamma (\bar{K}+W_4)(\bar{U}+W_1)}\frac{1}{r}\p_{\theta} W_4.
\end{cases}\ee

It follows from \eqref{den11} that
\be\no
&&(1-\bar{M}^2(r))\p_r W_1 + \frac{1}{r}\p_{\theta} W_2 +\p_{x_3} W_3 + \frac{1}{r}(1+\frac{\bar{M}^2(2+(\gamma-1)\bar{M}^2)}{1-\bar{M}^2})W_1\\\label{den15}
&&\quad=-\frac{(\gamma-1)(\bar{U}'+\frac{\bar{U}}{r})}{\bar{c}^2(r)} W_5+ F({\bf W}),
\ee
where
\be\no
&&F({\bf W})=-\frac{(\gamma-1)(\p_r W_1+\frac{W_1}{r})}{\bar{c}^2(r)} W_5+ \frac{\bar{U}'+\p_r W_1}{\bar{c}^2(r)}(\frac{\gamma+1}{2} W_1^2+\frac{\gamma-1}{2}(W_2^2+W_3^2))\\\no
&&\quad\quad+\frac{(\gamma-1)(\bar{U}+W_1)}{2r \bar{c}^2(r)}\sum_{j=1}^3 W_j^2+\frac{(\gamma+1)\bar{U}W_1 \p_r W_1+ (\gamma-1)\bar{U} \frac{W_1^2}{r}}{\bar{c}^2(r)}\\\no
&&\quad\quad- \frac{1}{\bar{c}^2(r)}((\gamma-1)W_5-\frac{\gamma-1}{2}\sum_{j=1}^3 W_j^2-(\gamma-1)\bar{U}W_1)(\frac{1}{r}\p_{\theta} W_2+\p_{x_3} W_3)\\\no
&&\quad\quad
+\frac{1}{\bar{c}^2(r)}(W_2^2\frac{1}{r}\p_{\theta} W_2 +W_3^2 \p_{x_3}W_3)+ \frac{\bar{U}+W_1}{\bar{c}^2(r)}(W_2\p_r W_2+W_3\p_r W_3)\\\no
&&\quad\quad+\frac{W_2}{\bar{c}^2(r)}((\bar{U}+W_1)\frac{1}{r}\p_{\theta} W_1 +\frac{W_3}{r}\p_{\theta} W_3)+ \frac{V_3}{\bar{c}^2(r)}((\bar{U}+W_1)\p_{x_3} W_1 +W_2\p_{x_3} W_2),\\\no
&&\bar{c}^2(r)=c^2(\bar{B},\bar{U}^2(r)).
\ee
Here $F({\bf W})$ and the following $H_i, G_i$ are quadratic and high order terms. In order to verify that these terms satisfy some compatibility conditions, their exact form are presented here.

Then to solve the problem \eqref{euler-cyl} with \eqref{super1}-\eqref{pressure}, and \eqref{rh} is equivalent to find a function $W_6$ defined on $E$ and vector functions $(W_1,\cdots, W_5)$ defined on the $\Omega_{W_6}:=\{(r,\theta,x_3): r_s +W_6(\theta,x_3)<r<r_2,(\theta,x_3)\in E\}$, which solves the equations \eqref{ent11}--\eqref{den15} with boundary conditions \eqref{shock12},\eqref{shock13},\eqref{pres2} and \eqref{slip15}.

\subsection{The coordinate transformation}\noindent

To fix the subsonic domain, relabeling $V_6(\theta,x_3)=\xi(\theta,x_3)-r_s$, one can introduce the following coordinates transformation
\be\label{coor}
y_1=\frac{r-r_s-V_6}{r_2-r_s-V_6}(r_2-r_s) + r_s,\ y_2=\theta,\  y_3=x_3.
\ee
Then
\be\no\begin{cases}
r= y_1+\frac{r_2-y_1}{r_2-r_s}V_6=: D_0^{V_6},\\
\p_r=\frac{r_2-r_s}{r_2-r_s-V_6(y_2,y_3)} \p_{y_1}=: D_1^{V_6},\\
\frac{1}{r}\p_{\theta}=\frac{1}{D_0^{V_6}}(\p_{y_2}+\frac{(y_1-r_2)\p_{y_2}V_6}{r_2-r_s-V_6}\p_{y_1})=: D_2^{V_6},\\
\p_{x_3}=\p_{y_3}+\frac{(y_1-r_2)\p_{y_3}V_6}{r_2-r_s-V_6}\p_{y_1}=: D_3^{V_6},
\end{cases}\ee
and the domain $\Omega^+$ is changed to be
\be\no
\mathbb{D}=\{(y_1,y'): y_1\in (r_s, r_2), y'=(y_2,y_3)\in E\}.
\ee
Denote
\be\no
&&\Sigma_2^{\pm}=\{(y_1,\pm\theta_0,y_3):(y_1,y_3)\in (r_s,r_2)\times (-1,1)\},\\\no
&&\Sigma_3^{\pm}=\{(y_1,y_2,\pm 1):(y_1,y_2)\in (r_s,r_2)\times (-\theta_0,\theta_0)\}.
\ee

Set
\be\no\begin{cases}
V_j(y)= W_j(y_1+\f{r_2-y_1}{r_2-r_s}V_6,y_2,y_3), j=1,\cdots, 5,\\
\tilde{\omega}_j(y)=\omega_j(y_1+\f{r_2-y_1}{r_2-r_s}V_6,y_2,y_3), j=1,2,3.
\end{cases}\ee
Then the functions $\rho(r,\theta,x_3)$ and $P(r,\theta,x_3)$ in \eqref{denw1}-\eqref{denw2} are transformed to be
\be\no
&&\tilde{\rho}({\bf V}(y),V_6)=\left(\frac{\gamma-1}{\gamma(\bar{K}+V_4)}\right)^{\frac{1}{\gamma-1}}\left(\bar{B}+V_5-\frac{1}{2}(\bar{U}(D_0^{V_6})+V_1)^2-\frac{1}{2}\sum_{i=2}^3 V_j^2\right)^{\frac{1}{\gamma-1}},\\\no
&&\tilde{P}({\bf V}(y),V_6)=\left(\frac{\gamma-1}{\gamma}\right)^{\frac{\gamma}{\gamma-1}}\left(\frac{1}{\bar{K}+V_4}\right)^{\frac{1}{\gamma-1}}\left(\bar{B}+V_5-\frac{1}{2}(\bar{U}(D_0^{V_6})+V_1)^2-\frac{1}{2}\sum_{i=2}^3 V_j^2\right)^{\frac{\gamma}{\gamma-1}}.
\ee
In the $y$-coordinates, \eqref{shock12} is changed to be
\be\no
&&\frac{1}{r_s}\p_{y_2} V_6(y')=a_0 V_2(r_s,y') + g_2({\bf V}(r_s,y'), V_6),\\\no
&&\p_{y_3} V_6(y')=a_0 V_3(r_s,y') + g_3({\bf V}(r_s,y'), V_6),
\ee
where
\be\label{g21}
&&g_2({\bf V}(r_s,y'), V_6)=\frac{1}{r_s}\left(\frac{(r_s+V_6)J_2({\bf V}(r_s,y'), V_6(y'))}{J({\bf V}(r_s,y'), V_6(y'))}-a_0r_s V_2(r_s,y')\right),\\\label{g31}
&&g_3({\bf V}(r_s,y'), V_6)=\frac{1}{r_s}\left(\frac{J_3({\bf V}(r_s,y'), V_6(y'))}{J({\bf V}(r_s,y'), V_6(y'))}-a_0 V_3(r_s,y')\right),
\ee
and the exact formulas for $J_2, J_3$ and $J$ are given in the Appendix (See \eqref{j21}-\eqref{j1}). These will be used to verify the compatibility conditions required below (See \eqref{j211}).

In the $y$ coordinates, the transonic shock problem can be reformulated as follows. The shock front will be determined by the first equation in \eqref{shock13} as follows
\be\label{shock400}
V_6(y_2,y_3)=\frac{1}{a_1}V_1(r_s,y_2,y_3)- \frac{1}{a_1}R_1({\bf V}(r_s,y'),V_6(y')),
\ee
where $R_{1}({\bf V}(r_s,y'),V_6(y'))=\displaystyle\sum_{i=1}^3 b_{1i}R_{0i}({\bf V}(r_s,y'),V_6(y'))$ and the exact formulas for $R_{0i}, i=1,2,3$ in $y$-coordinates are given in the Appendix (See \eqref{rv01}-\eqref{rv03}). These will be employed to verify the compatibility conditions (See \eqref{rv11}-\eqref{rv41} below).

The second and third formula in \eqref{shock13} will be used to solve the Bernoulli's quantity and entropy. The function $V_4$ and $V_5$ would satisfy (see \eqref{ber11} and \eqref{ent11})
\be\label{ber31}\begin{cases}
\b(D_1^{V_6}+\f{V_2}{\bar{U}(D_0^{V_6})+V_1}D_2^{V_6}+\frac{V_3}{\bar{U}(D_0^{V_6})+V_1}D_3^{V_6}\b)V_5=0,\\
V_5(r_s,y')=B^-(r_s+V_6(y'),y')-\bar{B}^-
\end{cases}\ee
and
\be\label{ent31}\begin{cases}
\b(D_1^{V_6}+\f{V_2}{\bar{U}(D_0^{V_6})+V_1}D_2^{V_6}+\frac{V_3}{\bar{U}(D_0^{V_6})+V_1}D_3^{V_6}\b)V_4=0,\\
V_4(r_s,y')=a_2 V_6(y')+R_2({\bf V}(r_s,y'),V_6(y'))
\end{cases}\ee
where $R_{2}({\bf V}(r_s,y'),V_6(y'))=\displaystyle\sum_{i=1}^3 b_{2i}R_{0i}({\bf V}(r_s,y'),V_6(y'))$.

The following reformulation of the jump conditions \eqref{shock12} is crucial for us to solve the transonic shock problem. Note that \eqref{shock12} is equivalent to
\be\no\begin{cases}
F_2(y'):=\frac{1}{r_s}\p_{y_2} V_6- a_0 V_2(r_s, y')- g_2({\bf V}(r_s,y'),V_6)\equiv 0,\ \ &\forall y'\in E,\\
F_3(y'):=\p_{y_3}V_6- a_0 V_3(r_s,y')- g_3({\bf V}(r_s,y'),V_6)\equiv 0,\ \ &\forall y'\in E.
\end{cases}\ee
Then a key observation is the following
\begin{lemma}\label{equi0}
{\it Let $F_j, j=2,3$ be two $C^1$ smooth functions defined on $\overline{E}$. Then the following two statements are equivalent
\begin{enumerate}[(1)]
  \item $F_2=F_3\equiv 0$ on $\overline{E}$;
  \item $F_2$ and $F_3$ solve the following problem
  \be\label{equi00}\begin{cases}
  \frac{1}{r_s}\p_{y_2}F_3-\p_{y_3} F_2=0,\ \ &\text{in}\ E,\\
  \frac{1}{r_s}\p_{y_2} F_2 + \p_{y_3} F_3=0,\ \ &\text{in}\ E,\\
  F_2(\pm\theta_0, y_3)=0,\ \ &\text{on}\ \ y_3\in [-1,1],\\
  F_3(y_2, \pm 1)=0,\ \ &\text{on}\ \ y_2\in [-\theta_0,\theta_0].
\end{cases}\ee
\end{enumerate}
}\end{lemma}

\begin{proof}
It suffices to show that (2) implies (1). Indeed, it follows from the first equation in \eqref{equi00} that there exists a potential function $\Phi(y_2,y_3)\in C^2(\overline{E})$ such that $F_2=\p_{y_2} \Phi$ and $\frac{1}{r_s} F_3=\p_{y_3} \Phi$ on $\overline{E}$. Then the second equation in \eqref{equi00} yields
\be\no\begin{cases}
\frac{1}{r_s^2}\p_{y_2}^2\Phi+\p_{y_3}^2 \Phi = 0,\ \ &\text{in }E,\\
\p_{y_2} \Phi(\pm\theta_0, y_3)=0,\ \ &\text{on}\ \ y_3\in [-1,1],\\
\p_{y_3}\Phi(y_2,\pm 1)=0,\ \ &\text{on}\ \ y_2\in [-\theta_0,\theta_0].
\end{cases}\ee
Thus one can conclude that $\Phi\equiv \text{const}$ on $\overline{E}$ and thus $F_2=F_3\equiv 0$ on $\overline{E}$.
\end{proof}

The first equation in \eqref{equi00} yields
\be\label{shock17}
(\frac{1}{r_s}\p_{y_2} V_3-\p_{y_3}V_2)(r_s,y')=\frac{1}{a_0} (\p_{y_3}g_2-\frac{1}{r_s}\p_{y_2}g_3)({\bf V}(r_s,y'), V_6),
\ee
which gives the boundary data on the shock front for the first component of the vorticity.

The second equation in \eqref{equi00} gives
\be\no
(\frac{1}{r_s^2}\p_{y_2}^2 V_6+ \p_{y_3}^2 V_6-\frac{a_0}{r_s}\p_{y_2} V_2-a_0\p_{y_3} V_3)(r_s, y')=(\frac{1}{r_s}\p_{y_2}g_2+ \p_{y_3}g_3)({\bf V}(r_s,y'), V_6).
\ee
This, together with \eqref{shock400}, shows
\be\label{shock19}
\{(\frac{1}{r_s^2}\p_{y_2}^2 V_1+\p_{y_3}^2 V_1)-a_0a_1(\frac{1}{r_s}\p_{y_2} V_2+\p_{y_3} V_3)\}(r_s,y')=q_1({\bf V}(r_s,y'),V_6),
\ee
with
\be\no
q_1=a_1(\frac{1}{r_s}\p_{y_2}g_2+\p_{y_3}g_3)({\bf V}(r_s,y'), V_6)+(\frac{1}{r_s^2}\p_{y_2}^2 R_1+\p_{y_3}^2 R_1)({\bf V}(r_s,y'), V_6).
\ee
The condition \eqref{shock19} is used as the boundary condition on the shock front for the deformation-curl system associated with the velocity field.

The boundary conditions in \eqref{equi00} can be rewritten as
\be\label{shock20}
&&\left(\frac{1}{r_s}\p_{y_2}V_1-a_0a_1V_2\right)(r_s,\pm \theta_0,y_3)=q_2^{\pm}({\bf V}(r_s,\pm\theta_0,y_3),V_6(\pm\theta_0,y_3)),\\\label{shock21}
&&(\p_{y_3} V_1-a_0 a_1 V_3)(r_s,y_2,\pm 1)= q_3^{\pm}({\bf V}(r_s,y_2,\pm 1),V_6(y_2,\pm 1)),
\ee
with
\be\no
&&q_2^{\pm}({\bf V}(r_s,\cdot),V_6)(\pm\theta_0,y_3)=(\frac{1}{r_s}\p_{y_2} \{R_1({\bf V}(r_s,\cdot), V_6(\cdot))\}+ g_2({\bf V}(r_s,\cdot), V_6))(\pm \theta_0, y_3),\\\no
&&q_3^{\pm}({\bf V}(r_s,\cdot),V_6)(y_2,\pm 1)=(\p_{y_3}\{R_1({\bf V}(r_s,\cdot), V_6(\cdot))\}+g_3({\bf V}(r_s,\cdot), V_6))(y_2,\pm 1).
\ee
The roles of \eqref{shock20} and \eqref{shock21} will be indicated later.

Next we determine the vorticity. Rewrite \eqref{vor21} as
\be\label{vor400}
\b(D_1^{V_6}+\frac{1}{\bar{U}(D_0^{V_6})+V_1}\sum\limits_{j=2}^3V_j D_j^{V_6}\b) \tilde{{\omega}}_1 + \mu({\bf V},V_6)\tilde{\omega}_1=H_0({\bf V},V_6),
\ee
where
\be\no
&&\mu({\bf V},V_6)=\bigg\{\displaystyle\sum_{j=2}^3D_j^{V_6}\bigg(\frac{V_j}{\bar{U}(D_0^{V_6})+V_1}\bigg)+\frac{1}{D_0^{V_6}}\bigg\},\\\no
&&H_0({\bf V}, V_6)=D_3^{V_6}\left(\frac{1}{\bar{U}(D_0^{V_6})+V_1}\right)D_2^{V_6} V_5- D_2^{V_6}\left(\frac{1}{\bar{U}(D_0^{V_6})+V_1}\right)D_3^{V_6} V_5\\\no
&&\quad+D_2^{V_6}\left(\frac{\bar{B}+V_5-\frac{1}{2}(\bar{U}(D_0^{V_6})+V_1)^2-\frac{1}{2}(V_2^2+V_3^2)}{\gamma(\bar{K}+V_4)(\bar{U}(D_0^{V_6})+V_1)}\right)D_3^{V_6} V_4\\\no
&&\quad- D_3^{V_6}\left(\frac{\bar{B}+V_5-\frac{1}{2}(\bar{U}(D_0^{V_6})+V_1)^2-\frac{1}{2}(V_2^2+V_3^2)}{\gamma(\bar{K}+V_4)(\bar{U}(D_0^{V_6})+V_1)}\right)D_2^{V_6} V_4.
\ee
Then \eqref{shock17} gives the boundary data for $\tilde{\omega}_1$ at $y_1=r_s$
\be\label{vor401}
\tilde{\omega}_1(r_s,y')=\frac{1}{a_0}(\p_{y_3} g_2-\frac{1}{r_s}\p_{y_2}g_3)({\bf V}(r_s,y'), V_6)+g_4({\bf V}(r_s,y'), V_6),
\ee
with
\be\label{g4}
&&g_4({\bf V}, V_6)=\frac{(y_1-r_2)V_6\p_{y_2}V_3(r_s,y')}{y_1((r_2-r_s)y_1+(r_2-y_1)V_6)}-\frac{(y_1-r_2)\p_{y_3}V_6\p_{y_1}V_2(r_s,y')}{r_2-r_s-V_6}\\\no
&&\quad\quad\quad +\frac{r_2-r_s}{(r_2-r_s)y_1+(r_2-y_1)V_6}\frac{(y_1-r_2)\p_{y_2}V_6 \p_{y_1}V_3(r_s,y')}{r_2-r_s-V_6}.
\ee
On the other hand, \eqref{vor22} implies that
\be\label{vor402}
&&\tilde{\omega}_2=D_3^{V_6} V_1- D_1^{V_6} V_3\\\no
&&=\frac{V_2\tilde{\omega}_1+D_3^{V_6} V_5}{\bar{U}(D_0^{V_6})+V_1}-\frac{\bar{B}-\frac{1}{2}\bar{U}^2(D_0^{V_6})+V_5-\bar{U}(D_0^{V_6})V_1-\frac{1}{2}\sum_{j=1}^3 V_j^2}{\gamma (\bar{U}(D_0^{V_6})+V_1)(\bar{K}+V_4)} D_3^{V_6} V_4,\\\label{vor403}
&&\tilde{\omega}_3=D_1^{V_6} V_2-D_2^{V_6} V_1+ \frac{V_2}{D_0^{V_6}}\\\no
&&=\frac{V_3\tilde{\omega}_1-D_2^{V_6} V_5}{\bar{U}(D_0^{V_6})+V_1}+\frac{\bar{B}-\frac{1}{2}\bar{U}^2(D_0^{V_6})+V_5-\bar{U}(D_0^{V_6})V_1-\frac{1}{2}\sum_{j=1}^3 V_j^2}{\gamma (\bar{U}(D_0^{V_6})+V_1)(\bar{K}+V_4)} D_2^{V_6} V_4.
\ee

Collecting the principal terms and putting the quadratic terms on the right hand sides, one gets from direct computations and \eqref{vor402}-\eqref{vor403} that
\be\label{vor404}
&&\frac{1}{y_1}\p_{y_2} V_3-\p_{y_3} V_2=\tilde{\omega}_1(y)+H_1({\bf V},V_6),\\\label{vor405}
&&\p_{y_3} V_1- \p_{y_1} V_3 + \frac{\bar{B}-\frac12 \bar{U}^2(y_1)}{\gamma \bar{K} \bar{U}(y_1)} \p_{y_3} V_4=\frac{V_2\tilde{\omega}_1+D_3^{V_6} V_5}{\bar{U}(D_0^{V_6})+V_1}+ H_2({\bf V},V_6),\\\label{vor406}
&&\p_{y_1} V_2+\frac{V_2}{y_1}-\frac{1}{y_1} \p_{y_2} V_1 - \frac{\bar{B}-\frac12 \bar{U}^2(y_1)}{\gamma \bar{K} \bar{U}(y_1)} \frac{1}{y_1}\p_{y_2} V_4=\frac{V_3\tilde{\omega}_1-D_2^{V_6} V_5}{\bar{U}(D_0^{V_6})+V_1}+H_3,
\ee
where
\be\no
&&H_1({\bf V},V_6)=\frac{(y_1-r_2)\p_{y_3}V_6 \p_{y_1}V_2}{(r_2-r_s-V_6)}-\frac{(y_1-r_2)V_6\p_{y_2} V_3}{y_1((r_2-r_s)y_1+(r_2-y_1)V_6)}\\\no
&&\quad\quad-\frac{r_2-r_s}{(r_2-r_s)y_1+(r_2-y_1)V_6}\frac{(y_1-r_2)\p_{y_2}V_6 \p_{y_1}V_3(y_1,y')}{r_2-r_s-V_6},\\\no
&&H_2({\bf V},V_6)=-\frac{(y_1-r_2)\p_{y_3} V_6}{r_2-r_s-V_6}\p_{y_1} V_1 + \frac{V_6\p_{y_1}V_3}{r_2-r_s-V_6}-\frac{V_5-\bar{U}(D_0^{V_6})V_1-\frac{1}{2}\sum_{j=1}^3 V_j^2}{\gamma (\bar{U}(D_0^{V_6})+V_1)(\bar{K}+V_4)} D_3^{V_6} V_4\\\no
&&\quad\quad-\left(\frac{\bar{B}-\frac{1}{2}\bar{U}^2(D_0^{V_6})}{\gamma (\bar{U}(D_0^{V_6})+V_1)(\bar{K}+V_4)}-\frac{\bar{B}-\frac12 \bar{U}^2(y_1)}{\gamma \bar{K} \bar{U}(y_1)}\right)\p_{y_3} V_4\\\no
&&\quad\quad-\frac{\bar{B}-\frac12 \bar{U}^2(D_0^{V_6})}{\gamma (\bar{K}+V_4)(\bar{U}(D_0^{V_6})+V_1)}\frac{(y_1-r_2)\p_{y_3}V_6}{r_2-r_s-V_6}\p_{y_1} V_4,\\\no
&&H_3({\bf V},V_6)=-\frac{V_6}{r_2-r_s-V_6}\p_{y_1} V_2+ \frac{(r_2-y_1)V_6(V_2+\p_{y_2} V_1)}{y_1((r_2-r_s)y_1+(r_2-y_1)V_6)}\\\no
&&\quad\quad+\frac{1}{D_0^{V_6}}\frac{(y_1-r_2)\p_{y_2}V_6\p_{y_1}V_1}{r_2-r_s-V_6}+ \frac{V_5-\bar{U}(D_0^{V_6})V_1-\frac{1}{2}\sum_{j=1}^3 V_j^2}{\gamma (\bar{U}(D_0^{V_6})+V_1)(\bar{K}+V_4)} D_2^{V_6} V_4\\\no
&&\quad\quad+ \left(\frac{\bar{B}-\frac{1}{2}\bar{U}^2(D_0^{V_6})}{\gamma (\bar{U}(D_0^{V_6})+V_1)(\bar{K}+V_4)}D_2^{V_6} V_4-\frac{\bar{B}-\frac12 \bar{U}^2(y_1)}{\gamma \bar{K} \bar{U}(y_1)}\frac{1}{y_1}\p_{y_2} V_4\right).
\ee
The boundary conditions on $\Sigma_2^{\pm}$ and $\Sigma_3^{\pm}$, \eqref{slip1}, become
\be\label{slip2}\begin{cases}
V_2(y_1,\pm\theta_0, y_3)=0,\ \ &\text{on }\Sigma_2^{\pm},\\
V_3(y_1,y_2,\pm 1)=0,\ \ &\text{on }\Sigma_3^{\pm}.
\end{cases}\ee

Furthermore, the equation \eqref{den15} can be rewritten as
\be\label{den20}
d_1\p_{y_1} V_1+ \frac{1}{y_1} \p_{y_2} V_2 + \p_{y_3} V_3 +\frac{V_1}{y_1}+ d_2V_1=-\frac{(\gamma-1)(\bar{U}'+\frac{\bar{U}}{y_1})}{c^2(\bar{\rho},\bar{K})} V_5+ G_0,
\ee
with
\be\no
&&d_1(y_1)=1-\bar{M}^2(y_1),\ \ d_2(y_1)=\frac{\bar{M}^2(2+(\gamma-1)\bar{M}^2(y_1))}{y_1(1-\bar{M}^2(y_1))},\\\no
&&G_0({\bf V},V_6)=\mathbb{F}({\bf V},V_6)-\left(d_1(D_0^{V_6})D_1^{V_6} V_1-d_1(y_1)\p_{y_1} V_1\right)-(D_2^{V_6} V_2-\frac{1}{y_1}\p_{y_2}V_2)\\\no
&&\quad-(D_3^{V_6}V_3-\p_{y_3} V_3)- \left((\frac{1}{D_0^{V_6}} +d_2(D_0^{V_6}))V_1-(\frac{1}{y_1}+d_2(y_1)) V_1\right)
\ee
and
\be\no
&&\mathbb{F}({\bf V},V_6)=-\frac{(\gamma-1)(D_1^{V_6} V_1+\frac{V_1}{D_0^{V_6}})}{\bar{c}^2(D_0^{V_6})} V_5+ \frac{\bar{U}'(D_0^{V_6})+D_1^{V_6} V_1}{\bar{c}^2(D_0^{V_6})}\left(\frac{\gamma+1}{2} V_1^2+\frac{\gamma-1}{2}(V_2^2+V_3^2)\right)\\\no
&&\quad\quad+\frac{(\gamma-1)(\bar{U}(D_0^{V_6})+V_1)}{2 D_0^{V_6} \bar{c}^2(D_0^{V_6})}\sum_{j=1}^3 V_j^2+\frac{(\gamma+1)\bar{U}(D_0^{V_6})V_1 D_1^{V_6} V_1+ (\gamma-1)\bar{U}(D_0^{V_6}) \frac{V_1^2}{D_0^{V_6}}}{\bar{c}^2(D_0^{V_6})}\\\no
&&\quad\quad- \frac{1}{\bar{c}^2(D_0^{V_6})}\left((\gamma-1)V_5-\frac{\gamma-1}{2}\sum_{j=1}^3 V_j^2-(\gamma-1)\bar{U}(D_0^{V_6})V_1\right)(D_2^{V_6} V_2+D_3^{V_6} V_3)\\\no
&&\quad\quad
+\frac{1}{\bar{c}^2(D_0^{V_6})}(V_2^2D_2^{V_6} V_2 +V_3^2 D_3^{V_6}V_3)+ \frac{\bar{U}(D_0^{V_6})+V_1}{\bar{c}^2(D_0^{V_6})}(V_2D_1^{V_6} V_2+V_3D_1^{V_6} V_3)\\\no
&&\quad\quad+\frac{V_2}{\bar{c}^2(D_0^{V_6})}((\bar{U}(D_0^{V_6})+V_1)D_2^{V_6} V_1 +V_3D_2^{V_6} V_3)+ \frac{V_3}{\bar{c}^2(D_0^{V_6})}((\bar{U}(D_0^{V_6})+V_1)D_3^{V_6} V_1 +V_2D_3^{V_6} V_2).
\ee

Finally, the boundary condition \eqref{pres2} at the exit becomes
\be\label{pres3}
&&V_1(r_2,y')+\frac{\bar{B}-\frac{1}{2}\bar{U}^2(r_2)}{\gamma \bar{K}\bar{U}(r_2)}V_4(r_2,y')=\frac{V_5(r_2,y')}{\bar{U}(r_2)}-\frac{\epsilon P_{ex}(y')}{(\bar{\rho}\bar{U})(r_2)}\\\no
&&\quad\quad\quad-\frac{1}{2\bar{U}(r_2)}\sum_{j=1}^3 V_j^2(r_2,y')-\frac{1}{\bar{U}(r_2)} E({\bf V}(r_2,y')),
\ee
where
\be\no
&&E({\bf V}(r_2,y'))=\frac{\gamma}{\gamma-1}(\bar{K}+V_4(r_2,y'))^{\frac1{\gamma}}(\tilde{P}({\bf V})(r_2,y'))^{\frac{\gamma-1}{\gamma}}-\frac{\gamma}{\gamma-1}\bar{K}^{\frac1{\gamma}}\bar{P}^{\frac{\gamma-1}{\gamma}}\\\label{err3}
&&\quad\quad-\frac{1}{\bar{\rho}(r_2)}(\tilde{P}({\bf V})(r_2,y')- \bar{P}(r_2))-\frac{\bar{B}-\frac{1}{2}\bar{U}^2(r_2)}{\gamma \bar{K}} V_4(r_2,y').
\ee

Therefore after the coordinates transformation \eqref{coor}, the transonic shock problem \eqref{euler-cyl} with \eqref{super1}-\eqref{pressure}, and \eqref{rh} is equivalent to solve the following problem:

{\bf Problem TS.} Find a function $V_6$ defined on $E$ and vector functions $(V_1,\cdots, V_5)$ defined on the $\mathbb{D}$, which solve the transport equations \eqref{ber31}-\eqref{ent31}, \eqref{vor400},\eqref{vor404}-\eqref{vor406} and \eqref{den20} with boundary conditions \eqref{shock400}, \eqref{shock19}-\eqref{shock21}, \eqref{vor401},\eqref{slip2} and \eqref{pres3}.

Theorem \ref{existence} then follows directly from the following result.
\begin{theorem}\label{main}
{\it Assume that the compatibility conditions \eqref{pressure-cp} and \eqref{super3} hold. There exists a small constant $\epsilon_0>0$ depending only on the background solution $\overline{\bm{\Psi}}$ and the boundary data $U_{1,0}^-,U_{2,0}^-,U_{3,0}^-$, $P_0^-, K_0^-$, $P_{ex}$ such that if $0\leq \epsilon<\epsilon_0$, the problem \eqref{ber31}-\eqref{ent31},\eqref{vor400},\eqref{vor404}-\eqref{vor406},\eqref{den20} with boundary conditions \eqref{shock400}, \eqref{shock19}-\eqref{shock21},\eqref{vor401},\eqref{slip2} and \eqref{pres3} has a unique solution $(V_1,V_2,V_3,V_4,V_5)(y)$ with the shock front $\mathcal{S}: y_1=V_6(y')$ satisfying the following properties.
\begin{enumerate}[(1)]
  \item The function $V_6(y')\in C^{3,\alpha}(\overline{E})$ satisfies
  \be\no
  \|V_6(y')\|_{C^{3,\alpha}(\overline{E})}\leq C_*\epsilon,
  \ee
  and
  \be\label{shock91}\begin{cases}
  \p_{y_2}V_6(\pm\theta_0, y_3)=\p_{y_2}^3V_6(\pm\theta_0, y_3)=0,\ \ \forall y_3\in [-1,1],\\
  \p_{y_3}V_6(y_2,\pm 1)=\p_{y_3}^3V_6(y_2, \pm 1)=0,\ \ \forall y_2\in [-\theta_0,\theta_0],
  \end{cases}\ee
  where $C_*$ depends only on the background solution and the supersonic incoming flow and the exit pressure.
  \item The solution $(V_1,V_2,V_3,V_4,V_5)(y)\in C^{2,\alpha}(\overline{\mathbb{D}})$ satisfies
  \be\no
  \sum_{j=1}^5\|V_j\|_{C^{2,\alpha}(\overline{\mathbb{D}})}\leq C_*\epsilon
  \ee
  and the compatibility conditions
  \be\label{sub52}
  \begin{cases}
  (V_2,\p_{y_2}^2V_2)(y_1,\pm\theta_0,y_3)=\p_{y_2}(V_1,V_3,V_4,V_5))(y_1,\pm\theta_0,y_3)=0, \text{on }\Sigma_2^{\pm},\\
  (V_3,\p_{y_3}^2V_3)(y_1,y_2,\pm 1)=\p_{y_3}(V_1,V_2,V_4,V_5))(y_1,y_2,\pm 1)=0,\ \text{on }\Sigma_3^{\pm}.
  \end{cases}\ee
\end{enumerate}
}\end{theorem}

\section{Proof of Theorem \ref{main}}\label{proof}

We proceed to prove Theorem \ref{main}. The solution class $\mathcal{X}$ consists of the vector functions $(V_1,\cdots, V_5,V_6)\in (C^{2,\alpha}(\overline{\mathbb{D}}))^5\times C^{3,\alpha}(\overline{E})$ satisfying the estimate
\be\no
\|({\bf V},V_6)\|_{\mathcal{X}}:=\sum_{j=1}^5 \|V_j\|_{C^{2,\alpha}(\overline{\mathbb{D}})}+\|V_6\|_{C^{3,\alpha}(\overline{E})}\leq \delta_0
\ee
and the following compatibility conditions (which is precisely \eqref{shock91} and \eqref{sub52})
\be\label{class2}\begin{cases}
(V_2,\p_{y_2}^2V_2,\p_{y_2}(V_1,V_3,V_4,V_5))(y_1,\pm\theta_0,y_3)=0,\ \ &\text{on }\Sigma_2^{\pm},\\
(V_3,\p_{y_3}^2V_3,\p_{y_3}(V_1,V_2,V_4,V_5))(y_1,y_2,\pm 1)=0,\ \ &\text{on }\Sigma_3^{\pm},\\
(\p_{y_2}V_6,\p_{y_2}^3V_6)(\pm\theta_0,y_3)=0,\ \ &\text{on }y_3\in [-1,1],\\
(\p_{y_3}V_6,\p_{y_3}^3V_6)(y_2,\pm 1)=0,\ \  &\text{on } y_2\in [-\theta_0,\theta_0]
\end{cases}\ee
with $\delta_0$ being a suitably small positive constant to be determined later.

For any given $(\hat{{\bf V}},\hat{V}_6)\in \mathcal{X}$, we will define an operator $\mathcal{T}$ mapping $\mathcal{X}$ to itself, and the unique fixed point of $\mathcal{T}$ will solve the {\bf Problem TS}. Note that the principal terms in the deformation-curl system \eqref{vor404}-\eqref{vor406}, \eqref{den20} contain hyperbolic quantities $V_4$ and $V_5$, so the problem is still hyperbolic-elliptic coupled. Since $V_4$ and $V_5$ are conserved along the trajectory, one can derive some representation formulas for $V_4$ and $V_5$. Indeed, the boundary data for $V_4$ and $V_5$ involves the shock position $V_6$, using the condition \eqref{shock400}, one can check that the principal term in $V_4(y)$ is given by a scalar multiple of $V_1(r_s,y')$ and $V_5$ can be regarded as a higher order term. Substituting these into \eqref{vor404}, \eqref{vor405}, \eqref{vor406}, \eqref{den20} and \eqref{pres3}, we derive a deformation-curl first order elliptic system for $(V_1, V_2,V_3)$ containing a nonlocal term $V_1(r_s,y')$ with some boundary conditions involving only $(V_1, V_2,V_3)$ from which $(V_1, V_2,V_3)$ can be solved uniquely. Then the entropy $V_4$ and the shock front $V_6$ are uniquely determined. Now we give a detailed derivation of this procedure.

{\bf Step 1.} The shock front is uniquely determined by the following algebraic equation
\be\label{shock41}
V_6(y')=\frac{1}{a_1} V_1(r_s,y')-\frac{R_1(\hat{{\bf V}}(r_s,y'),\hat{V}_6)}{a_1},
\ee
provided that $V_1(r_s,y')$ is obtained.

{\bf Step 2.} We solve the transport equations for the Bernoulli's quantity and the entropy respectively. The Bernoulli's quantity will be determined by (See \eqref{ber31})
\be\lab{ber41}\begin{cases}
\bigg(D_1^{\hat{V}_6}+\frac{\hat{V}_2}{\bar{U}(D_0^{\hat{V}_6})+\hat{V}_1}D_2^{\hat{V}_6}+\frac{\hat{V}_3}{\bar{U}(D_0^{\hat{V}_6})+\hat{V}_1}D_3^{\hat{V}_6}\bigg)V_5=0,\\
V_5(r_s,y_2,y_3)=B^-(r_s+\hat{V}_6(y'),y')-\bar{B}^-.
\end{cases}\ee
Set
\be\label{char}\begin{cases}
K_2(y):=\frac{r_2-r_s-\hat{V}_6}{r_2-r_s}\frac{\hat{V}_2}{D_0^{\hat{V}_6}(\bar{U}(D_0^{\hat{V}_6})+\hat{V}_1)
+\frac{y_1-r_2}{r_2-r_s}(\hat{V}_2\p_{y_2}\hat{V}_6+D_0^{\hat{V}_6}\hat{V}_3\p_{y_3}\hat{V}_6)},\\
K_3(y):=\frac{r_2-r_s-\hat{V}_6}{r_2-r_s}\frac{D_0^{\hat{V}_6}\hat{V}_3}{\bar{U}(D_0^{\hat{V}_6})+\hat{V}_1
+\frac{y_1-r_2}{r_2-r_s}(\hat{V}_2\p_{y_2}\hat{V}_6+D_0^{\hat{V}_6}\hat{V}_3\p_{y_3}\hat{V}_6)}.
\end{cases}\ee
Then $K_2,K_3\in C^{2,\alpha}(\overline{\mathbb{D}})$ for any $(\hat{{\bf V}},\hat{V}_6)\in \mathcal{X}$. The function $V_5$ is conserved along the trajectory determined by the following ODE system
\begin{eqnarray}\label{char1} \left\{\begin{array}{ll}
\frac{d \bar{y}_2(\tau; y)}{d\tau}=K_2(\tau,\bar{y}_2(\tau;y),\bar{y}_3(\tau;y)),\ \ \forall \tau\in [r_s,r_2],\\
\frac{d \bar{y}_3(\tau; y)}{d\tau}=K_3(\tau,\bar{y}_2(\tau;y),\bar{y}_3(\tau;y)),\ \ \forall \tau\in [r_s,r_2],\\
\bar{y}_2(y_1; y)=y_2,\ \bar{y}_3(y_1;y)=y_3.
\end{array}\right. \end{eqnarray}
Denote $(\beta_2(y),\beta_3(y))=(\bar{y}_2(r_s;y),\bar{y}_3(r_s;y))$. Since $(\hat{{\bf V}},\hat{V}_6)\in \mathcal{X}$ satisfies the compatibility conditions \eqref{class2}, then
\be\label{char201}\begin{cases}
K_2(y_1,\pm\theta_0,y_3)=\p_{y_2} K_3(y_1,\pm\theta_0,y_3)=0,\ \ &\text{on }\Sigma_2^{\pm},\\
K_3(y_1,y_2,\pm 1)=\p_{y_3} K_2(y_1,y_2, \pm 1)=0,\ \ &\text{on }\Sigma_3^{\pm}.
\end{cases}\ee
According to the uniqueness of the solution to \eqref{char1} and \eqref{char201}, there hold
\be\label{char20}\begin{cases}
\bar{y}_2(\tau;y_1,\pm\theta_0,y_3)=\pm\theta_0,\ \ \forall \tau\in [r_s,r_2], (y_1,y_3)\in \Sigma_2^{\pm},\\
\bar{y}_3(\tau;y_1,y_2,\pm 1)=\pm 1,\ \ \forall \tau\in [r_s,r_2], (y_1,y_2)\in \Sigma_3^{\pm}
\end{cases}\ee
and
\be\label{char2}\begin{cases}
\beta_2(y_1,\pm\theta_0,y_3)=\pm\theta_0,\ \ \forall  (y_1,y_3)\in \Sigma_2^{\pm},\\
\beta_3(y_1,y_2,\pm 1)=\pm 1,\ \ \forall  (y_1,y_2)\in \Sigma_3^{\pm}.
\end{cases}\ee
The existence and uniqueness of $(\bar{y}_2(\tau;y),\bar{y}_3(\tau;y))$ on the whole interval $[r_s,r_2]$ follow from the standard theory of systems of ordinary differential equations and \eqref{char20}. Then it follows from \eqref{char1} that
\be\no
&&y_2-\beta_2(y)=\int_{r_s}^{y_1} K_2(\tau,\bar{y}_2(\tau;y),\bar{y}_3(\tau;y)) d\tau,\\\no
&&\delta_{2j}-\p_{y_j}\beta_2(y)=\delta_{1j}K_2(y)+\int_{r_s}^{y_1} \p_{y_2}K_2\p_{y_j}\bar{y}_2(\tau;y)+ \p_{y_3}K_2\p_{y_j}\bar{y}_3(\tau;y) d\tau,\ \ j=1,2,3,\\\no
&&-\p_{y_iy_j}^2\beta_2(y)=\delta_{1j}\p_{y_i}K_2(y)+\int_{r_s}^{y_1} \p_{y_2}^2K_2\p_{y_j}\bar{y}_2 \p_{y_i}\bar{y}_2 + \p_{y_2y_3}^2 K_2(\p_{y_j}\bar{y}_2\p_{y_i}\bar{y}_3+\p_{y_i}\bar{y}_2\p_{y_j}\bar{y}_3)\\\no
&&\quad\quad +\p_{y_3}^2K_2\p_{y_j}\bar{y}_3 \p_{y_i}\bar{y}_3+ \p_{y_2} K_2 \p_{y_i y_j}^2 \bar{y}_2(\tau;y)+\p_{y_3} K_2 \p_{y_i y_j}^2 \bar{y}_3(\tau;y)d\tau.
\ee
Thus there holds
\be\no
\sum_{j=2}^3\|\beta_j(y)-y_j\|_{C^{2,\alpha}(\overline{\mathbb{D}})}\leq C_*\|(\hat{{\bf V}},\hat{V}_6)\|_{\mathcal{X}}.
\ee

Differentiating the second equation in \eqref{char1} with respect to $y_2$ and restricting the resulting equation on $y_2=\pm \theta_0$, one obtains by \eqref{char201} that
\be\no\begin{cases}
\frac{d}{d\tau}\p_{y_2} \bar{y}_3(\tau;y_1,\pm\theta_0,y_3)=\p_{y_3} K_3(\tau,\bar{y}_2(\tau;y),\bar{y}_3(\tau;y))\p_{y_2}\bar{y}_3(\tau;y_1,\pm\theta_0,y_3),\\
\p_{y_2}\bar{y}_3(y_1;\pm\theta_0,y_3)=0,
\end{cases}\ee
from which one concludes that $\p_{y_2}\bar{y}_3(\tau;y_1,\pm\theta_0,y_3)=0$ for any $\tau\in [r_s,r_2]$. Similarly, one has $\p_{y_3}\bar{y}_2(\tau;y_1,y_2,\pm 1)=0$ for any $\tau\in [r_s,r_2]$. Thus
\be\label{char3}\begin{cases}
\p_{y_2}\beta_3(y_1,\pm\theta_0, y_3)=0,\ \ &\text{on }\Sigma_2^{\pm},\\
\p_{y_3}\beta_2(y_1,y_2,\pm 1)=0,\ \ &\text{on }\Sigma_3^{\pm}.
\end{cases}\ee

Since $V_5$ is conserved along the trajectory, one has
\be\no
V_5(y)=V_5(r_s,\beta_2(y),\beta_3(y))=B^-(r_s+\hat{V}_6(\beta_2(y),\beta_3(y)),\beta_2(y),\beta_3(y))-\bar{B}^-.
\ee
Thus $V_5$ can be regarded as a high order term satisfying the following estimate
\be\label{ber43}
&&\|V_5\|_{C^{2,\alpha}(\overline{\mathbb{D}})}\leq C_*\epsilon (\|\hat{V}_6\|_{C^{2,\alpha}(\overline{E})}+\sum_{j=2}^3\|\beta_j\|_{C^{2,\alpha}(\overline{\mathbb{D}})})\\\no
&&\leq C_*(\epsilon+\epsilon\|(\hat{{\bf V}}, \hat{V}_6)\|_{\mathcal{X}})\leq C_*(\epsilon+\epsilon \delta_0).
\ee
It follows from \eqref{super5},\eqref{class2} and \eqref{char3} that the following compatibility conditions hold
\be\label{ber44}\begin{cases}
\p_{y_2} V_5(y_1,\pm \theta_0, y_3)=0,\ \ \text{on }\Sigma_2^{\pm},\\
\p_{y_3} V_5(y_1,y_2,\pm 1)=0,\ \ \text{on }\Sigma_3^{\pm}.
\end{cases}\ee

The function $V_4$ satisfies
\be\no\begin{cases}
\b(D_1^{\hat{V}_6}+\frac{\hat{V}_2}{\bar{U}(D_0^{\hat{V}_6})+\hat{V}_1}D_2^{\hat{V}_6}+\frac{\hat{V}_3}{\bar{U}(D_0^{\hat{V}_6})+\hat{V}_1}D_3^{\hat{V}_6}\b)V_4=0,\\
V_4(r_s,y')= a_2 V_6(y')+R_2(\hat{{\bf V}}(r_s,y'),\hat{V}_6(y')).
\end{cases}\ee

By the characteristic method and the equation \eqref{shock41}, one has
\be\no
&&V_4(y)=V_4(r_s,\beta_2(y),\beta_3(y))\\\label{ent42}
&&=a_2 V_6(\beta_2(y),\beta_3(y))+ R_2(\hat{{\bf V}}(r_s,\beta_2(y),\beta_3(y)),\hat{V}_6(\beta_2(y),\beta_3(y)))\\\no
&&=a_2 V_6(y')+ a_2(V_6(\beta_2(y),\beta_3(y))-V_6(y'))+ R_2(\hat{{\bf V}}(r_s,\beta_2(y),\beta_3(y)),\hat{V}_6(\beta_2(y),\beta_3(y)))\\\no
&&=\frac{a_2}{a_1} V_1(r_s,y')+a_2(V_6(\beta_2(y),\beta_3(y))-V_6(y'))+R_3(\hat{{\bf V}}(r_s,\beta_2(y),\beta_3(y)),\hat{V}_6(\beta_2(y),\beta_3(y))),
\ee
where
\be\no
R_3(\hat{{\bf V}}(r_s,y'),\hat{V}_6)=R_2(\hat{{\bf V}}(r_s,y'),\hat{V}_6)-\frac{a_2 R_1(\hat{{\bf V}}(r_s,y'),\hat{V}_6)}{a_1}
\ee
Since $V_6(y')$ is still unknown, one may rewrite \eqref{ent42} as
\be\label{ent43}
V_4(y_1,y')=\frac{a_2}{a_1} V_1(r_s,y')+R_4(\hat{{\bf V}}(r_s,\beta_2(y),\beta_3(y)),\hat{V}_6(\beta_2(y),\beta_3(y))),
\ee
with
\be\no
R_4=a_2(\hat{V}_6(\beta_2(y),\beta_3(y))-\hat{V}_6(y'))+R_3(\hat{{\bf V}}(r_s,\beta_2(y),\beta_3(y)),\hat{V}_6(\beta_2(y),\beta_3(y))).
\ee
Therefore $V_4$ is decomposed as a scalar multiple of $V_1(r_s,y')$ with high order terms satisfying
\be\no
&&\|V_4\|_{C^{2,\alpha}(\overline{\mathbb{D}})}\leq C_*\|V_1(r_s,\cdot)\|_{C^{2,\alpha}(\overline{E})}+\|R_4\|_{C^{2,\alpha}(\overline{\mathbb{D}})}\\\no
&&\leq C_*(\|V_1(r_s,\cdot)\|_{C^{2,\alpha}(\overline{E})}+\|\hat{V}_6\|_{C^{3,\alpha}(\overline{E})}\sum_{j=2}^3\|\beta_j(y)-y_j\|_{C^{2,\alpha}(\overline{\mathbb{D}})})\\\no
&&\q+ C_*(\epsilon \|(\hat{{\bf V}}, \hat{V}_6)\|_{\mathcal{X}}+\|(\hat{{\bf V}}, \hat{V}_6)\|_{\mathcal{X}}^2)\leq C_*\|V_1(r_s,\cdot)\|_{C^{2,\alpha}(\overline{E})}+ C_*(\epsilon \delta_0+ \delta_0^2).
\ee
Furthermore, since $(\hat{{\bf V}}, \hat{V}_6)\in\mathcal{X}$ satisfies the compatibility conditions \eqref{class2} and the upcoming supersonic flow satisfies \eqref{super5}, using the formulas \eqref{j21}-\eqref{rv03}, one could verify by direct but tedious computations that
\be\label{j211}\begin{cases}
J_2(\hat{{\bf V}}(r_s,y'),\hat{V}_6)\}|_{y_2=\pm\theta_0}=\p_{y_2}^2\{J_2(\hat{{\bf V}}(r_s,y'),\hat{V}_6)\}|_{y_2=\pm\theta_0}=0,\ \ \\
\p_{y_2}\{J_3(\hat{{\bf V}}(r_s,y'),\hat{V}_6)\}|_{y_2=\pm\theta_0}=  \p_{y_2}\{J(\hat{{\bf V}}(r_s,y'),\hat{V}_6)\}|_{y_2=\pm\theta_0}=0,\ \\
J_3(\hat{{\bf V}}(r_s,y'),\hat{V}_6)\}|_{y_1=\pm 1}=\p_{y_3}^2\{J_3(\hat{{\bf V}}(r_s,y'),\hat{V}_6)\}|_{y_3=\pm 1}=0,\ \ \\
\p_{y_3}\{J_2(\hat{{\bf V}}(r_s,y'),\hat{V}_6)\}|_{y_3=\pm 1}=  \p_{y_3}\{J(\hat{{\bf V}}(r_s,y'),\hat{V}_6)\}|_{y_3=\pm 1}=0,\
\end{cases}\ee
and for all $j=1,2,3$
\be\label{rv11}\begin{cases}
\p_{y_2}\{ R_{0j}(\hat{{\bf V}}(r_s,y'),\hat{V}_6)\}|_{y_2=\pm\theta_0}=0,\ \ \forall y_3\in [-1,1],\\
\p_{y_3}\{ R_{0j}(\hat{{\bf V}}(r_s,y'),\hat{V}_6)\}|_{y_3=\pm 1}=0,\ \ \forall y_2\in [-\theta_0,\theta_0].
\end{cases}\ee
Thus for $k=1,2,3$
\be\label{rv41}\begin{cases}
\p_{y_2}\{ R_{k}(\hat{{\bf V}}(r_s,y'),\hat{V}_6)\}|_{y_2=\pm\theta_0}=0,\ \ \forall y_3\in [-1,1],\\
\p_{y_3}\{ R_{k}(\hat{{\bf V}}(r_s,y'),\hat{V}_6)\}|_{y_3=\pm 1}=0,\ \ \forall y_2\in [-\theta_0,\theta_0].
\end{cases}\ee
These, together with \eqref{char2} and \eqref{char3}, imply that
\be\label{rv42}
\begin{cases}
\p_{y_2}\{R_4(\hat{{\bf V}}(r_s,\beta_2(y),\beta_3(y)),\hat{V}_6(\beta_2(y),\beta_3(y)))\}(y_1,\pm \theta_0, y_3)=0,\ \ \text{on }\Sigma_2^{\pm},\\
\p_{y_3}\{R_4(\hat{{\bf V}}(r_s,\beta_2(y),\beta_3(y)),\hat{V}_6(\beta_2(y),\beta_3(y)))\}(y_1,y_2,\pm 1)=0,\ \ \text{on }\Sigma_3^{\pm}.
\end{cases}\ee
and
\be\no\begin{cases}
\p_{y_2} V_4(y_1,\pm \theta_0, y_3)=\frac{a_2}{a_1}\p_{y_2} V_1(r_s,\pm\theta_0,y_3),\ \ \text{on }\Sigma_2^{\pm},\\
\p_{y_3} V_4(y_1,y_2,\pm 1)=\frac{a_2}{a_1}\p_{y_3} V_1(r_s,y_2,\pm 1),\ \ \text{on }\Sigma_3^{\pm}.
\end{cases}\ee

\begin{remark}
{\it In the $C^{2,\alpha}(\overline{\mathbb{D}})$ estimate of the first term in $R_4$, the $C^{3,\alpha}(\overline{E})$ norm of $\hat{V}_6$ is required.
}\end{remark}

{\bf Step 3.} We solve the transport equation for the first component of the vorticity. Due to \eqref{vor400} and \eqref{vor401}, it suffices to consider the following problem
\be\label{vor501}\begin{cases}
\bigg(D_1^{\hat{V}_6}+\displaystyle\sum_{j=2}^3\frac{\hat{V}_j D_j^{\hat{V}_6}}{\bar{U}(D_0^{\hat{V}_6})+\hat{V}_1}\bigg) \tilde{{\omega}}_1 + \mu(\hat{{\bf V}},\hat{V}_6)\tilde{\omega}_1=H_0(\hat{{\bf V}},\hat{V}_6),\\
\tilde{\omega}_1(r_s,y')=R_6(\hat{{\bf V}}(r_s,y'), \hat{V}_6(y')),
\end{cases}\ee
where
\be\no
&&R_6(\hat{{\bf V}}(r_s,y'), \hat{V}_6(y'))=\frac{1}{a_0}(\p_{y_3} \{g_2(\hat{{\bf V}}(r_s,y'), \hat{V}_6(y'))\}-\frac{1}{r_s}\p_{y_2}\{g_3(\hat{{\bf V}}(r_s,y'), \hat{V}_6(y'))\})\\\no
&&\quad\quad\quad+g_4(\hat{{\bf V}}(r_s,y'), \hat{V}_6(y')).
\ee
Since $(\hat{{\bf V}},\hat{V}_6)\in \mathcal{X}$ satisfies the compatibility conditions \eqref{class2}, using the first formula in \eqref{g21},\eqref{g31} and \eqref{g4}, one can verify that
\be\label{vor5011}\begin{cases}
\tilde{\omega}_1(r_s,\pm \theta_0, y_3)=0, & \forall y_3\in [-1,1],\\
\tilde{\omega}_1(r_s,y_2,\pm 1)=0,& \forall y_2\in [-\theta_0,\theta_0],\\
H_0(\hat{{\bf V}},\hat{V}_6)(y_1,\pm\theta_0, y_3)=0,\ \ &\text{on }\Sigma_2^{\pm},\\
H_0(\hat{{\bf V}},\hat{V}_6)(y_1,y_2,\pm 1)=0,\ \ &\text{on }\Sigma_3^{\pm}.
\end{cases}\ee

Integrating the equation in \eqref{vor501} along the trajectory $(\tau,\bar{y}_2(\tau;y),\bar{y}_3(\tau;y))$ yields
\begin{eqnarray}\label{vor502}
&&\tilde{\omega}_1(y)= R_6(\beta_2(y),\beta_3(y)) e^{-\int_{r_s}^{y_1} \mu(\hat{{\bf V}},\hat{V}_6)(t;\bar{y}_2(t;y),\bar{y}_3(t;y))dt}\\\nonumber
&&\quad\quad\quad+ \int_{r_s}^{y_1} H_0(\hat{{\bf V}},\hat{V}_6)(\tau,\bar{y}_2(\tau;y),\bar{y}_3(\tau;y)) e^{-\int_{\tau}^{y_1} \mu(\hat{{\bf V}},\hat{V}_6)(t;\bar{y}_2(t;y),\bar{y}_3(t;y))dt} d\tau.
\end{eqnarray}
Thus the following estimate holds
\be\no
&&\|\tilde{\omega}_1\|_{C^{1,\alpha}(\overline{\mathbb{D}})}\leq C_*(\|\tilde{\omega}_1(r_s,\cdot)\|_{C^{1,\alpha}(\overline{E})}+\|H_0(\hat{{\bf V}},\hat{V}_6)\|_{C^{1,\alpha}(\overline{\mathbb{D}})})\\\no
&&\leq C_*(\epsilon\|(\hat{{\bf V}}, \hat{V}_6)\|_{\mathcal{X}}+\|(\hat{{\bf V}}, \hat{V}_6)\|_{\mathcal{X}}^2)\leq C_*(\epsilon\delta_0+\delta_0^2).
\ee
Also \eqref{char2},\eqref{char3}, \eqref{vor5011} and \eqref{vor502} imply the following compatibility conditions
\be\label{vor5041}\begin{cases}
\tilde{\omega}_1(y_1,\pm\theta_0, y_3)=0,\ \ &\text{on }\Sigma_2^{\pm},\\
\tilde{\omega}_1(y_1,y_2, \pm 1)=0,\ \ &\text{on }\Sigma_3^{\pm}.
\end{cases}\ee

Substituting \eqref{vor502} and \eqref{ent43} into \eqref{vor404}-\eqref{vor406} yields
\be\label{vor504}
&&\frac{1}{y_1}\p_{y_2} V_3-\p_{y_3} V_2=G_1(\hat{{\bf V}},\hat{V}_6),\\\label{vor505}
&&\p_{y_3} V_1- \p_{y_1} V_3 + d_3(y_1)\p_{y_3} V_1(r_s,y')=G_2(V_5,\hat{{\bf V}},\hat{V}_6),\\\label{vor506}
&&\p_{y_1} V_2+\frac{V_2}{y_1}-\frac{1}{y_1} \p_{y_2} V_1 - d_3(y_1) \frac{1}{y_1}\p_{y_2}V_1(r_s,y')=G_3(V_5,\hat{{\bf V}},\hat{V}_6),
\ee
where
\be\no
&&d_3(y_1)=\frac{a_2}{a_1}\frac{\bar{B}-\frac12 \bar{U}^2(y_1)}{\gamma \bar{K} \bar{U}(y_1)},\\\no
&&G_1(\hat{{\bf V}},\hat{V}_6)=\tilde{\omega}_1(y)+H_1(\hat{{\bf V}},\hat{V}_6),\\\no
&&G_2(V_5,\hat{{\bf V}},\hat{V}_6)=\frac{\hat{V}_2\tilde{\omega}_1+D_3^{\hat{V}_6} V_5}{\bar{U}(D_0^{\hat{V}_6})+\hat{V}_1}+ H_2(\hat{{\bf V}},\hat{V}_6)+ \frac{\bar{B}-\frac12 \bar{U}^2(y_1)}{\gamma \bar{K} \bar{U}(y_1)} \p_{y_3} \{R_4(\hat{{\bf V}},\hat{V}_6)\},\\\no
&&G_3(V_5,\hat{{\bf V}},\hat{V}_6)=\frac{\hat{V}_3\tilde{\omega}_1-D_2^{\hat{V}_6} V_5}{\bar{U}(D_0^{\hat{V}_6})+\hat{V}_1}+H_3(\hat{{\bf V}},\hat{V}_6)-\frac{\bar{B}-\frac12 \bar{U}^2(y_1)}{\gamma \bar{K} \bar{U}(y_1)} \frac{1}{y_1}\p_{y_2}  \{R_4(\hat{{\bf V}},\hat{V}_6)\}.
\ee

Using \eqref{class2}, \eqref{ber44}, \eqref{vor5041} and \eqref{rv42}, one can further verify the following compatibility conditions:
\be\label{cp20}\begin{cases}
G_1(\hat{{\bf V}},\hat{V}_6)|_{y_2=\pm\theta_0}= G_3(V_5,\hat{{\bf V}},\hat{V}_6)|_{y_2=\pm\theta_0}=\p_{y_2}\{G_2(V_5,\hat{{\bf V}},\hat{V}_6)\}|_{y_2=\pm\theta_0}=0,\\
G_1(\hat{{\bf V}},\hat{V}_6)|_{y_3=\pm 1}= G_2(V_5,\hat{{\bf V}},\hat{V}_6)|_{y_3=\pm 1}=\p_{y_3}\{G_3(V_5,\hat{{\bf V}},\hat{V}_6)\}|_{y_3=\pm 1}=0,
\end{cases}\ee

Furthermore, \eqref{den20} implies that
\be\label{den251}
d_1(y_1)\p_{y_1} V_1+\frac{1}{y_1}\p_{y_2} V_2 +\p_{y_3} V_3+ \frac{V_1}{y_1}+ d_2 V_1=-\frac{(\gamma-1)(\bar{U}'+\frac{\bar{U}}{y_1})}{c^2(\bar{\rho},\bar{K})} V_5+G_0
\ee

It follows from \eqref{ent43} and \eqref{pres3} that
\be\label{pres4}
V_1(r_2,y')+ d_3(r_2) V_1(r_s,y')=q_4(y'),
\ee
where
\be\no
&&q_4(y')=d_3(r_2) R_4(\hat{{\bf V}}(r_s,\beta_2(r_2,y'),\beta_3(r_2,y')),\hat{V}_6(\beta_2(r_2,y'),\beta_3(r_2,y')))+
\frac{V_5(r_2,y')}{\bar{U}(r_2)}\\\no
&&\quad\quad-\frac{\epsilon P_{ex}(y')}{(\bar{\rho}\bar{U})(r_2)}-\frac{1}{2\bar{U}(r_2)}\sum_{j=1}^3 \hat{V}_j^2(r_2,y')-\frac{1}{\bar{U}(r_2)} E(\hat{{\bf V}}(r_2,y')).
\ee
And using \eqref{pressure-cp} and the explicit expression of $E(\hat{{\bf V}}(r_2,y'))$ in \eqref{err3}, one has
\be\no\begin{cases}
\p_{y_2} q_4(\pm\theta_0, y_3)=0,\ \ \forall y_3\in [-1,1],\\
\p_{y_3} q_4(y_2, \pm 1)=0,\ \ \forall y_2\in [-\theta_0,\theta_0].
\end{cases}\ee

{\bf Step 4.} We have derived a deformation-curl system for the velocity field which consists of the equations \eqref{den251}, \eqref{vor504}-\eqref{vor506} supplemented with the boundary conditions \eqref{shock19}, \eqref{slip2}, \eqref{pres4} and \eqref{shock20}-\eqref{shock21}, where $q_1$ and $q_i^{\pm} (i=2,3)$ are evaluated at $(\hat{{\bf V}},\hat{V}_6)$. However, due to the linearization, the vector field $(G_1,G_2,G_3)(\hat{{\bf V}},\hat{V}_6)$ may not be divergence free and thus the solvability condition of the curl system \eqref{vor504}-\eqref{vor506} does not hold in general. To overcome this obstacle, we first consider the following enlarged deformation-curl system, which includes an additional new unknown function $\Pi$ with homogeneous Dirichlet boundary condition for $\Pi$:
\be\label{den32}\begin{cases}
d_1\p_{y_1} V_1+\frac{1}{y_1}\p_{y_2} V_2 +\p_{y_3} V_3+ \frac{V_1}{y_1}+ d_2 V_1=-\frac{(\gamma-1)(\bar{U}'+\frac{\bar{U}}{y_1})}{c^2(\bar{\rho},\bar{K})} V_5+G_0(\hat{{\bf V}},\hat{V}_6),\text{in }\mathbb{D},\\
\frac{1}{y_1}\p_{y_2} V_3- \p_{y_3} V_2+\p_{y_1} \Pi=G_1(\hat{{\bf V}},\hat{V}_6),\ \text{in }\mathbb{D},\\
\p_{y_3} V_1-\p_{y_1} V_3+d_3(y_1)\p_{y_3} V_1(r_s,y')+ \frac1{y_1}\p_{y_2}\Pi= G_2(V_5,\hat{{\bf V}},\hat{V}_6),\ \ \text{in }\mathbb{D},\\
\p_{y_1} V_2-\frac{1}{y_1}\p_{y_2} V_1+\f{V_2}{y_1}-\frac{d_3(y_1)}{y_1}\p_{y_2} V_1(r_s,y')+\p_{y_3}\Pi = G_3(V_5,\hat{{\bf V}},\hat{V}_6),\ \ \text{in }\mathbb{D},\\
(\frac{1}{r_s^2}\p_{y_2}^2+\p_{y_3}^2) V_1(r_s,y')=a_0 a_1 (\frac{1}{r_s}\p_{y_2} V_2+\p_{y_3}V_3)(r_s,y')+q_1(\hat{{\bf V}}(r_s,y'),\hat{V}_6),\\
V_2(y_1,\pm\theta_0,y_3)=\Pi(y_1,\pm \theta_0, y_3)=0,\ \ \text{on }\Sigma_2^{\pm},\\
\Pi(r_s,y')=\Pi(r_2,y')=0,\ \ \forall y'\in E,\\
V_3(y_1,y_2,\pm 1)=\Pi(y_1,y_2,\pm 1)= 0,\ \ \text{on }\Sigma_3^{\pm}, \\
V_1(r_2,y')+ d_3(r_2) V_1(r_s,y')=q_4(y'),\ \ \forall y'\in E.
\end{cases}\ee

The system \eqref{den32} should be supplemented with the boundary conditions \eqref{shock20}-\eqref{shock21} where $q_i^{\pm} (i=2,3)$ are evaluated at $(\hat{{\bf V}},\hat{V}_6)$ so that a unique solvability result can be derived. Thanks to \eqref{j211}, \eqref{rv41}, \eqref{g21} and \eqref{g31}, it follows from \eqref{shock20}-\eqref{shock21} and the compatibility condition \eqref{class2} that on the intersection of the shock front with the nozzle wall:
\be\label{shock22}
&&\left(\frac{1}{r_s}\p_{y_2}V_1-a_0a_1V_2\right)(r_s,\pm \theta_0,y_3)=0,\ \forall y_3\in [-1,1],\\\label{shock23}
&&(\p_{y_3} V_1-a_0 a_1 V_3)(r_s,y_2,\pm 1)= 0,\ \ \forall y_2\in [-\theta_0,\theta_0].
\ee

The unique solvability of the problem \eqref{den32} with \eqref{shock22}-\eqref{shock23} can be verified by several steps using the Duhamel's principle as follows.

{\bf Step 4.1} First, taking the divergence operator for the second, third and fourth equations in \eqref{den32} leads to
\be\label{den33}\begin{cases}
\p_{y_1}^2 \Pi +\frac{1}{y_1}\p_{y_1}\Pi+ \frac{1}{y_1^2}\p_{y_2}^2\Pi + \p_{y_3}^2 \Pi= \p_{y_1} G_1 +\frac{G_1}{y_1} + \frac{1}{y_1}\p_{y_2} G_2 + \p_{y_3}G_3,\ \ \text{in }\mathbb{D},\\
\Pi(r_s,y')=\Pi(r_2,y')=0,\ \ \forall y'\in E,\\
\Pi(y_1,\pm \theta_0,y_3)=0,\ \ \text{on }\Sigma_2^{\pm},\\
\Pi(y_1,y_2,\pm 1)= 0\ \ \text{on }\Sigma_3^{\pm}.
\end{cases}\ee

The existence and uniqueness of $C^{2,\alpha}(\mathbb{D})\cap C^{1,\alpha}(\overline{\mathbb{D}})$ smooth solution $\Pi$ to \eqref{den33} can be found in \cite{gt98}. To deal with the singularity near the corner, one can use the standard symmetric extension technique to extend $\Pi, G_1, G_2, G_3$ as follows:
\be\no
(\tilde{\Pi},\tilde{G}_1,\tilde{G}_3)(\vec{y})=\begin{cases}
(\Pi,G_1,G_3)(y_1,y_2,y_3),\ \ \ \vec{y}\in \overline{\mathbb{D}},\\
-(\Pi,G_1,G_3)(y_1,2\theta_0-y_2,y_3),\ \vec{y}\in [r_s,r_2]\times [\theta_0,3\theta_0]\times [-1,1],\\
-(\Pi,G_1,G_3)(y_1,-2\theta_0-y_2,y_3),\ \vec{y}\in [r_s,r_2]\times [-3\theta_0,-\theta_0]\times [-1,1],
\end{cases}\ee
and
\be\label{ext2}
\tilde{G}_2(\vec{y})=\begin{cases}
G_2(y_1,y_2,y_3),\ \ &\vec{y}\in \overline{\mathbb{D}},\\
G_2(y_1,2\theta_0-y_2,y_3),\ &\vec{y}\in [r_s,r_2]\times [\theta_0,3\theta_0]\times [-1,1],\\
G_2(y_1,-2\theta_0-y_2,y_3),\ &\vec{y}\in [r_s,r_2]\times [-3\theta_0,-\theta_0]\times [-1,1],
\end{cases}\ee
The extension of $\Pi, G_i,i=1,2,3$ along the $x_3$ direction can be done similarly.

Thanks to the compatibility conditions \eqref{cp20}, one has
\be\no\begin{cases}
\p_{y_1}^2 \tilde{\Pi} +\frac{1}{y_1}\p_{y_1}\tilde{\Pi}+ \frac{1}{y_1^2}\p_{y_2}^2\tilde{\Pi} + \p_{y_3}^2 \tilde{\Pi}= \p_{y_1} \tilde{G}_1 +\frac{\tilde{G}_1}{y_1} + \frac{1}{y_1}\p_{y_2} \tilde{G}_2 + \p_{y_3}\tilde{G}_3\in C^{1,\alpha}(\mathbb{D}_e),\\
\tilde{\Pi}(r_s,y_2,y_3)=\tilde{\Pi}(r_2,y_2,y_3)=0,\ \ \forall (y_2,y_3)\in [-3\theta_0,3\theta_0]\times [-3,3],
\end{cases}\ee
where $\mathbb{D}_e=(r_s,r_2)\times (-3\theta_0,3\theta_0)\times (-3,3)$. Then the regularity of $\Pi$ on $\mathbb{D}$ can be improved to be $C^{2,\alpha}(\overline{\mathbb{D}})$ with the estimate
\be\no
&&\|\Pi\|_{C^{2,\alpha}(\overline{\mathbb{D}})}\leq C_*\sum_{j=1}^3\|G_j\|_{C^{1,\alpha}(\overline{\mathbb{D}})}\leq C_*(\epsilon\|(\hat{{\bf V}}, \hat{V}_6)\|_{\mathcal{X}}+\|(\hat{{\bf V}}, \hat{V}_6)\|_{\mathcal{X}}^2)\leq C_*(\epsilon \delta_0 +\delta_0^2).
\ee
Furthermore, the following compatibility conditions hold
\be\label{cp331}\begin{cases}
\p_{y_1}\Pi(y_1,\pm\theta_0,y_3)=\p_{y_3}\Pi(y_1,\pm\theta_0,y_3)=0,\ \ &\text{on }\Sigma_2^{\pm},\\
\p_{y_1}\Pi(y_1,y_2,\pm 1)=\p_{y_2}\Pi(y_1,y_2,\pm 1)=0,\ \ &\text{on }\Sigma_3^{\pm}.
\end{cases}\ee

{\bf Step 4.2} Next we are going to solve the following divergence-curl system with homogeneous normal boundary conditions
\be\label{den34}\begin{cases}
\p_{y_1} \dot{V}_1+ \frac{\dot{V}_1}{y_1}+\frac{1}{y_1}\p_{y_2} \dot{V}_2 + \p_{y_3}\dot{V}_3=0,\ \ &\text{in }\mathbb{D},\\
\frac{1}{y_1}\p_{y_2} \dot{V}_3- \p_{y_3} \dot{V}_2=G_1(\hat{{\bf V}},\hat{V}_6)-\p_{y_1}\Pi:=\tilde{G}_1,\ \ &\text{in }\mathbb{D},\\
\p_{y_3} \dot{V}_1-\p_{y_1} \dot{V}_3= G_2(V_5,\hat{{\bf V}},\hat{V}_6)-\frac{1}{y_1}\p_{y_2}\Pi:=\tilde{G}_2,\ \ &\text{in }\mathbb{D},\\
\p_{y_1} \dot{V}_2-\frac{1}{y_1}\p_{y_2} \dot{V}_1+\frac{\dot{V}_2}{y_1}= G_3(V_5,\hat{{\bf V}},\hat{V}_6)-\p_{y_3}\Pi:=\tilde{G}_3,\ \ &\text{in }\mathbb{D},\\
\dot{V}_1(r_s,y_2,y_3)= \dot{V}_1(r_2,y_2,y_3)=0,\ \ &\forall y'\in E,\\
\dot{V}_2(y_1,\pm\theta_0,y_3)=0,\ \ &\text{on }\Sigma_2^{\pm},\\
\dot{V}_3(y_1,y_2,\pm 1)=0\ \ &\text{on }\Sigma_3^{\pm}.
\end{cases}\ee

Since $\Pi$ satisfies the equation in \eqref{den33}, then
\be\no
\p_{y_1} \tilde{G}_1+\frac{1}{y_1} \tilde{G}_1+ \frac{1}{y_1}\p_{y_2} \tilde{G}_2+\p_{y_3} \tilde{G}_3\equiv 0,\ \text{in }\mathbb{D}.
\ee
Also it follows from \eqref{cp20} and \eqref{cp331} that
\be\label{cp340}\begin{cases}
\tilde{G}_1(y_1,\pm\theta_0,y_3)=\tilde{G}_3(y_1,\pm\theta_0,y_3)=\p_{y_2}\tilde{G}_2(y_1,\pm\theta_0,y_3)=0,\ \ &\text{on }\Sigma_2^{\pm},\\
\tilde{G}_1(y_1,y_2,\pm 1)=\tilde{G}_2(y_1,y_2,\pm 1)==\p_{y_3}\tilde{G}_3(y_1,y_2,\pm 1)=0,\ \ &\text{on }\Sigma_3^{\pm}.
\end{cases}
\ee

The unique solvability of the divergence-curl system with the homogeneous normal boundary condition is well-known (cf. \cite{ky09} and the references therein). By the compatibility condition \eqref{cp340} and the symmetric extension technique as above, there exists a unique $C^{2,\alpha}(\overline{\mathbb{D}})$ smooth vector field $(\dot{V}_1,\dot{V}_2,\dot{V}_3)$ solving \eqref{den34} with the estimate
\be\no
&&\sum_{j=1}^3\|\dot{V}_j\|_{C^{2,\alpha}(\overline{\mathbb{D}})}\leq C_*\sum_{j=1}^3\|\tilde{G}_j\|_{C^{1,\alpha}(\overline{\mathbb{D}})}\leq C_*\sum_{j=1}^3\|G_j\|_{C^{1,\alpha}(\overline{\mathbb{D}})}+
\|\Pi\|_{C^{2,\alpha}(\overline{D})}\\\no
&&\leq C_*\sum_{j=1}^3\|G_j\|_{C^{1,\alpha}(\overline{\mathbb{D}})}\leq C_*(\epsilon\|(\hat{{\bf V}}, \hat{V}_6)\|_{\mathcal{X}}+\|(\hat{{\bf V}}, \hat{V}_6)\|_{\mathcal{X}}^2)\leq C_*(\epsilon \delta_0 +\delta_0^2)
\ee
and the following compatibility conditions hold
\be\label{den342}\begin{cases}
\p_{y_2} (\dot{V}_1,\dot{V}_3)(y_1,\pm\theta_0,y_3)=(\dot{V}_2,\p_{y_2}^2 \dot{V}_2)(y_1,\pm\theta_0,y_3)=0, \ &\text{on }\Sigma_2^{\pm},\\
\p_{y_3} (\dot{V}_1,\dot{V}_2)(y_1,y_2,\pm 1)=(\dot{V}_3,\p_{y_3}^2 \dot{V}_3)(y_1,y_2, \pm 1)=0, \ &\text{on }\Sigma_3^{\pm}.
\end{cases}\ee

{\bf Step 4.3} Let $(V_1,V_2,V_3)$ be the solution to \eqref{den32}, and set
\be\no
N_j(y)= V_j(y)- \dot{V}_j(y), j=1,2,3.
\ee
Then $N_j, j=1,2,3$ solve the following problem
\be\label{den36}\begin{cases}
d_1\p_{y_1} N_1+\frac{1}{y_1}\p_{y_2} N_2 +\p_{y_3} N_3 +\frac{1}{y_1} N_1+d_2 N_1=G_4(\hat{{\bf V}},\hat{V}_6),\ \ \text{in }\mathbb{D},\\
\frac{1}{y_1}\p_{y_2} N_3-\p_{y_3} N_2=0,\ \ \text{in }\mathbb{D},\\
\p_{y_3} \left(N_1+d_3(y_1) N_1(r_s,y')\right)-\p_{y_1} N_3= 0,\ \ \text{in }\mathbb{D},\\
\p_{y_1} N_2+\frac{N_2}{y_1}-\frac{1}{y_1}\p_{y_2} \left(N_1+d_3(y_1) N_1(r_s,y')\right)= 0,\ \ \text{in }\mathbb{D},\\
(\frac{1}{r_s^2}\p_{y_2}^2+\p_{y_3}^2) N_1(r_s,y')-a_0 a_1 (\frac{1}{r_s}\p_{y_2} N_2+\p_{y_3}N_3)(r_s,y')=q_5(y'),\ \ \forall y'\in E,\\
N_2(y_1,\pm \theta_0,y_3)=0,\ \ \text{on }\Sigma_2^{\pm},\\
N_3(y_1,y_2,\pm 1)=0,\ \ \text{on }\Sigma_3^{\pm},\\
N_1(r_2,y')+ d_3(r_2) N_1(r_s,y')=q_4(\hat{{\bf V}}(r_s,y'),\hat{V}_6(y')), \ \forall y'\in E,
\end{cases}\ee
where
\be\no
&&G_4(\hat{{\bf V}},\hat{V}_6)=-\frac{(\gamma-1)(\bar{U}'+\frac{\bar{U}}{r})}{c^2(\bar{\rho},\bar{K})} V_5+G_0(\hat{{\bf V}},\hat{V}_6)+\bar{M}^2(y_1)\p_{y_1}\dot{V}_1-d_2(y_1)\dot{V}_1,\\\no
&&q_5(\hat{{\bf V}}(r_s,y'),\hat{V}_6(y'))=q_1(\hat{{\bf V}}(r_s,y'),\hat{V}_6(y'))+a_0a_1(\frac{1}{r_s}\p_{y_2}\dot{V}_2+\p_{y_3} \dot{V}_3)(r_s,y').
\ee

The boundary conditions \eqref{shock22}, \eqref{shock23} and \eqref{den342} imply that
\be\label{shock24}
&&\left(\frac{1}{r_s}\p_{y_2}N_1-a_0a_1N_2\right)(r_s,\pm \theta_0,y_3)=0,\ \forall y_3\in [-1,1],\\\label{shock25}
&&(\p_{y_3} N_1-a_0 a_1 N_3)(r_s,y_2,\pm 1)= 0,\ \ \forall y_2\in [-\theta_0,\theta_0].
\ee

It follows from the second, third and fourth equations in \eqref{den36} that there exists a potential function $\phi$ such that
\be\no
N_1(y_1,y') + d_3(y_1) N_1(r_s,y')=\p_{y_1}\phi(y_1,y'),\, N_2=\frac{1}{y_1}\p_{y_2}\phi, \ N_3=\p_{y_3}\phi.
\ee
Therefore
\be\no
&&N_1(r_s,y')=\frac 1{a_3}\p_{y_1}\phi(r_s,y'),\\\no
&&N_1(y_1,y')=\p_{y_1}\phi(y_1,y')- \frac{1}{a_3} d_3(y_1)\p_{y_1}\phi(r_s,y')
\ee
with
\be\no
a_3=\frac{(\gamma-1) \bar{M}^2(r_s)+1}{\gamma \bar{M}^2(r_s)}>0.
\ee
Thus the problem \eqref{den36} is equivalent to
\be\label{den37}\begin{cases}
d_1(y_1) \p_{y_1}^2 \phi+\frac{1}{y_1^2}\p_{y_2}^2 \phi +\p_{y_3}^2 \phi +(\frac{1}{y_1}+ d_2(y_1))\p_{y_1}\phi\\
\q\q-\frac{1}{a_3} d_4(y_1)\p_{y_1}\phi(r_s,y')=G_4(\hat{{\bf V}},\hat{V}_6),\ \ &\text{in }\mathbb{D},\\
(\frac{1}{r_s^2}\p_{y_2}^2+\p_{y_3}^2)\b(\p_{y_1}\phi(r_s,y_2,y_3)-a_4 \phi(r_s,y')\b)=a_3q_5(y'),\ \ &\forall y'\in E,\\
\p_{y_2}\phi(y_1,\pm \theta_0,y_3)=0,\ \ &\text{on }\Sigma_2^{\pm},\\
\p_{y_3}\phi(y_1,y_2,\pm 1)= 0,\ \ \ \ &\text{on }\Sigma_3^{\pm},\\
\p_{y_1}\phi(r_2,y')=q_4(y'),\ \ &\forall y'\in E,
\end{cases}\ee
where
\be\no
&&d_4(y_1)=d_1(y_1)d_3'(y_1)+(\frac{1}{y_1}+d_2(y_1)) d_3(y_1)\\\no
&&\quad=- \frac{a_2}{a_1}d_1(y_1)\frac{\bar{B}+\frac{1}{2}\bar{U}^2(y_1)}{\gamma\bar{K}}\frac{\bar{U}'(y_1)}{\bar{U}(y_1)}+(\frac{1}{y_1}+d_2(y_1)) d_3(y_1)\\\no
&&\quad=\frac{a_2}{a_1\gamma\bar{K}y_1\bar{U}(y_1)}\left(\bar{B}+\frac{1}{2}\bar{U}^2(y_1)+(\bar{B}-\frac{1}{2}\bar{U}^2(y_1))\left(1+\frac{\bar{M}^2(2+(\gamma-1)\bar{M}^2)}{1-\bar{M}^2}\right)\right)\\\no
&&\quad=\frac{a_2}{\gamma a_1\bar{K} y_1 \bar{U}(y_1)}\left(2\bar{B}+\frac{\bar{U}^2(y_1)(2+(\gamma-1)\bar{M}^2(y_1))}{(\gamma-1)(1-\bar{M}^2(y_1))}\right)>0,\\\no
&& a_4=a_0 a_1 a_3>0.
\ee

Moreover, the boundary conditions \eqref{shock24} and \eqref{shock25} can be rewritten as
\be\label{shock26}
&&\frac{1}{r_s}\p_{y_2}(\p_{y_1}\phi-a_4 \phi)(r_s,\pm \theta_0,y_3)=0,\ \forall y_3\in [-1,1],\\\label{shock27}
&&\p_{y_3} (\p_{y_1}\phi-a_4\phi)(r_s,y_2,\pm 1)=0,\ \ \forall y_2\in [-\theta_0,\theta_0].
\ee

A key issue here is the derivation of the oblique boundary condition for the potential $\phi$ on the boundary $\{(r_s,y'): y'\in E\}$ by solving the Poisson equation (the first boundary condition in \eqref{den37}) with the Neumann boundary conditions \eqref{shock26}-\eqref{shock27}.
\begin{lemma}\label{oblique}({\bf The oblique boundary condition on the shock front.})
On the shock front $\{(r_s, y'): y'\in E\}$, there exists a unique $C^{2,\alpha}(\overline{E})$ function $m_1(y')$ such that
\be\no
\p_{y_1}\phi(r_s,y')-a_4 \phi(r_s,y')= m_1(y')
\ee
where $m_1(y')$ satisfies the Poisson equation with the homogeneous Neumann boundary conditions
\be\label{den38}\begin{cases}
(\frac{1}{r_s^2}\p_{y_2}^2+\p_{y_3}^2)m_1(y')=a_3q_5(\hat{{\bf V}}(r_s,y'),\hat{V}_6(y')),\ \ &\text{in } E:=(-\theta_0,\theta_0)\times (-1,1),\\
\frac{1}{r_s}\p_{y_2}m_1(\pm \theta_0,y_3)=0,\ &\forall y_3\in [-1,1],\\
\p_{y_3} m_1(y_2,\pm 1)=0,\ \ &\forall y_2\in [-\theta_0,\theta_0],
\end{cases}\ee
and the condition
\be\label{den39}
\iint_{E} m_1(y_2,y_3) dy'=0.
\ee
\end{lemma}



\begin{proof}

Due to \eqref{g21},\eqref{g31},\eqref{class2},\eqref{rv11} and \eqref{den342}, the following solvability condition for \eqref{den38} holds
\be\no
&&\iint_{E} q_5(\hat{{\bf V}}(r_s,y'),\hat{V}_6(y')) dy'=\iint_{E} q_1(\hat{{\bf V}}(r_s,y'),\hat{V}_6(y')) dy'\\\no
&&=\int_{-1}^1\frac{a_1}{r_s}g_2(\hat{{\bf V}}(r_s,y'), \hat{V}_6(y'))+\frac{1}{r_s^2} \p_{y_2} \{R_1(\hat{{\bf V}}(r_s,y'), \hat{V}_6(y'))\}\bigg|_{y_2=-\theta_0}^{\theta_0} dy_3\\\no
&&\quad+ \int_{-\theta_0}^{\theta_0}a_1g_3(\hat{{\bf V}}(r_s,y'), \hat{V}_6(y'))+\p_{y_3}\{R_1(\hat{{\bf V}}(r_s,y'), \hat{V}_6(y'))\}\bigg|_{y_3=-1}^{1} dy_2=0.
\ee
Thanks to the homogeneous Neumann boundary conditions in \eqref{den38}, by the symmetric extension technique, one can get that there exists a unique solution $m_1(y')\in C^{2,\alpha}(\overline{E})$ to \eqref{den38} satisfying \eqref{den39} with
\be\no
&&\|m_1\|_{C^{2,\alpha}(\overline{E})}\leq C_*\|q_5(\hat{{\bf V}}(r_s,y'),\hat{V}_6(y'))\|_{C^{\alpha}(\overline{E})}\\\no
&&\leq C_*(\|R_1(y')\|_{C^{2,\alpha}(\overline{E})}+\sum_{j=2}^3\|\dot{V}_j(r_s,\cdot)\|_{C^{1,\alpha}(\overline{E})}+\|g_j\|_{C^{1,\alpha}(\overline{E})})\\\no
&&\leq C_*(\epsilon\|(\hat{{\bf V}}, \hat{V}_6)\|_{\mathcal{X}}+\|(\hat{{\bf V}}, \hat{V}_6)\|_{\mathcal{X}}^2)\leq C_*(\epsilon \delta_0 +\delta_0^2).
\ee

\end{proof}

Then the problem, \eqref{den37} and \eqref{shock26}-\eqref{shock27}, can be reduced to
\be\label{den41}\begin{cases}
\p_{y_1}(d_1\p_{y_1} \phi)+\frac{1}{y_1^2}\p_{y_2}^2 \phi +\p_{y_3}^2 \phi + d_5\p_{y_1}\phi-a_0a_1 d_4\phi(r_s,y')=G_5(y),\ \ \text{in }\mathbb{D},\\
\p_{y_1}\phi(r_s,y')-a_4 \phi(r_s,y')=m_1(y'),\ \ \forall y'\in E,\\
\p_{y_2}\phi(y_1,\pm \theta_0,y_3)=0,\ \ \text{on }\Sigma_2^{\pm},\\
\p_{y_3}\phi(y_1,y_2,\pm 1)= 0,\ \ \ \text{on }\Sigma_3^{\pm}, \\
\p_{y_1}\phi(r_2,y')=m_2(y'),\ \ \forall y'\in E,
\end{cases}\ee
where
\be\no
&&G_5(y)=G_4(y)+ \frac{d_4(y_1)}{a_3} m_1(y'),\\\no
&&d_5(y_1)=\frac{1}{y_1}+d_2(y_1)-d_1'(y_1),\ \ m_2(y')=q_4(\hat{{\bf V}}(r_s,y'),\hat{V}_6(y')).
\ee
It can be checked easily that the function $G_5$ and $m_i, i=1,2$ satisfy the following compatibility conditions
\be\label{den415}\begin{cases}
\p_{y_2}G_5(y_1,\pm\theta_0, y_3)=0,\ &\text{on }\Sigma_2^{\pm},\\
\p_{y_3}G_5(y_1,y_2, \pm 1)=0,\ &\text{on }\Sigma_3^{\pm},\\
\p_{y_2} m_1(\pm \theta_0,y_3)=\p_{y_2} m_2(\pm \theta_0,y_3)=0, \ &\forall y_3\in [-1,1],\\
\p_{y_3} m_1(y_2,\pm 1)=\p_{y_3} m_2(y_2, \pm 1)=0, \ &\forall y_2\in [-\theta_0,\theta_0].
\end{cases}\ee

Note that in \eqref{den41} we have replaced the term $\frac{d_4(y_1)}{a_3}\p_{y_1}\phi(r_s,y')$ in the first equation of \eqref{den37} by $a_0a_1d_4(y_1)\phi(r_s,y')$ using the oblique boundary condition at $y_1=r_s$. It simplifies greatly the unique solvability of the problem \eqref{den41} since the nonlocal term involves only the trace $\phi(r_s,y')$ not the derivative $\p_{y_1}\phi(r_s,y')$, so that one can use the Lax-Milgram theorem and the Fredholm alternatives for second order elliptic equations to establish the existence and uniqueness of the solution to \eqref{den41}.

Indeed, first, the weak solution to \eqref{den41} can be obtained as follows. $\phi\in H^1(\mathbb{D})$ is said to be a weak solution to \eqref{den41}, if for any $\psi\in H^1(\mathbb{D})$, the following equality holds
\be\label{den411}
\mathcal{B}(\phi,\psi)= \mathcal{L}(\psi),\ \ \forall \psi\in H^1(\mathbb{D}),
\ee
where
\be\no
&&\mathcal{B}(\phi,\psi)=\iiint_{\mathbb{D}} d_1(y_1)\p_{y_1}\phi \p_{y_1}\psi + \frac{1}{y_1^2}\p_{y_2}\phi \p_{y_2}\psi + \p_{y_3}\phi \p_{y_3}\psi\\\no
&&- d_5(y_1) \p_{y_1}\phi \psi + a_0a_1 d_4(y_1)\phi(r_s,y')\psi(y_1,y')dy_1 dy' + \iint_{E} d_1(r_s)a_4 \phi(r_s,y')\psi(r_s,y') dy',\\\no
&&\mathcal{L}(\psi)= -\iiint_{\mathbb{D}} \psi G_5 dy_1 dy'+ \iint_{E} d_1(r_2) m_2(y')\psi(r_2,y') -d_1(r_s) m_1(y')\psi(r_s,y') dy'.
\ee
Next, to solve \eqref{den411}, one observes that
\begin{lemma}\label{weak}
{\it There exists a positive constant $K$ depending only on the background solution such that the following problem has a unique weak solution in $H^1(\mathbb{D})$
\be\label{den420}\begin{cases}
\p_{y_1}(d_1\p_{y_1} \phi)+\frac{1}{y_1^2}\p_{y_2}^2 \phi +\p_{y_3}^2 \phi + d_5\p_{y_1}\phi-a_0a_1 d_4\phi(r_s,y')-K\phi=G_5(y),\ \ \text{in }\mathbb{D},\\
\p_{y_1}\phi(r_s,y')-a_4 \phi(r_s,y')=m_1(y'),\ \ \forall y'\in E,\\
\p_{y_2}\phi(y_1,\pm \theta_0,y_3)=0,\ \ \text{on }\Sigma_2^{\pm},\\
\p_{y_3}\phi(y_1,y_2,\pm 1)= 0,\ \ \ \text{on }\Sigma_3^{\pm},\\
\p_{y_1}\phi(r_2,y')=m_2(y'),\ \ \forall y'\in E.
\end{cases}\ee
}\end{lemma}

\begin{proof}
The weak formulation of the problem \eqref{den420} is the existence of a $H^1(\mathbb{D})$ function $\phi$ such that
\be\label{den421}
\mathcal{B}_K(\phi,\psi):=\mathcal{B}(\phi,\psi) + K\iiint_{\mathbb{D}} \phi\psi dy= \mathcal{L}(\psi),\ \ \forall \psi\in H^1(\mathbb{D}).
\ee

For any $\epsilon>0$, there holds
\be\no
\iint_{E}\phi^2(r_s,y')dy'\leq \frac{C_0}{\epsilon} \iiint_{\mathbb{D}} \phi^2(y_1,y')dy' dy_1+ \epsilon \iiint_{\mathbb{D}}(\p_{y_1}\phi)^2(y_1,y') dy' dy_1.
\ee
The boundedness and coercivity of $\mathcal{B}_K$ can be verified as follows
\be\no
&&|\mathcal{B}_K(\phi,\psi)|\leq C_0\|\phi\|_{H^1(\mathbb{D})}\|\psi\|_{H^1(\mathbb{D})},\\\no
&&|\mathcal{L}(\psi)|\leq C_0(\|G_5\|_{L^2(\mathbb{D})}+\sum_{j=1}^2\|m_j\|_{L^2(E)})\|\psi\|_{H^1(\mathbb{D})}
\ee
and
\be\no
&&\mathcal{B}_K(\phi,\phi)=\iiint_{\mathbb{D}} d_1(y_1)(\p_{y_1}\phi)^2 + \frac{1}{y_1^2}(\p_{y_2}\phi)^2+ (\p_{y_3}\phi)^2 - d_5(y_1) \p_{y_1}\phi \phi \\\no
&&\quad\quad+ a_0a_1 d_4(y_1)\phi(r_s,y')\phi(y_1,y')dy_1 dy' +K\|\phi\|_{L^2(\mathbb{D})}^2 + \iint_{E} d_1(r_s)a_4 (\phi(r_s,y'))^2 dy',\\\no
&&\geq C_*(\|\nabla\phi\|_{L^2(\mathbb{D})}^2+\|\phi(r_s,\cdot)\|_{L^2(E)}^2)+K\|\phi\|_{L^2(\mathbb{D})}^2-\frac{C_*}{4}\|\p_{y_1}\phi\|_{L^2(\mathbb{D})}^2\\\no
&&\quad-\frac{C_*}{4}\|\phi(r_s,\cdot)\|_{L^2(E)}^2-\tilde{C}_*\|\phi\|_{L^2(\mathbb{D})}^2\\\no
&&\geq \frac{C_*}{2}(\|\nabla\phi\|_{L^2(\mathbb{D})}^2+\|\phi(r_s,\cdot)\|_{L^2(E)}^2)+\frac{K}{2}\|\phi\|_{L^2(\mathbb{D})}^2,
\ee
provided that $K$ is large enough. By the Lax-Milgram theorem, there exists a unique $H^1(\mathbb{D})$ solution $\phi$ satisfying \eqref{den421}, which completes the proof of Lemma \ref{weak}.
\end{proof}

Now we are ready to solve the problem \eqref{den41}.
\begin{proposition}\label{solvability}
{\it Suppose that $G_5\in C^{1,\alpha}(\overline{\mathbb{D}})$ and $(m_1,m_2)\in (C^{2,\alpha}(\overline{E}))^2$ satisfy \eqref{den415}. Then there exists a unique $C^{3,\alpha}(\overline{\mathbb{D}})$ to the problem \eqref{den41} with
\be\label{den414}
\|\phi\|_{C^{3,\alpha}(\overline{\mathbb{D}})}\leq C_*(\|G_5\|_{C^{1,\alpha}(\overline{\mathbb{D}})}+\sum_{j=1}^2\|m_j\|_{C^{2,\alpha}(\overline{E})}),
\ee
where $C_*$ depends only on $d_1,d_4,d_5, a_3,a_4$ and thus on the background solution.
}\end{proposition}

\begin{proof}

We first improve the regularity of any $H^1(\mathbb{D})$ weak solutions to \eqref{den41}. The goal is to show that for any weak solution $\phi\in H^1(\mathbb{D})$ to \eqref{den41}, $\phi$ indeed has a better regularity $\phi\in C^{3,\alpha}(\overline{\mathbb{D}})$ satisfying the following estimate:
\be\label{den412}
\|\phi\|_{C^{3,\alpha}(\overline{\mathbb{D}})}\leq C_*(\|\phi\|_{H^1(\mathbb{D})}+\|G_5\|_{C^{1,\alpha}(\overline{\mathbb{D}})}+\sum_{j=1}^2\|m_j\|_{C^{2,\alpha}(\overline{E})}).
\ee

To this end, one can rewrite the first equation in \eqref{den41} as a standard second order elliptic equation for $\phi$:
\be\no
\p_{y_1}(d_1(y_1) \p_{y_1} \phi)+\frac{1}{y_1^2}\p_{y_2}^2 \phi +\p_{y_3}^2 \phi + d_5(y_1)\p_{y_1}\phi=G_6(y):=G_5+a_0a_1 d_4(y_1)\phi(r_s,y').
\ee
Since $\phi\in H^1(\mathbb{D})$, $\phi(r_s,y')\in L^2(E)$ by the trace theorem. Together with the boundary conditions in \eqref{den41}, one can apply \cite[Theorems 5.36 and 5.45]{lieberman13} to obtain global $L^{\infty}$ bound and $C^{\alpha_1}$ estimates (for some $\alpha_1\in (0,1)$) on $\phi$ as follows,
\be\no
\|\phi\|_{C^{0,\alpha_1}(\overline{\mathbb{D}})}&\leq& C_*\bigg(\|a_0a_1d_4(y_1)\phi(r_s,y')\|_{L^2(E)}+\|G_5\|_{L^4(\mathbb{D})}+\sum_{j=1}^2\|m_j\|_{C^{1,\alpha}(\overline{E})}\bigg)\\\no
&\leq&C_*(\|\phi\|_{H^1(\mathbb{D})}+\|G_5\|_{C^{1,\alpha}(\overline{\mathbb{D}})}+\sum_{j=1}^2\|m_j\|_{C^{1,\alpha}(\overline{E})}).
\ee
Hence the term $a_0a_1 d_4(y_1) \phi(r_s,y')\in C^{\alpha_1}(\overline{\mathbb{D}})$ and the Schauder type estimate (cf. \cite[Theorem 4.6]{lieberman13}) would imply that
\be\no
\|\phi\|_{C^{1,\alpha}(\overline{\mathbb{D}})}\leq C_*(\|\phi\|_{H^1(\mathbb{D})}+\|G_5\|_{C^{1,\alpha}(\overline{\mathbb{D}})}+\sum_{j=1}^2\|m_j\|_{C^{\alpha}(\overline{E})}).
\ee
We extend $\phi,G_6$ to $\tilde{\phi}$ and $\tilde{G}_6$ on $\mathbb{D}_e$ as for $G_2$ in \eqref{ext2}, and extend $m_i,i=1,2$ as
\be\no
\tilde{m}_i(y')=\begin{cases}
m_i(y_2,y_3),\ y'\in \overline{\mathbb{D}},\\
m_i(2\theta_0-y_2,y_3),\ y'\times [\theta_0,3\theta_0]\times [-1,1],\\
m_i(-2\theta_0-y_2,y_3),\ y'\times [-3\theta_0,-\theta_0]\times [-1,1],
\end{cases}\ee
The extension of $m_i,i=1,2$ along the $x_3$ direction can be defined similarly. Then $\tilde{\phi}$ satisfies
\be\no\begin{cases}
\p_{y_1}(d_1(y_1) \p_{y_1} \tilde{\phi})+\frac{1}{y_1^2}\p_{y_2}^2 \tilde{\phi} +\p_{y_3}^2 \tilde{\phi} + d_5(y_1)\p_{y_1}\tilde{\phi}=\tilde{G}_6(y),\ \ \text{in }\mathbb{D}_e,\\
\p_{y_1}\tilde{\phi}(r_s,y')-a_4 \tilde{\phi}(r_s,y')=\tilde{m}_1(y'),\ \forall y'\in [-3\theta_0,3\theta_0]\times [-3,3],\\
\p_{y_1}\tilde{\phi}(r_2,y')=\tilde{m}_2(y'),\ \forall y'\in [-3\theta_0,3\theta_0]\times [-3,3].
\end{cases}\ee
Thanks to \eqref{den415}, one has $\tilde{G}_6\in C^{1,\alpha}(\overline{\mathbb{D}_e})$ and $\tilde{m}_i\in C^{2,\alpha}([-3\theta_0,3\theta_0]\times [-3,3])$ and the desired estimate \eqref{den412} follows from the standard Schauder estimates.

Next we will show the uniqueness of the $H^1(\mathbb{D})$ weak solution to \eqref{den41}, i.e. if $G_4\equiv 0, m_1=m_2=0$, and $\phi(y)\in H^1(\mathbb{D})$ is a weak solution to \eqref{den41}, then $\phi\equiv 0$ on $\mathbb{D}$.

Let $\{\beta_i(y_2)\}_{i=1}^{\infty}$ be the family of all eigenfunctions
to the eigenvalue problem
\be\no\begin{cases}
-\beta_i''(y_2)= \tau_i^2 \beta_i(y_2),\  y_2 \in (-\theta_0,\theta_0),\\
\beta_i'(-\theta_0)= \beta_i'(\theta_0)=0.
\end{cases}\ee
Then one has
\be\no
\{\beta_i(y_2)\}_{i=0}^{\infty}= \{\frac{1}{\sqrt{2 \theta_0}}\}\cup \{\frac{1}{\sqrt{\theta_0}}\cos(\frac{k\pi}{\theta_0} y_2)\}_{k=1}^{\infty} \cup \{\sin (\frac{2k+1}{2\theta_0}\pi y_2)\}_{k=0}^{\infty},
\ee
which form a complete orthonormal basis in $L^2((-\theta_0,\theta_0))$ and an orthogonal basis in $H^1((-\theta_0,\theta_0))$. Let
\be\no
\{\mu_j(y_3)\}_{j=0}^{\infty}=\{\frac{1}{\sqrt{2}}\}\cup \{\cos(k\pi y_3)\}_{k=1}^{\infty} \cup \{\sin (\frac{2k+1}{2}\pi y_3)\}_{k=0}^{\infty}.
\ee
Then $\{\mu_j(y_3)\}_{j=0}^{\infty}$ form a complete orthonormal basis in $L^2((-1,1))$ and an orthogonal basis in $H^1((-1,1))$. Denote the eigenvalue associated to $\mu_j$ by $\lambda_j^2$ for any $j\geq 0$. Then the set $\{\beta_i(y_2)\mu_j(y_3)\}_{i,j=0}^{\infty}$ forms a complete orthonormal basis in $L^2((-\theta_0,\theta_0)\times (-1,1))$ and an orthogonal basis in $H^1((-\theta_0,\theta_0)\times(-1,1))$.

It follows from the previous regularity that $\phi\in C^{3,\alpha}(\overline{\mathbb{D}})$. Thus its Fourier series converges
\be\no
\phi(y_1,y')=\sum_{i,j=0}^{\infty} X_{i,j}(y_1)\beta_i(y_2)\mu_j(y_3).
\ee
Substituting this into \eqref{den41} yields that for $i,j\geq 0$, it holds that
\be\no\begin{cases}
d_1(y_1) X_{i,j}''(y_1) +(\frac{1}{y_1}+d_2(y_1)) X_{i,j}'(y_1)-(\frac{\tau_i^2}{y_1^2}+\lambda_j^2)  X_{i,j}(y_1)-a_0 a_1 d_4(y_1)X_{i,j}(r_s)=0,\\
X_{i,j}'(r_s)-a_4 X_{i,j}(r_s)=0,\\
X_{i,j}'(r_2)=0.
\end{cases}\ee
Suppose that $X_{i,j}(r_s)=0$, then the maximum principle and Hopf's lemma imply that $X_{i,j}(y_1)\equiv 0$ for $\forall y_1\in [r_s,r_2]$. Suppose that $X_{i,j}(r_s)>0$. Then
\be\label{den45}\begin{cases}
d_1(y_1) X_{i,j}''(y_1) + (\frac{1}{y_1}+d_2(y_1)) X_{i,j}'(y_1)-(\frac{\tau_i^2}{y_1^2}+\lambda_j^2) X_{i,j}(y_1)\\
\q\q=a_0a_1 d_4(y_1)X_{i,j}(r_s)>0,\ \forall y_1\in [r_s,r_2],\\
X_{i,j}'(r_s)=a_4 X_{i,j}(r_s)>0,\ \ \ \ X_{i,j}'(r_2)=0.
\end{cases}\ee
Assume that $X_{i,j}(y_1)$ achieves its maximum at $y_1=\hat{y}_1$: $X_{i,j}(\hat{y}_1)=\max_{y_1\in [r_s,r_2]} X_{i,j}(y_1)>0$. Since $X_{i,j}(r_s)>0$ and $X_{i,j}'(r_s)>0$, $\hat{y}_1\in (r_s,r_2]$. If $\hat{y}_1\in (r_s,r_2)$, then $X_{i,j}'(\hat{y}_1)=0, X_{i,j}''(\hat{y}_1)\leq 0$, which contradicts to the first equation in \eqref{den45}. If $\hat{y}_1=r_2$, by Hopf's lemma, one has $X_{i,j}'(r_2)>0$, which is also a contradiction too. Similarly, $X_{i,j}(r_s)<0$ will induce a contradiction. Therefore $X_{i,j}(y_1)\equiv 0$ for all $y_1\in [r_s,r_2]$. Consequently, $\phi\equiv 0$ in $\mathbb{D}$.

Thus we have proved the uniqueness of the $H^1$ weak solution to \eqref{den41}. This together with Lemma \ref{weak} and the Fredholm alternatives for elliptic equations implies the existence and uniqueness of $H^1(\mathbb{D})$ weak solution to \eqref{den41} (See the argument in \cite[Theorem 8.6]{gt98}). With the aid of uniqueness, the estimate \eqref{den414} follows from \eqref{den412}. Hence the proof of the proposition is completed.

\end{proof}

Thus $N_1(y)=\p_{y_1}\phi(y)-\frac{1}{a_3} d_3(y_1)\p_{y_1}\phi(r_s,y'), N_2(y)=\frac{1}{y_1}\p_{y_2}\phi(y_1,y')$ and $N_3(y)=\p_{y_3}\phi$ would solve the problem \eqref{den36} with \eqref{shock24}-\eqref{shock25}. Differentiating the first equation in \eqref{den36} with respect to $y_2$ (resp. $y_3$) and evaluating at $y_2=\pm\theta_0$ (resp. $y_3=\pm 1$), one gets from \eqref{den342},\eqref{shock26} and \eqref{shock27} that
\be\no\begin{cases}
\p_{y_2}^2 N_2(y_1,\pm \theta_0,y_3)=0,\ \ &\text{on } \Sigma_2^{\pm},\\
\p_{y_3}^2 N_3(y_1,y_2,\pm 1)=0,\ \ &\text{on } \Sigma_3^{\pm}.
\end{cases}\ee

Then
\be\no
&&V_1(y_1,y')=\dot{V}_1(y)+\p_{y_1}\phi(y)-\frac{1}{a_3} d_3(y_1)\p_{y_1}\phi(r_s,y'),\\\no
&&V_2(y_1,y')=\dot{V}_2(y_1,y')+\frac{1}{y_1}\p_{y_2}\phi(y_1,y'),\ \ \ V_3(y_1,y')=\dot{V}_3(y_1,y')+\p_{y_3}\phi(y_1,y'),
\ee
will solve the problem \eqref{den32} with \eqref{shock22}-\eqref{shock23} and satisfy
\be\no
&&\sum_{j=1}^3\|V_j\|_{C^{2,\alpha}(\overline{\mathbb{D}})}\leq C_*(\sum_{j=1}^3\|\dot{V}_j\|_{C^{2,\alpha}(\overline{\mathbb{D}})}+\|\nabla \phi\|_{C^{2,\alpha}(\overline{\mathbb{D}})}+ \|\p_{y_1}\phi(r_s,y')\|_{C^{2,\alpha}(\overline{E})})\\\label{den46}
&&\leq C_*(\epsilon +C_*(\epsilon\|(\hat{{\bf V}}, \hat{V}_6)\|_{\mathcal{X}}+\|(\hat{{\bf V}}, \hat{V}_6)\|_{\mathcal{X}}^2)\leq C_*(\epsilon+\epsilon \delta_0 +\delta_0^2).
\ee
Also the following compatibility conditions hold
\be\label{den461}\begin{cases}
(V_2,\p_{y_2}^2 V_2,\p_{y_2} V_1,\p_{y_2} V_3)(y_1,\pm\theta_0,y_3)=0,\ \ \ &\text{on } \Sigma_2^{\pm},\\
(V_3,\p_{y_3}^2 V_3,\p_{y_3} V_1,\p_{y_3} V_2)(y_1,y_2,\pm 1)=0,\ \ &\text{on } \Sigma_3^{\pm}.
\end{cases}\ee

{\bf Step 5.} After obtaining $V_1,V_2,V_3$, one can determine uniquely the function $V_4$ in \eqref{ent43}:
\be\label{ent53}
V_4(y_1,y')=\frac{a_2}{a_1} V_1(r_s,y')+R_4(\hat{{\bf V}}(r_s,\beta_2(y),\beta_3(y)),\hat{V}_6(\beta_2(y),\beta_3(y))).
\ee
Then it can be checked easily that the following estimate and compatibility conditions hold:
\be\no
&&\|V_4\|_{C^{2,\alpha}(\overline{\mathbb{D}})}\leq C_*\|V_1(r_s,\cdot)\|_{C^{2,\alpha}(\overline{E})}+ C_*(\epsilon \|(\hat{{\bf V}}, \hat{V}_6)\|_{\mathcal{X}}+\|(\hat{{\bf V}}, \hat{V}_6)\|_{\mathcal{X}}^2)\\\label{ent54}
&&\leq C_*(\epsilon \delta_0+ \delta_0^2)
\ee
and
\be\label{ent55}\begin{cases}
\p_{y_2} V_4(y_1,\pm \theta_0, y_3)=\frac{a_2}{a_1}\p_{y_2} V_1(r_s,\pm\theta_0,y_3)=0,\ \ \text{on }\Sigma_2^{\pm},\\
\p_{y_3} V_4(y_1,y_2,\pm 1)=\frac{a_2}{a_1}\p_{y_3} V_1(r_s,y_2,\pm 1)=0,\ \ \text{on }\Sigma_3^{\pm}.
\end{cases}\ee

Finally, the shock front is given by
\be\label{shock50}
V_6(y')=\frac{1}{a_1} V_1(r_s,y')-\frac{1}{a_1} R_1(\hat{{\bf V}}(r_s,y'),\hat{V}_6(y')),
\ee
and it is clear that $V_6\in C^{2,\alpha}(\overline{E})$ and
\be\label{shock51}\begin{cases}
\p_{y_2}V_6(\pm\theta_0, y_3)=0,\ \ &\text{on }y_3\in [-1,1],\\
\p_{y_3}V_6(y_2, \pm 1)=0,\ \ &\text{on }y_2\in [-\theta_0,\theta_0].
\end{cases}\ee
It remains to improve the regularity of $V_6$ to be $C^{3,\alpha}(\overline{E})$. To this purpose, one can define
\be\no\begin{cases}
F_2(y'):= (\frac{1}{r_s}\p_{y_2} V_1-a_0 V_2)(r_s,y')-(a_1 g_2+\frac{1}{r_s}\p_{y_2}R_1)(\hat{{\bf V}}(r_s,y'),\hat{V}_6(y')),\\
F_3(y'):= (\p_{y_3} V_1-a_0 V_3)(r_s,y')-(a_1 g_3+\p_{y_3}R_1)(\hat{{\bf V}}(r_s,y'),\hat{V}_6(y')).
\end{cases}\ee
Then it follows from the first boundary condition in \eqref{den32}, the boundary data in \eqref{vor501} and the boundary conditions \eqref{shock22}-\eqref{shock23} that
\be\no\begin{cases}
\frac{1}{r_s}\p_{y_2} F_2 + \p_{y_3} F_3=0,\ \ &\text{in }E,\\
\frac{1}{r_s}\p_{y_2} F_3-\p_{y_3} F_2=0,\ \ &\text{in }E,\\
F_2(\pm\theta_0, y_3)=0,\ \ &\text{on }y_3\in [-1,1],\\
F_3(y_2,\pm 1)=0,\ \ \ &\text{on }y_2 \in [-\theta_0,\theta_0].
\end{cases}\ee
Thus by Lemma \ref{equi0}, $F_2=F_3\equiv 0$ in $E$. Using the equation \eqref{shock50}, there holds
\be\label{shock53}\begin{cases}
\frac{1}{r_s}\p_{y_2} V_6(y')= a_0 V_2(r_s,y') + g_2(\hat{{\bf V}}(r_s,y'),\hat{V}_6(y')), \ \ &\text{in }E,\\
\p_{y_3} V_6(y')= a_0 V_3(r_s,y') + g_3(\hat{{\bf V}}(r_s,y'),\hat{V}_6(y')), \ \ &\text{in }E.
\end{cases}\ee
Therefore $V_6\in C^{3,\alpha}(\overline{E})$ admits the following estimate
\be\label{shock54}
&&\|V_6\|_{C^{3,\alpha}(\overline{E})}\leq C_*(\|V_1(r_s,\cdot)\|_{C^{2,\alpha}(\overline{E})}+ \|R_1(\hat{{\bf V}}(r_s,y'),\hat{V}_6(y'))\|_{C^{2,\alpha}(\overline{E})})\\\no
&&\quad + C_*\sum_{j=2}^3(\|V_j(r_s,\cdot)\|_{C^{2,\alpha}(\overline{E})}+\|g_j(\hat{{\bf V}}(r_s,y'),\hat{V}_6(y'))\|_{C^{2,\alpha}(\overline{E})})\\\no
&&\leq C_*(\epsilon + \epsilon  \|(\hat{{\bf V}}, \hat{V}_6)\|_{\mathcal{X}}+\|(\hat{{\bf V}}, \hat{V}_6)\|_{\mathcal{X}}^2)\leq C_*(\epsilon +\epsilon \delta_0+\delta_0^2).
\ee
Differentiating the first (second) equation in \eqref{shock53} with respective to $y_2$ (resp. $y_3$) twice and evaluating at $y_2=\pm \theta_0$ (resp. $y_3=\pm 1$), using \eqref{j211} and \eqref{den461}, one can verify that
\be\label{shock55}\begin{cases}
\p_{y_2}^3V_6(\pm\theta_0, y_3)=0,\ \ &\forall y_3\in [-1,1],\\
\p_{y_3}^3 V_6(y_2, \pm 1)=0,\ \ &\forall y_2\in [-\theta_0,\theta_0].
\end{cases}\ee

Combining the estimates \eqref{ber43}, \eqref{den46}, \eqref{ent54} and \eqref{shock54}, one concludes that
\be\no
\|({\bf V}, V_6)\|_{\mathcal{X}}= \sum_{j=1}^5 \|V_j\|_{C^{2,\alpha}(\overline{\mathbb{D}})}+\|V_6\|_{C^{3,\alpha}(\overline{E})}\leq C_*(\epsilon +\epsilon \delta_0+\delta_0^2)\leq C_*(\epsilon +\delta_0^2).
\ee
Choose $\delta_0=\sqrt{\epsilon}$ and let $\epsilon<\epsilon_0=\frac{1}{4 C_*^2}$. Then $\|({\bf V}, V_6)\|_{\mathcal{X}}\leq 2C_*\epsilon\leq \delta_0$. Furthermore, the compatibility conditions \eqref{ber44}, \eqref{den461},\eqref{ent55}, \eqref{shock51} and \eqref{shock55} hold, thus $({\bf V}, V_6)\in \mathcal{X}$. We now can define the operator $\mathcal{T}:(\hat{{\bf V}}, \hat{V}_6)\mapsto ({\bf V}, V_6)$ which maps $\mathcal{X}$ to itself.

{\bf Step 6.} The contraction of the operator $\mathcal{T}$. It remains to prove that the operator $\mathcal{T}$ is a contraction in the norm
\be\no
\|({\bf V}, V_6)\|_w:=\sum_{j=1}^5 \|V_j\|_{C^{1,\alpha}(\overline{\mathbb{D}})}+\|V_6\|_{C^{2,\alpha}(\overline{E})},
\ee
so that one can find a unique fixed point to the operator $\mathcal{T}$ by the contraction mapping theorem. For any two elements $(\hat{{\bf V}}^i, \hat{V}_6^i), i=1,2$, define $({\bf V}^i, V_6^i)=\mathcal{T}(\hat{{\bf V}}^i, \hat{V}_6^i)$ for $i=1,2$. Denote
\be\no
(\hat{{\bf Z}}, \hat{Z}_6)=(\hat{{\bf V}}^1, \hat{V}_6^1)-(\hat{{\bf V}}^2, \hat{V}_6^2), \ \ ({\bf Z}, Z_6)=({\bf V}^1, V_6^1)-({\bf V}^2, V_6^2).
\ee

It follows from \eqref{ber41} that $Z_5$ satisfies
\be\no\begin{cases}
\bigg(D_1^{\hat{V}_6^1}+\frac{\hat{V}_2^1}{\bar{U}(D_0^{\hat{V}_6^1})+\hat{V}_1^1}D_2^{\hat{V}_6^1}+\frac{\hat{V}_3^1}{\bar{U}(D_0^{\hat{V}_6^1})
+\hat{V}_1^1}D_3^{\hat{V}_6^1}\bigg)Z_5=F_5(y),\\
Z_5(r_s,y')=h_5(y'),
\end{cases}\ee
where
\be\no
&&F_5(y)=-(D_1^{\hat{V}_6^1}-D_1^{\hat{V}_6^2})V_5^2-\sum_{j=2}^3\left(\frac{\hat{V}_j^1}{\bar{U}(D_0^{\hat{V}_6^1})+\hat{V}_1^1}D_j^{\hat{V}_6^1}-
\frac{\hat{V}_j^2}{\bar{U}(D_0^{\hat{V}_6^2})+\hat{V}_1^2}D_j^{\hat{V}_6^2}\right)V_5^2,\\\no
&&h_5(y')=B^-(r_s+\hat{V}_6^1(y'),y')-B^-(r_s+\hat{V}_6^2(y'),y').
\ee
Let $(\tau, \bar{y}_2^i(\tau;y),\bar{y}_3^i(\tau;y))$ be the trajectory associated with the vector field $(1, K_2^i, K_3^i)$ for $i=1,2$ respectively, where $K_2^i, K_3^i$ are defined as in \eqref{char} with $(\hat{{\bf V}}, \hat{V}_6)$ replaced by $(\hat{{\bf V}}^i, \hat{V}_6^i)$. Then
\be\no
Z_5(y_1,y')=h_5(\beta_2^1(y),\beta_3^1(y))+ \int_{r_s}^{y_1} F_5(\tau,\bar{y}_2^1(\tau;y),\bar{y}_3^1(\tau;y)) d\tau.
\ee
Thus one can get that
\be\no
&&\|Z_5\|_{C^{1,\alpha}(\overline{\mathbb{D}})}\leq C_* \epsilon \|\hat{Z}_6\|_{C^{1,\alpha}(\overline{E})}+ C_*\|(\hat{{\bf V}}^2, \hat{V}_6^2)\|_{\mathcal{X}}(\sum_{j=1}^3\|\hat{Z}_j\|_{C^{1,\alpha}(\overline{\mathbb{D}})}+\|\hat{Z}_6\|_{C^{2,\alpha}(\overline{\mathbb{D}})})\\\label{ber62}
&&\leq C_*(\epsilon + \delta_0)\|(\hat{{\bf Z}}, \hat{Z}_6)\|_w.
\ee

Next, we estimate the difference of the vorticities. Denote the vorticity associated with $(\hat{{\bf V}}^i, \hat{V}_6^i)$ by $(\tilde{\omega}_1^i, \tilde{\omega}_2^i, \tilde{\omega}_3^i)$ for $i=1,2$ and set
\be\no
J_1=\tilde{\omega}_1^1-\tilde{\omega}_1^2, \ J_2=\tilde{\omega}_2^1-\tilde{\omega}_2^2,\ \ J_3=\tilde{\omega}_3^1-\tilde{\omega}_3^2.
\ee
Then \eqref{vor501} implies that
\be\no\begin{cases}
\bigg(D_1^{\hat{V}_6^1}+\sum_{j=2}^3\frac{\hat{V}_j^1}{\bar{U}(D_0^{\hat{V}_6^1})+\hat{V}_j^1}D_j^{\hat{V}_6^1}\bigg)J_1+\mu(\hat{{\bf V}}^1,\hat{V}_6^1) J_1=F_6(y),\\
J_1(r_s,y')=h_6(y'),
\end{cases}\ee
where
\be\no
&&F_6(y)=-(D_1^{\hat{V}_6^1}-D_1^{\hat{V}_6^2})\tilde{\omega}_1^2-\sum_{j=2}^3\b(\frac{\hat{V}_j^1}{\bar{U}(D_0^{\hat{V}_6^1})+\hat{V}_1^1}D_j^{\hat{V}_6^1}-
\frac{\hat{V}_j^2}{\bar{U}(D_0^{\hat{V}_6^2})+\hat{V}_1^2}D_j^{\hat{V}_6^2}\b)\tilde{\omega}_1^2\\\no
&&\quad\quad-\b(\mu(\hat{{\bf V}}^1,\hat{V}_6^1)-\mu(\hat{{\bf V}}^2,\hat{V}_6^2)\b)\tilde{\omega}_1^2+ H_0(\hat{{\bf V}}^1,\hat{V}_6^1)-H_0(\hat{{\bf V}}^2,\hat{V}_6^2),\\\no
&&h_6(y')=R_6(\hat{{\bf V}}^1,\hat{V}_6^1)-R_6(\hat{{\bf V}}^2,\hat{V}_6^2).
\ee
Similar to the previous analysis, there holds
\be\no
&&\|J_1\|_{C^{\alpha}(\overline{\mathbb{D}})}\leq C_* (\epsilon + \|\tilde{\omega}_1^2\|_{C^{1,\alpha}(\overline{\mathbb{D}})}+\|(\hat{{\bf V}}^2,\hat{V}_6^2)\|_{\mathcal{X}})\|(\hat{{\bf Z}}, \hat{Z}_6)\|_w\\\no
&&\leq C_*(\epsilon + \delta_0)\|(\hat{{\bf Z}}, \hat{Z}_6)\|_w.
\ee

Now we can turn to the estimate of $Z_i, i=1,2,3$. It follows from the definition and \eqref{den32} that
\be\label{den62}\begin{cases}
d_1(y_1)\p_{y_1} Z_1+\frac{1}{y_1}\p_{y_2} Z_2 +\p_{y_3} Z_3+ \frac{Z_1}{y_1}+ d_2(y_1) Z_1=F_0(y),\ \ &\text{in }\mathbb{D},\\
\frac{1}{y_1}\p_{y_2} Z_3- \p_{y_3} Z_2+\p_{y_1} \Upsilon=F_1(y),\ \ &\text{in }\mathbb{D},\\
\p_{y_3} Z_1-\p_{y_1} Z_3+d_3(y_1)\p_{y_3} Z_1(r_s,y')+ \frac1{y_1}\p_{y_2}\Upsilon= F_2(y),\ \ &\text{in }\mathbb{D},\\
\p_{y_1} Z_2-\frac{1}{y_1}\p_{y_2} Z_1+\f{V_2}{y_1}-\frac{d_3(y_1)}{y_1}\p_{y_2} Z_1(r_s,y')+\p_{y_3}\Upsilon = F_3(y),\ \ &\text{in }\mathbb{D},\\
(\frac{1}{r_s^2}\p_{y_2}^2+\p_{y_3}^2) Z_1(r_s,y')=a_0 a_1 (\frac{1}{r_s}\p_{y_2} Z_2+\p_{y_3}Z_3)(r_s,y')+h_1(y'),\ \ &\text{on }E,\\
Z_2(y_1,\pm\theta_0,y_3)=\Upsilon(y_1,\pm \theta_0, y_3)=0,\ \ &\text{on }\Sigma_2^{\pm},\\
\Upsilon(r_s,y')=\Upsilon(r_2,y')=0,\ \ &\text{on }E,\\
Z_3(y_1,y_2,\pm 1)=\Upsilon(y_1,y_2,\pm 1)= 0,\ \ &\text{on }\Sigma_3^{\pm}, \\
Z_1(r_2,y')+ d_3(r_2) Z_1(r_s,y')=h_2(y'),\ \ &\text{on }E,
\end{cases}\ee
where
\be\no
&&F_0(y)=-\frac{(\gamma-1)(\bar{U}'+\frac{\bar{U}}{r})}{c^2(\bar{\rho},\bar{K})} Z_5+G_0(\hat{{\bf V}}^1,\hat{V}_6^1)-G_0(\hat{{\bf V}}^2,\hat{V}_6^2),\ \ \Upsilon=\Pi^1-\Pi^2,\\\no
&&F_1(y)=J_1+H_1(\hat{{\bf V}}^1,\hat{V}_6^1)-H_1(\hat{{\bf V}}^2,\hat{V}_6^2),\\\no
&&F_2(y)=\frac{\hat{V}_2^1 J_1+ D_3^{V_6^1} Z_5}{\bar{U}(D_0^{\hat{V}_6^1})+\hat{V}_1^1}+\b(\frac{\hat{V}_2^1}{\bar{U}(D_0^{\hat{V}_6^1})+\hat{V}_1^1}-\frac{\hat{V}_2^2 }{\bar{U}(D_0^{\hat{V}_6^2})+\hat{V}_1^2}\b)\tilde{\omega}_1^2 + H_2(\hat{{\bf V}}^1,\hat{V}_6^1)-H_2(\hat{{\bf V}}^2,\hat{V}_6^2)\\\no
&&\quad+\b(\frac{D_3^{V_6^1} V_5^2}{\bar{U}(D_0^{\hat{V}_6^1})+\hat{V}_1^1}-\frac{D_3^{V_6^2} V_5^2}{\bar{U}(D_0^{\hat{V}_6^2})+\hat{V}_1^2}\b)+\frac{\bar{B}-\frac{1}{2}\bar{U}^2(y_1)}{\gamma \bar{K}\bar{U}(y_1)}\p_{y_3}(R_4^1-R_4^2),\\\no
&&F_3(y)=\frac{\hat{V}_2^1 J_1-D_2^{V_6^1} Z_5}{\bar{U}(D_0^{\hat{V}_6^1})+\hat{V}_1^1}+\b(\frac{\hat{V}_3^1}{\bar{U}(D_0^{\hat{V}_6^1})+\hat{V}_1^1}-\frac{\hat{V}_3^2 }{\bar{U}(D_0^{\hat{V}_6^2})+\hat{V}_1^2}\b)\tilde{\omega}_1^2 + H_3(\hat{{\bf V}}^1,\hat{V}_6^1)-H_3(\hat{{\bf V}}^2,\hat{V}_6^2)\\\no
&&\quad -\b(\frac{D_2^{V_6^1} V_5^2}{\bar{U}(D_0^{\hat{V}_6^1})+\hat{V}_1^1}-\frac{D_2^{V_6^2} V_5^2}{\bar{U}(D_0^{\hat{V}_6^2})+\hat{V}_1^2}\b)-\frac{\bar{B}-\frac{1}{2}\bar{U}^2(y_1)}{\gamma \bar{K}\bar{U}(y_1)}\frac{1}{y_1}\p_{y_2}(R_4^1-R_4^2),\\\no
&&h_1(y')= q_1(\hat{{\bf V}}^1(r_s,y'),\hat{V}_6^1(y'))-q_1(\hat{{\bf V}}^2(r_s,y'),\hat{V}_6^2(y')),\\\no
&&h_2(y')= q_4(\hat{{\bf V}}^1(r_s,y'),\hat{V}_6^1(y'))-q_4(\hat{{\bf V}}^2(r_s,y'),\hat{V}_6^2(y')).
\ee
Furthermore, the system \eqref{den62} should be supplemented with
\be\no
&&\b(\frac{1}{r_s}\p_{y_2}Z_1-a_0a_1Z_2\b)(r_s,\pm \theta_0,y_3)=0,\ \forall y_3\in [-1,1],\\\label{shock63}
&&(\p_{y_3} Z_1-a_0 a_1 Z_3)(r_s,y_2,\pm 1)= 0,\ \ \forall y_2\in [-\theta_0,\theta_0].
\ee

By similar analysis in {\bf Step 4}, one can obtain
\be\label{den63}
&&\sum_{j=1}^3\|Z_j\|_{C^{1,\alpha}(\overline{\mathbb{D}})}\leq C_*(\sum_{j=0}^3 \|F_j\|_{C^{\alpha}(\overline{\mathbb{D}})}+\sum_{j=1}^2\|h_j\|_{C^{1,\alpha}(\overline{E})})\\\no
&&\leq C_*(\epsilon+\sum_{j=1}^2\|(\hat{{\bf V}}^j,\hat{V}_6^j)\|_{\mathcal{X}})\|(\hat{{\bf Z}}, \hat{Z}_6)\|_w\leq C_*(\epsilon+\delta_0)\|(\hat{{\bf Z}}, \hat{Z}_6)\|_w.
\ee
Due to \eqref{ent53}, $Z_4$ can be expressed as
\be\no
&&Z_4(y_1,y')=\frac{a_2}{a_1} Z_1(r_s,y')+R_4(\hat{{\bf V}}^1(r_s,\beta_2^1(y),\beta_3^1(y)),\hat{V}_6^1(\beta_2^1(y),\beta_3^1(y)))\\\no
&&\quad\quad-R_4(\hat{{\bf V}}^2(r_s,\beta_2^2(y),\beta_3^2(y)),\hat{V}_6^2(\beta_2^2(y),\beta_3^2(y)))
\ee
It can be checked that there is a term $\hat{V}_6(\beta_2(y),\beta_3(y))-\hat{V}_6(y')$ in $R_4$ which needs a further analysis:
\be\no
&&\|(\hat{V}_6^1(\beta_2^1(y),\beta_3^1(y))-\hat{V}_6^1(y'))-(\hat{V}_6^2(\beta_2^2(y),\beta_3^2(y))-\hat{V}_6^2(y'))\|_{C^{1,\alpha}(\overline{\mathbb{D}})}\\\no
&&\leq \|\hat{Z}_6(\beta_2^1(y),\beta_3^1(y))-\hat{Z}_6(y')\|_{C^{1,\alpha}(\overline{\mathbb{D}})}
+\|\hat{V}_6^2(\beta_2^1(y),\beta_3^1(y))-\hat{V}_6^2(\beta_2^2(y),\beta_3^2(y))\|_{C^{1,\alpha}(\overline{\mathbb{D}})}\\\no
&&\leq C_*\sum_{j=2}^3\|\beta_j^1-y_j\|_{C^{1,\alpha}(\overline{\Omega})}\|\hat{Z}_6\|_{C^{2,\alpha}(\overline{E})}
+\|\hat{V}_6^2\|_{C^{2,\alpha}(\overline{E})}\sum_{j=2}^3\|\beta_j^1-\beta_j^2\|_{C^{1,\alpha}(\overline{\mathbb{D}})}.
\ee

To estimate $\|\beta_j^1-\beta_j^2\|_{C^{1,\alpha}(\overline{\mathbb{D}})}$, one denotes $Y_j(\tau;y)=\bar{y}_j^1(\tau;y)-\bar{y}_j^2(\tau;y)$ for $j=2,3$, so that $Y_j(r_s;y)=\beta_j^1(y)-\beta_j^2(y)$. It follows from \eqref{char1} that
\be\no\begin{cases}
\frac{d}{d\tau} Y_2(\tau;y)= a_{22}(\tau;y) Y_2(\tau;y)+ a_{23}(\tau;y)Y_3(\tau;y)+ b_2(\tau;y),\\
\frac{d}{d\tau} Y_3(\tau;y)= a_{32}(\tau;y) Y_2(\tau;y)+ a_{33}(\tau;y)Y_3(\tau;y)+ b_3(\tau;y),\\
Y_2(y_1;y)=Y_3(y_1;y)=0,
\end{cases}\ee
where $a_{ij}, i,j=2,3$ are functions of $\hat{{\bf V}}^1,\hat{V}_6^1$ and $b_i, i=2,3$ can be expressed as functions of $\hat{{\bf Z}}, \hat{Z}_6$.

Then
\be\label{char12}\begin{cases}
Y_2(t;y)= \int_{y_1}^t  a_{22}(\tau;y) Y_2(\tau;y)+ a_{23}(\tau;y)Y_3(\tau;y)d\tau+ \int_{y_1}^t b_2(\tau;y)d\tau,\\
Y_3(t;y)= \int_{y_1}^t  a_{32}(\tau;y) Y_2(\tau;y)+ a_{33}(\tau;y)Y_3(\tau;y)d\tau+ \int_{y_1}^t b_3(\tau;y)d\tau.
\end{cases}\ee
Define $Y(t;y)=\max_{t\leq s\leq y_1} (|Y_2(s;y)|+|Y_3(s;y)|)$ and
$$a(t;y)=\max_{t\leq s\leq y_1} \sum_{i,j=2}^3|a_{ij}(s;y)|, \ \ b(t;y)=\max_{t\leq s\leq y_1} (|b_2(s;y)|+|b_3(s;y)|).$$
It then follows from \eqref{char12} that
\be\no
Y(t;y)\leq \int_{y_1}^t a(\tau) Y(\tau;y) d\tau + \int_{y_1}^t b(\tau;y) d\tau.
\ee
Then the Gronwall's inequality yields
\be\no
\sum_{j=2}^3\|\beta_j^1-\beta_j^2\|_{C^0(\overline{\mathbb{D}})}\leq C_*\sum_{j=1}^5(\|\hat{Z}_j\|_{C^0(\overline{\mathbb{D}})}+\|\hat{Z}_6\|_{C^1(\overline{\mathbb{D}})}).
\ee
Similarly, one can derive further that
\be\no
\sum_{j=2}^3\|\beta_j^1-\beta_j^2\|_{C^{1,\alpha}(\overline{\mathbb{D}})}\leq C_*\|(\hat{{\bf Z}},\hat{Z}_6)\|_{w}.
\ee
Hence it holds that
\be\label{ent61}
\|Z_4\|_{C^{1,\alpha}}\leq C_*(\|Z_1(r_s,\cdot)\|_{C^{1,\alpha}(\overline{E})}+\|R_4^1-R_4^2\|_{C^{1,\alpha}})\leq C_*(\epsilon+\delta_0)\|(\hat{{\bf Z}}, \hat{Z}_6)\|_w.
\ee

Finally, it remains to estimate $Z_6$. It follows from \eqref{shock50} that
\be\no
Z_6(y')=\frac{1}{a_1} Z_1(r_s,y')-\frac{1}{a_1}(R_1(\hat{{\bf V}}^1(r_s,y'),\hat{V}_6^1(y'))-R_1(\hat{{\bf V}}^2(r_s,y'),\hat{V}_6^2(y'))),
\ee
from which one may infer that
\be\no
\|Z_6\|_{C^{1,\alpha}(\overline{E})}&\leq& C_*(\|Z_1(r_s,\cdot)\|_{C^{1,\alpha}(\overline{E})}+\|R_1(\hat{{\bf V}}^1(r_s,\cdot),\hat{V}_6^1)-R_1(\hat{{\bf V}}^2(r_s,\cdot),\hat{V}_6^2)\|_{C^{1,\alpha}(\overline{E})})\\\label{shock61}
&\leq& C_*(\epsilon+\delta_0)\|(\hat{{\bf Z}}, \hat{Z}_6)\|_w.
\ee
Furthermore, it follows from \eqref{shock53} that
\be\no\begin{cases}
\frac{1}{r_s}\p_{y_2} Z_6(y')= a_0 Z_2(r_s,y') + (g_2(\hat{{\bf V}}^1(r_s,y'),\hat{V}^1_6(y'))-g_2(\hat{{\bf V}}^2(r_s,y'),\hat{V}^2_6(y'))), \ \ &\text{in }E,\\
\p_{y_3} Z_6(y')= a_0 Z_3(r_s,y') + (g_3(\hat{{\bf V}}^1(r_s,y'),\hat{V}^1_6(y'))-g_3(\hat{{\bf V}}^2(r_s,y'),\hat{V}^2_6(y'))), \ \ &\text{in }E.
\end{cases}\ee
Then one can conclude that
\be\no
\|(\nabla_{y'}Z_6)\|_{C^{1,\alpha}(\overline{E})}&\leq& C_*\sum_{j=2}^3(\|Z_j(r_s,\cdot)\|_{C^{1,\alpha}(\overline{E})}+\|g_j(\hat{{\bf V}}^1(r_s,\cdot),\hat{V}_6^1)-g_j(\hat{{\bf V}}^2(r_s,\cdot),\hat{V}_6^2)\|_{C^{1,\alpha}(\overline{E})})\\\label{shock63}
&\leq& C_*(\epsilon+\delta_0)\|(\hat{{\bf Z}}, \hat{Z}_6)\|_w.
\ee

Collecting all the estimates \eqref{ber62},\eqref{den63}, \eqref{ent61},\eqref{shock61} and \eqref{shock63} leads to
\be\no
\|({\bf Z}, Z_6)\|_w\leq C_*(\epsilon+\delta_0)\|(\hat{{\bf Z}}, \hat{Z}_6)\|_w.
\ee
Since $\delta_0=\sqrt{\epsilon}$, if $\epsilon<\epsilon_0=\frac{1}{16C_*^2}$, then $\|({\bf Z}, Z_6)\|_w\leq \frac{1}{2}\|(\hat{{\bf Z}}, \hat{Z}_6)\|_w$ so that the mapping $\mathcal{T}$ is a contraction operator in the weak norm $\|\cdot\|_{w}$. Thus there exists a unique fixed point $({\bf V},V_6)\in \mathcal{X}$ such that $\mathcal{T}({\bf V},V_6)=({\bf V},V_6)$. Let us recall the auxiliary function $\Pi$ that is associated with the fixed point $({\bf V},V_6)$ in solving the problem \eqref{den32} with \eqref{shock22}-\eqref{shock23}. To finish the proof of Theorem \ref{main}, we still need to prove that $\Pi\equiv 0$ in $\mathbb{D}$. Thanks to the definitions of $G_j({\bf V},V_6)$ for $j=1,2,3$, one may infer from \eqref{den32} that
\be\no\begin{cases}
-\p_{y_1}\Pi= D_2^{V_6} V_3- D_3^{V_6} V_2 -\tilde{\omega}_1,\\
-\frac{1}{y_1}\p_{y_2}\Pi= D_3^{V_6} V_1- D_1^{V_6} V_3- \frac{V_2\tilde{\omega}_1+ D_3^{V_6} V_5}{\bar{U}(D_0^{V_6})+V_1}+\frac{\bar{B}+V_5-\frac{1}{2}(\bar{U}(D_0^{V_6})+V_1)^2-\frac{1}{2}(V_2^2+V_3^2)}{\gamma(\bar{K}+V_4)(\bar{U}(D_0^{V_6})+V_1)}D_3^{V_6} V_4,\\
-\p_{y_3}\Pi= D_1^{V_6} V_2+\frac{V_2}{D_0^{V_6}}- D_2^{V_6} V_1- \frac{V_3\tilde{\omega}_1-D_2^{V_6} V_5}{\bar{U}(D_0^{V_6})+V_1}+\frac{\bar{B}+V_5-\frac{1}{2}(\bar{U}(D_0^{V_6})+V_1)^2-\frac{1}{2}(V_2^2+V_3^2)}{\gamma(\bar{K}+V_4)(\bar{U}(D_0^{V_6})+V_1)}D_2^{V_6} V_4.
\end{cases}\ee
Since $\tilde{\omega}_1$ satisfies \eqref{vor400} and the following commutator relations hold
\be\no
D_1^{V_6}D_2^{V_6}-D_2^{V_6}D_1^{V_6}=-\frac{1}{D_0^{V_6}} D_2^{V_6},\ \ D_2^{V_6}D_3^{V_6}=D_3^{V_6}D_2^{V_6}, \ \ D_1^{V_6}D_3^{V_6}=D_3^{V_6}D_1^{V_6},
\ee
one can conclude that
\be\no
-D_1^{V_6}(\p_{y_1}\Pi)-\frac{1}{D_0^{V_6}}(\p_{y_1}\Pi)- D_2^{V_6}(\frac{1}{y_1}\p_{y_2}\Pi)-D_3^{V_6}(\p_{y_3}\Pi)=0,\ \ \text{in }\mathbb{D}.
\ee
Since $\|V_6\|_{C^{3,\alpha}(\overline{E})}\leq \delta_0$, where $\delta_0$ is sufficiently small, thus $\Pi$ satisfies a second order uniformly elliptic equation without zeroth order term. Thanks to $\Pi=0$ on $\partial\mathbb{D}$, it follow directly from the maximum principle that $\Pi\equiv0$ in $\mathbb{D}$. Thus $({\bf V},V_6)$ is the desired solution. The proof of Theorem \ref{existence} is completed.

\section{Appendix: The compatibility conditions} \noindent

In this appendix, we verify the compatibility conditions for the supersonic flow and also the one on the intersection of the shock front and the cylinder walls.

\begin{proof}[Proof of Lemma \ref{supersonic}.] We prove only \eqref{super5}. The slip boundary condition \eqref{slip1} implies
\be\no\begin{cases}
\p_{\theta} P^-(r,\pm\theta_0,x_3)= 0,\ \ \forall r\in [r_1,r_2],x_3\in [-1,1],\\
\p_{x_3} P^-(r,\theta,\pm 1)= 0,\ \ \forall r\in [r_1,r_2],\theta\in [-\theta_0,\theta_0].
\end{cases}\ee

Differentiating the second, fourth and fifth equations in \eqref{euler-cyl} with respect to $\theta$, and restricting the resulting equations on the surface $\theta=\pm \theta_0$, one obtains
\be\no\begin{cases}
(U_1^-\p_r+U_3^-\p_{x_3})\p_{\theta} U_1^- + (\p_r U_1^-+\frac{\p_{\theta} U_2^-}{r}) \p_{\theta} U_1^- + \p_{x_3} U_1^- \p_{\theta} U_3^- -\frac{\p_r P^-}{\rho^2}\frac{\p \rho^-}{\p K} \p_{\theta} K^-=0,\\
(U_1^-\p_r+U_3^-\p_{x_3})\p_{\theta} U_3^- + \p_r U_3^-\p_{\theta} U_1^-+ (\frac{\p_{\theta} U_2^-}{r}+ \p_{x_3} U_3^-)\p_{\theta} U_3^- -\frac{\p_{x_3} P^-}{\rho^2} \frac{\p \rho^-}{\p K}\p_{\theta} K^-=0,\\
(U_1^-\p_r+U_3^-\p_{x_3})\p_{\theta} K^-+ \p_r K^- \p_{\theta} U_1^- + \p_{x_3} K^- \p_{\theta} U_3^- + \frac{\p_{\theta} U_2^-}{r}\p_{\theta} K^-=0.
\end{cases}\ee
This is a homogeneous system of transport equations. Due to the compatibility condition \eqref{super3} at the entrance $r=r_1$, one has
\be\no
(\p_{\theta} U_1^-,\p_{\theta} U_3^-, \p_{\theta} K^-)(r,\pm \theta_0, x_3)=0,\ \ \forall r\in [r_1,r_2], x_3\in [-1,1].
\ee

Differentiating the first equation in \eqref{euler-cyl} with respect to $\theta$, and restricting the resulting equations on the surface $\theta=\pm \theta_0$, one gets
\be\no
\p_{\theta}^2 U_2^-(r,\pm \theta_0,x_3)=0,\ \ \forall r\in [r_1,r_2], x_3\in [-1,1].
\ee

Similarly, differentiating the second, third and fifth equations in \eqref{euler-cyl} with respect to $x_3$, and restricting the resulting equations on the surface $x_3=\pm 1$, one obtains
\be\no\begin{cases}
(U_1^-\p_r+\frac{U_2^-}{r}\p_{\theta})\p_{x_3} U_1^- + (\p_r U_1^-+\p_{x_3} U_3^-) \p_{x_3} U_1^- \\
\q+(\frac{1}{r}\p_{\theta} U_1^--\frac{2 U_2^-}{r}) \p_{x_3} U_2^-  -\frac{\p_r P^-}{\rho^2}\frac{\p \rho^-}{\p K} \p_{x_3} K^-=0,\\
(U_1^-\p_r+\frac{U_2^-}{r}\p_{\theta})\p_{x_3} U_3^- + (\p_r U_2^-+\frac{U_2^-}{r})\p_{x_3} U_1^-\\
\q+ (\frac{\p_{\theta} U_2^-+ U_1^-}{r}+ \p_{x_3} U_3^-)\p_{x_3} U_2^- -\frac{\p_{\theta} P^-}{r\rho^2} \frac{\p \rho^-}{\p K}\p_{x_3} K^-=0,\\
(U_1^-\p_r+\frac{U_2^-}{r}\p_{\theta})\p_{x_3} K^-+ \p_r K^- \p_{x_3} U_1^- + \frac{1}{r}\p_{\theta} K^- \p_{x_3} U_2^- + \p_{x_3} U_3^-\p_{x_3} K^-=0,
\end{cases}\ee
which imply that
\be\no
(\p_{x_3} U_1^-,\p_{x_3} U_2^-, \p_{x_3} K^-)(r,\theta, \pm 1)=0,\ \ \forall r\in [r_1,r_2], \theta\in [-\theta_0,\theta_0].
\ee

Finally, differentiating the first equation in \eqref{euler-cyl} with respect to $x_3$, and restricting the resulting equations on the surface $x_3=\pm 1$ show that
\be\no
\p_{x_3}^2 U_3^-(r,\theta,\pm 1)=0,\ \ \forall r\in [r_1,r_2], \theta\in [-\theta_0,\theta_0].
\ee
\end{proof}

We further prove that the compatibility conditions \eqref{super5} is preserved when the supersonic flow moves across the shock. This is proved under the assumption that the existence of a transonic shock solution with $C^{2,\alpha}(\overline{\Omega^+})$ regularity in subsonic region and the $C^{3,\alpha}(\overline{E})$ regularity of the shock front.

\begin{lemma}\label{l41}({\bf Compatibility conditions on the intersection of the shock front and the nozzle wall.})
{\it Suppose that the supersonic incoming flow satisfies the compatibility conditions \eqref{super5}. Assume further that the system \eqref{euler-cyl}, \eqref{slip1}-\eqref{pressure} and \eqref{rh} has a piecewise smooth solution $(U_1^{\pm},U_2^{\pm},U_3^{\pm},P^{\pm},K^{\pm})$ defined on $\Omega^{\pm}$ respectively and the shock front $r=\xi(\theta,x_3)$ with the properties $(U_1^+,U_2^+,U_3^+,P^+,K^+)\in C^{2,\alpha}(\overline{\Omega^+})$ and $\xi\in C^{3,\alpha}(\overline{E})$. Then the following compatibility conditions hold on the intersection of the shock front and the cylinder wall
\be\label{c1}\begin{cases}
(U_{2}^+,\p_{\theta}^2U_{2}^+,\p_{\theta}(U_{1}^+,U_{3}^+, P^+,K^+))(r,\pm\theta_0, x_3)=0, \forall \xi(\pm\theta_0,x_3)\leq r\leq r_2, x_3\in [-1,1],\\
(U_{3}^+,\p_{x_3}^2U_{3}^+,\p_{x_3}(U_{1}^+,U_{2}^+, P^+,K^+))(r,\theta, \pm 1)=0, \forall \xi(\theta,\pm 1)\leq r\leq r_2,\theta\in[-\theta_0,\theta_0],\\
\p_{\theta}\xi(\pm\theta_0,x_3)=\p_{\theta}^3\xi(\pm\theta_0,x_3)=0,\ \text{on } x_3\in [-1,1],\\
\p_{x_3}\xi(\theta, \pm 1)=\p_{x_3}^3\xi(\theta, \pm 1)=0,\ \ \text{on } \theta\in [-\theta_0,\theta_0].
\end{cases}\ee

}\end{lemma}

\begin{proof}

It suffices to show that \eqref{c1} holds on the intersection of the shock front and the nozzle wall. Then the rest can be shown as for Lemma \ref{supersonic}.  It follows from the third equation in \eqref{rh} and \eqref{slip1} that
\be\label{perpen}\begin{cases}
\p_{\theta}\xi(\pm\theta_0, x_3)=0,\ \ \forall x_3\in [-1,1],\\
\p_{x_3}\xi(\theta, \pm 1)=0,\ \ \forall \theta\in [-\theta_0,\theta_0],\\
\end{cases}\ee
Substitute \eqref{perpen} into the last equation in \eqref{rh} to get
\be\label{shock-cyl1}\begin{cases}
\p_{x_3}\xi(\pm \theta_0, x_3) =\frac{[\rho U_1 U_3]}{[P+\rho U_3^2]},\ \ \forall x_3\in [-1,1],\\
\frac{\p_{\theta}\xi}{\xi}(\theta, \pm 1) =\frac{[\rho U_1 U_2]}{[P+\rho U_2^2]}(\theta,\pm 1),\ \ \forall \theta\in [-\theta_0,\theta_0].
\end{cases}\ee

Differentiating \eqref{rh} with respect to $\theta$ and restricting the resulting equations on $\theta=\pm \theta_0$, utilizing \eqref{slip1}, \eqref{perpen} and \eqref{shock-cyl1}, one obtains
\be\no\begin{cases}
\p_{\theta}(\rho^+ U_1^+)-\frac{[\rho U_1 U_3]}{[P+\rho U_3^2]} \p_{\theta}(\rho^+ U_3^+)=0,\\
\rho^+ U_1^+\p_{\theta} U_1^+-\frac{[\rho U_1 U_3]}{[P+\rho U_3^2]} \rho^+ U_3^+ \p_{\theta} U_1^+=0,\\
\rho^+ U_1^+\p_{\theta} U_3^+-\frac{[\rho U_1 U_3]}{[P+\rho U_3^2]}\rho^+ U_3^+\p_{\theta} U_3^+=0.
\end{cases}\ee

While differentiating \eqref{rh} with respect to $x_3$ and restricting the resulting equations on $x_3=\pm 1$, one can get by utilizing \eqref{slip1}, \eqref{perpen} and \eqref{shock-cyl1} that
\be\no\begin{cases}
\p_{x_3}(\rho^+ U_1^+)-\frac{[\rho U_1 U_2]}{[P+\rho U_2^2]} \p_{x_3}(\rho^+ U_2^+)=0,\\
\rho^+ U_1^+\p_{x_3} U_1^+-\frac{[\rho U_1 U_2]}{[P+\rho U_2^2]} \rho^+ U_2^+ \p_{x_3} U_1^+=0,\\
\rho^+ U_1^+\p_{x_3} U_2^+-\frac{[\rho U_1 U_2]}{[P+\rho U_2^2]}\rho^+ U_2^+\p_{x_3} U_1^+ =0.
\end{cases}\ee

Then
\be\no\begin{cases}
(\p_{\theta} U_1^+,\p_{\theta} U_3^+,\p_{\theta}\rho^+)(\xi(\pm\theta_0,x_3),\pm\theta_0, x_3)=0,\ \ \ \forall x_3\in [-1,1],\\
(\p_{x_3} U_1^+,\p_{x_3} U_2^+,\p_{x_3}\rho^+)(\xi(\theta,\pm 1),\theta, \pm 1)=0,\ \ \ \forall \theta\in [-\theta_0,\theta_0].
\end{cases}\ee
Then the first two conditions in \eqref{c1} follow as in the proof of Lemma \ref{supersonic}.

Differentiating the first (second) equation in \eqref{shock11} with respect to $\theta (x_3)$ twice and evaluating at $\theta=\pm \theta_0$ ($x_3\pm 1$ respectively) lead to
\be\no
&&\p_{\theta}^3\xi(\pm \theta_0,x_3)=0,\ \ \ \forall x_3\in [-1,1],\\\no
&&\p_{x_3}^3\xi(\theta,\pm 1)=0,\ \ \ \forall \theta\in [-\theta_0,\theta_0].
\ee
Thus Lemma \ref{l41} is verified.
\end{proof}

Finally, we give the explicit expressions of $J_i (i=2,3),J$ and $R_{0i}, i=1,2,3$ needed in \eqref{g21}-\eqref{shock400}, and \eqref{ent31}:
\be\no
&& J_2=\b(\tilde{\rho}({\bf V}(r_s,y'), V_6)V_3^2(r_s,y')+\tilde{P}({\bf V}(r_s,y'), V_6)-(\rho^-(U_3^-)^2+P^-)(r_s+V_6,y')\b)\\\no
&&\quad\times \b(\tilde{\rho}({\bf V}(r_s,y'), V_6)(\bar{U}(r_s+V_6)+V_1(r_s,y'))V_2(r_s,y')-(\rho^- U_1^- U_2^-)(r_s+V_6,y')\b)\\\no
&&\quad\quad-\b(\tilde{\rho}({\bf V}(r_s,y'), V_6)(\bar{U}(r_s+V_6)+V_1(r_s,y'))V_3(r_s,y')-(\rho^- U_1^- U_3^-)(r_s+V_6,y')\b)\\\label{j21}
&&\quad\times \b(\tilde{\rho}({\bf V}(r_s,y'), V_6)(V_2V_3)(r_s,y')-(\rho^-U_2^-U_3^-)(r_s+V_6,y')\b),\\\no
&& J_3=\b(\tilde{\rho}({\bf V}(r_s,y'), V_6)V_2^2(r_s,y')+\tilde{P}({\bf V}(r_s,y'), V_6)-(\rho^-(U_2^-)^2+P^-)(r_s+V_6,y')\b)\\\no
&&\quad\times\b(\tilde{\rho}({\bf V}(r_s,y'), V_6)(\bar{U}(r_s+V_6)+V_1(r_s,y'))V_3(r_s,y')-(\rho^- U_1^- U_3^-)(r_s+V_6,y')\b)\\\no
&&\quad\quad-\b(\tilde{\rho}({\bf V}(r_s,y'), V_6)(\bar{U}(r_s+V_6)+V_1(r_s,y'))V_2(r_s,y')-(\rho^- U_1^- U_2^-)(r_s+V_6,y')\b)\\\no
&&\quad\times \b(\tilde{\rho}({\bf V}(r_s,y'), V_6)(V_2V_3)(r_s,y')-(\rho^-U_2^-U_3^-)(r_s+V_6,y')\b),\\\no
&& J=\b(\tilde{\rho}({\bf V}(r_s,y'), V_6)V_2^2(r_s,y')+\tilde{P}({\bf V}(r_s,y'), V_6)-(\rho^-(U_2^-)^2+P^-)(r_s+V_6,y')\b)\\\no
&&\quad\times\b(\tilde{\rho}({\bf V}(r_s,y'), V_6)V_3^2(r_s,y')+\tilde{P}({\bf V}(r_s,y'), V_6)-(\rho^-(U_3^-)^2+P^-)(r_s+V_6,y')\b)\\\label{j1}
&&\quad\quad-\b(\tilde{\rho}({\bf V}(r_s,y'), V_6)(V_2V_3)(r_s,y')-(\rho^- U_2^- U_3^-)(r_s+V_6,y')\b)^2,
\ee
and
\be\no
&&R_{01}=-[\bar{\rho} \bar{U}](r_s+V_6)+(\rho^- U_1^-)(r_s+V_6,y')-(\bar{\rho}^-\bar{U}^-)(r_s+V_6)
\\\label{rv01}
&&-(V_1(r_s,y')+\bar{U}^+(r_s+V_6)-\bar{U}^+(r_s))(\tilde{\rho}({\bf V}(r_s,y'),V_6)-\bar{\rho}^+(r_s+V_6))\\\no
&&+ \sum_{i=2}^3 (\tilde{\rho}({\bf V}(r_s,y'),V_6)V_i(r_s,y')-(\rho^- U_i^-)(r_s+V_6,y'))\frac{J_i}{J}-(\bar{\rho}^+(r_s+V_6)-\bar{\rho}^+(r_s))V_1(r_s,y'),\\\no
&&R_{02}=-\b\{[\bar{\rho} \bar{U}^2+ \bar{P}](r_s +V_6)-\frac{1}{r_s}[\bar{P}(r_s)]V_6\b\}+ (\rho^- (U_1^-)^2 +P^-)(r_s+V_6,y')\\\no
&&-(\bar{\rho}^-(\bar{U}^-)^2+\bar{P}^-)(r_s+V_6)-\bigg\{\tilde{\rho}({\bf V}(r_s,y'),V_6)(\bar{U}(r_s+V_6)+V_1(r_s,y'))^2 \\\no
&&+\tilde{P}({\bf V}(r_s,y'),V_6)-(\bar{\rho}^+(\bar{U}^+)^2+\bar{P}^+)(r_s+V_6)-2(\bar{\rho}^+ \bar{U}^+)(r_s)V_1(r_s,y')\\\no
&&- \{(\bar{U}^+(r_s))^2+c^2(\bar{\rho}^+(r_s),\bar{K}^+)\}(\tilde{\rho}({\bf V}(r_s,y'),V_6)-\bar{\rho}^+(r_s+V_6))-(\bar{\rho}^+(r_s))^{\gamma} V_4(r_s,y')\bigg\}
\\\no
&&+\sum_{i=2}^3 \b(\tilde{\rho}({\bf V}(r_s,y'),V_6)(\bar{U}(r_s+V_6)+V_1(r_s,y'))V_i(r_s,y')-(\rho^- U_1^- U_i^-)(r_s+V_6,y')\b)\frac{J_i}{J},\\\no
&&R_{03}= B^-(r_s+ V_6,y')- \bar{B}^-- \bar{U}^+(r_s+V_6) V_1(r_s,y')-\frac{1}{2}\sum_{j=1}^3 V_j^2(r_s,y')+ \bar{U}^+(r_s) V_1(r_s,y')\\\label{rv03}
&&-\frac{\gamma}{\gamma-1}\b((\bar{K}^+ +V_4(r_s,y'))(\tilde{\rho}({\bf V}(r_s,y'), V_6))^{\gamma-1}-\bar{K}^+ (\bar{\rho}^+(r_s+V_6))^{\gamma-1}\b)\\\no
&& + \frac{c^2(\bar{\rho}^+(r_s),\bar{K}^+)}{\bar{\rho}^+(r_s)} (\tilde{\rho}({\bf V}(r_s,y'), V_6)-\bar{\rho}^+(r_s+V_6))+ \frac{\ga (\bar{\rho}^+(r_s))^{\gamma-1}}{(\ga-1)} V_4(r_s,y').
\ee

\section*{Acknowledgment}
Weng is supported by National Natural Science Foundation of China 12071359, 12221001. Xin is supported in part by the Zheng Ge Ru Foundation, Hong Kong RGC Earmarked Research Grants CUHK-14301421, CUHK-14300917, CUHK-14302819 and CUHK-14300819, the key project of NSFC (Grant No. 12131010) and by Guangdong Basic and Applied Basic Research Foundation 2020B1515310002.

\end{document}